\documentclass[a4paper,11pt]{amsart}
\usepackage{amssymb,amscd,amsxtra,xypic}
\usepackage[all]{xy}
\usepackage{array}
\usepackage{nicematrix}
\usepackage{textcomp}
\usepackage{bm}

\usepackage[pagebackref=true, linktocpage=true]{hyperref}
\hypersetup{%
bookmarksnumbered=true,%
colorlinks=true,%
linkcolor=blue,%
citecolor=blue,%
urlcolor=blue,%
setpagesize=false,%
pdftitle={},%
pdfauthor={Takuzo Okada}}

\theoremstyle{plain}
\newtheorem{Thm}{Theorem}[section]
\newtheorem{Lem}[Thm]{Lemma}
\newtheorem{Cor}[Thm]{Corollary}
\newtheorem{Prop}[Thm]{Proposition}

\theoremstyle{definition}
\newtheorem{Def}[Thm]{Definition}
\newtheorem{Def-Lem}[Thm]{Definition-Lemma}
\newtheorem{Cond}[Thm]{Condition}
\newtheorem{Rem}[Thm]{Remark}

\newtheorem*{Ack}{Acknowledgments}

\numberwithin{equation}{section}

\newcommand{\Proj}{\operatorname{Proj}}

\newcommand{\Sing}{\operatorname{Sing}}
\newcommand{\Spec}{\operatorname{Spec}}

\newcommand{\Cl}{\operatorname{Cl}}

\newcommand{\Bs}{\operatorname{Bs}}

\newcommand{\mult}{\operatorname{mult}}
\newcommand{\reg}{\operatorname{reg}}

\newcommand{\Pli}{\operatorname{Pli}}
\newcommand{\wt}{\operatorname{wt}}

\newcommand{\ord}{\operatorname{ord}}
\newcommand{\Supp}{\operatorname{Supp}}
\newcommand{\Qsm}{\operatorname{Qsm}}
\newcommand{\EL}{\operatorname{EL}}
\newcommand{\Cent}{\operatorname{Cent}}
\renewcommand{\wt}{\operatorname{wt}}
\newcommand{\sigmawt}{\text{$\sigma$-$\operatorname{wt}$}}

\newcommand{\tauonewt}{\text{$\tau_1$-$\operatorname{wt}$}}

\newcommand{\msi}{\mathsf{i}}
\newcommand{\coeff}{\operatorname{coeff}}

\newcommand{\mbA}{\mathbb{A}}

\newcommand{\mbC}{\mathbb{C}}
\newcommand{\mbD}{\mathbb{D}}
\newcommand{\mbE}{\mathbb{E}}

\newcommand{\mbP}{\mathbb{P}}
\newcommand{\mbQ}{\mathbb{Q}}

\newcommand{\mbT}{\mathbb{T}}
\newcommand{\mbU}{\mathbb{U}}

\newcommand{\mbZ}{\mathbb{Z}}

\newcommand{\mcC}{\mathcal{C}}

\newcommand{\mcE}{\mathcal{E}}
\newcommand{\mcF}{\mathcal{F}}

\newcommand{\mcH}{\mathcal{H}}
\newcommand{\mcI}{\mathcal{I}}

\newcommand{\mcL}{\mathcal{L}}
\newcommand{\mcM}{\mathcal{M}}

\newcommand{\mcO}{\mathcal{O}}

\newcommand{\mcW}{\mathcal{W}}

\newcommand{\mcY}{\mathcal{Y}}
\newcommand{\mcZ}{\mathcal{Z}}

\newcommand{\msp}{\mathsf{p}}
\newcommand{\msq}{\mathsf{q}}
\newcommand{\msr}{\mathsf{r}}

\newcommand{\msw}{\mathsf{w}}

\newcommand{\inj}{\hookrightarrow}

\newcommand{\ratmap}{\dashrightarrow}
\newcommand{\lrd}{\llcorner}
\newcommand{\rrd}{\lrcorner}

\makeatletter
\def\imod#1{\allowbreak\mkern10mu({\operator@font mod}\,\,#1)}
\makeatother

\title[Birationally solid Fano 3-fold hypersurfaces]{Birationally solid Fano 3-fold hypersurfaces}
\author[Takuzo Okada]{Takuzo Okada}
\address{Department of Mathematics, Faculty of Science and Engineering\endgraf
Saga University, Saga 840-8502 Japan}
\email{okada@cc.saga-u.ac.jp}
\date{}

\begin{document}

\begin{abstract}
A Fano variety of Picard number $1$ is said to be {\it birationally solid} if it is not birational to a Mori fiber space over a positive dimensional base.
In this paper we complete the classification of quasi-smooth birationally solid Fano $3$-fold weighted hypersurfaces.
\end{abstract}

\maketitle

\tableofcontents

\section{Introduction} \label{sec:intro}

Throughout the paper we work over the field $\mbC$ of complex numbers.
A Fano variety of Picard number one is \textit{birationally solid} (or simply \textit{solid}) if it is not birational to a Mori fiber space over a positive dimensional base.
Shokurov originally introduced this notion in \cite{Sho88} with a different name: birationally solid varieties were called \textit{primitive} in \cite{Sho88}.
Later on the name ``birational solidity" was proposed in \cite{AO}.
This notion was introduced as a generalization of birational rigidity, where we recall that a Fano variety of Picard number one is \textit{birationally rigid} if it is not birational to a Mori fiber space other than those that are isomorphic to itself.

By a \textit{quasi-smooth Fano $3$-fold weighted hypersurface}, we mean a hypersurface $X$ in a weighted projective $4$-space $\mbP (a_0, \dots, a_4)$ such that $X$ is quasi-smooth, that is, the affine cone of $X$ is smooth outside the origin, $X$ has only terminal singularities, and
\[
\iota_X := \sum_{i=0}^4 a_i - d > 0,
\]
where $d$ is the degree of the quasi-homogenous polynomial defining $X$.
The number $\iota_X$ is called the {\it index} (or \textit{Fano index}) of $X$.
Quasi-smooth Fano $3$-fold weighted hypersurfaces are classified by Reid, Iano-Fletcher, Brown, and Suzuki (see \cite{IF}, \cite{BS07a}, \cite{BS07b}), and they form $130$ families among which $95$ families consist of those with index $1$ (see \cite[Table 5]{IF} and \cite[Table 1]{Okstrat} for the list, and see Theorem~\ref{thm:wtbdd} for the completeness of the list).

After the pioneering work on smooth quartic $3$-folds by Iskovskikh--Manin \cite{IM}, Corti--Pukhlikov--Reid \cite{CPR} proved birational rigidity of quasi-smooth members in each of the $95$ families under suitable generality assumptions.
Cheltsov--Park \cite{CP} removed the generality assumptions and the classification of birationally rigid quasi-smooth Fano $3$-fold weighted hypersurfaces is finally completed as follows.

\begin{Thm}[{Cheltsov--Park \cite{CP}}] \label{thm:CP}
Any quasi-smooth Fano $3$-fold weighted hypersurface of index $1$ is birationally rigid.
\end{Thm}

\begin{Thm}[{Abban--Cheltsov--Park \cite{ACP}}] \label{thm:ACP}
Let $X$ be a quasi-smooth Fano $3$-fold weighted hypersurface of index greater than $1$.
Then $X$ is not birationally rigid.
If $X$ belongs to Family \textnumero~$\msi$, where $\msi \notin \{100, 101, 102, 103, 110\}$, then $X$ is not birationally solid.
\end{Thm}

It took many years---50 years after \cite{IM}---to complete the classification of birationally rigid quasi-smooth Fano $3$-fold weighted hypersurfaces.
However the classification of birationally solid ones has not been finished yet: birational solidity or non-solidity for members of Family \textnumero~$\msi$ with $\msi \in \{100, 101, 102, 103, 110\}$ remains undetermined.
The list of the remaining 5 families is as follows:
\begin{itemize}
\item Family \textnumero~$100$: $X_{18} \subset \mbP (1, 2, 3, 5, 9)$.
\item Family \textnumero~$101$: $X_{22} \subset \mbP (1, 2, 3, 7, 11)$.
\item Family \textnumero~$102$: $X_{26} \subset \mbP (1, 2, 5, 7, 13)$.
\item Family \textnumero~$103$: $X_{38} \subset \mbP (2, 3, 5, 11, 19)$.
\item Family \textnumero~$110$: $X_{21} \subset \mbP (1, 3, 5, 7, 8)$.
\end{itemize}
In the above list, the subscript of $X$ is the degree of the defining equations of weighted hypersurfaces in the corresponding weighted projective space (see Sections~\ref{sec:birig} and \ref{sec:tririg} for details).
The following is the main result of this paper, which completes the classification of birationally solid quasi-smooth Fano $3$-fold weighted hypersurfaces.

\begin{Thm} \label{thm:main}
Let $X$ be a quasi-smooth Fano $3$-fold weighted hypersurface of index greater than $1$.
Then $X$ is birationally solid if and only if it belongs to family \textnumero~$\msi$ for $\msi \in \{100, 101, 102, 103, 110\}$.
\end{Thm}

\begin{Rem}
\textit{Any} quasi-smooth Fano $3$-fold weighted hypersurface of index $1$ is nonrational by Theorem \ref{thm:CP}.
Rationality questions for \textit{very general} quasi-smooth Fano $3$-fold hypersurfaces of index $> 1$ are settled in \cite{Okstrat}. 
One of the main motivations of introducing birational solidity lies in the rationality questions.
Birational solidity implies nonrationality, and in principle one can prove birational solidity of a given Fano variety.
This enables us to drop the assumption of \textit{very generality} in the above result for the 5 families in Theorem \ref{thm:main}: \textit{any} quasi-smooth member of the $5$ families in Theorem \ref{thm:main} is nonrational.
\end{Rem}

We explain a main difficulty in the proof of Theorem~\ref{thm:main} and how to overcome it as well as the organization of the paper.
In Section \ref{sec:prelim}, we give basic definitions and explain the framework of proof of birational solidity.
In Section~\ref{sec:exclmethod}, we give various methods for excluding a given subvariety as a so called maximal center which is the center of a divisorial contraction initiating an elementary link.
Sections~\ref{sec:birig} and \ref{sec:tririg} are devoted to the proof of Theorem~\ref{thm:main}. 
Let $X$ be a member of one of the $5$ families in Theorem~\ref{thm:main}.
Then an elementary (Sarkisov) link $X \ratmap \hat{X}$ to a Fano $3$-fold $\hat{X}$ is constructed in \cite{ACP}.
For Family \textnumero $110$, $X$ admits another elementary link $X \ratmap \breve{X}$ to a Fano $3$-fold $\breve{X}$.
The constructions of these links are explained in Sections~\ref{sec:birig} and \ref{sec:tririg}.
The proof of birational solidity will be done by classifying elementary links from $X$ and from $\hat{X}$ (and also from $\breve{X}$).
To this end, we need to understand divisorial contractions to $X$ and to $\hat{X}$ (and also to $\breve{X}$) especially with center a singular point.
The main difficulty lies in the fact that the Fano $3$-fold $\hat{X}$ has a terminal singularity of type $cE$ except when $X$ belongs to Family \textnumero $102$, and that detailed analysis of such a singularity is quite difficult (see Remark~\ref{rem:divcontcE} for the classification of divisorial contractions to $cE$ points).
We will overcome this difficulty by classifying divisorial contractions to particular types of $cE$ points based on ideas from the currently unpublished preprint \cite{HayakawaE} by Hayakawa.
These arguments are the technical core of the paper and are summarized in Section~\ref{sec:divcont}.
We believe that the results in Section~\ref{sec:divcont} are of interest on their own and will be applied to further study of Fano $3$-folds admitting $cE$ points.

In the literature there is no known example of a Fano $3$-fold with a $cE$ point whose birational models are determined, and this paper  provides first such Fano $3$-folds.
 
\begin{Ack}
The author is supported by JSPS KAKENHI Grant Numbers JP18K03216 and JP22H01118.
The author would like to thank Takayuki Hayakawa for his generosity to share ideas in the preprint \cite{HayakawaE}.
The ideas are on various weighted blow-ups obtaining divisors of discrepancy $1$ over $cE$ points, which are crucial in Section~\ref{sec:divdisc1}.
The author also thanks Erik Paemurru for comments and questions that improved the exposition of the paper.
\end{Ack}

\section{Preliminaries} \label{sec:prelim}

\subsection{Notation and basic definitions}

\subsubsection{Pliability set and birational solidity}

\begin{Def}
Let $\pi \colon X \to S$ be a morphism with connected fibers between normal projective varieties.
We say that $\pi \colon X \to S$ (or simply $X/S$) is a \textit{Mori fiber space} if $X$ is $\mbQ$-factorial, has only terminal singularities, the relative Picard number is $1$, $-K_X$ is ample over $S$, and $\dim S < \dim X$.
\end{Def}

\begin{Def}
A \textit{Fano variety} is a normal projective $\mbQ$-factorial variety with only terminal singularities whose anti-canonical divisor is ample.
A Fano variety of Picard number $1$ is called a \textit{Mori-Fano variety}.
\end{Def}

A Mori-Fano variety $X$ together with its structure morphism $X \to \Spec (\mbC)$ is nothing but a Mori fiber space $X/\Spec (\mbC)$ over a point, and we usually omit the description $/\Spec (\mbC)$ in this case.
By a birational map $f \colon X/S \ratmap Y/T$ between Mori fiber spaces, we mean a birational map $f \colon X \ratmap Y$ between their total spaces.

\begin{Def}
A birational map $f \colon X/S \ratmap Y/T$ between Mori fiber spaces is \textit{square} if there is a birational map $g \colon S \ratmap T$ such that the diagram
\[
\xymatrix{
X \ar[d] \ar@{-->}[r]^{f} & Y \ar[d] \\
S \ar@{-->}[r]_{g} & T}
\]
commutes and the induced birational map between generic fibers of $X/S$ and $Y/T$ is a biregular isomorphism.

Mori fiber spaces $X/S$ and $Y/T$ are \textit{square birational}, denoted $X/S \sim_{\operatorname{sq}} Y/T$, if there is a square birational map $f \colon X/S \ratmap Y/T$.
\end{Def}

For a Mori-Fano variety $X$, a square birational map $X \ratmap Y/T$ (or $Y/T \ratmap X$) to (or from) a Mori fiber space $Y/T$ is simply a biregular isomorphism.

\begin{Def}
The \textit{pliability set} of a Mori fiber space $X/S$ is the set
\[
\Pli (X/S) := \{\, \text{Mori fiber space $Y/T$} \mid \text{$Y$ is birational to $X$} \,\}/\sim_{\operatorname{sq}}.
\]
\end{Def}

\begin{Def}
Let $X$ be a Mori-Fano variety.
We say that $X$ is \textit{birationally solid} if $\Pli (X)$ consists only of Mori-Fano varieties, or in other words, if $X$ is not birational to a Mori fiber space over a positive dimensional base.
We say that $X$ is \textit{birationally rigid} if $\Pli (X) = \{X\}$.
\end{Def}

\subsubsection{Weighted hypersurfaces}

Let $a_0, \dots, a_n$ be positive integers and let
\[
\mbP := \mbP (a_0, \dots, a_n) = \Proj \mbC [x_0, \dots, x_n]
\]
be the weighted projective space with homogeneous coordinates $x_0, \dots, x_n$ of weights $a_0, \dots, a_n$, respectively.
We say that $\mbP$ is \textit{well-formed} if 
\[
\gcd \{\, a_i \mid 0 \le i \le n, i \ne k \ \} = 1
\]
for $k = 0, 1, \dots, n$.
For a homogeneous coordinate $v \in \{x_0, \dots, x_n\}$, the point $(0\!:\!\cdots\!:\!0\!:\!1\!:\!0\!:\!\cdots\!:\!0) \in \mbP$ at which all the homogeneous coordinates except for $v$ vanish is denoted by $\msp_v$.
For quasi-homogeneous polynomials $F_1, \dots, F_m \in \mbC [x_0, \dots, x_n]$, we define 
\[
(F_1 = \cdots = F_m = 0) := \Proj \mbC [x_0, \dots, x_m]/(F_1, \dots, F_m)
\]
and, for a subscheme $V \subset \mbP$, we denote by $(F_1 = \cdots = F_m)_V$ the scheme-theoretic intersection $(F_1 = \cdots = F_m = 0) \cap V$.

\begin{Def}
Let $V$ be a closed subscheme of $\mbP$ defined by a homogeneous ideal $I \subset \mbC [x_0, \dots, x_n]$.
The \textit{quasi-smooth locus} $\Qsm (V)$ of $V$ is defined to be the image of the smooth locus of $C^*_V := C_V \setminus \{o\}$ under the natural morphism $\mbA^{n+1} \setminus \{o\} \to \mbP$, where
\[
C_V := \Spec \mbC [x_0, \dots, x_n]/I
\]
is the \textit{affine cone} of $V$ and $o \in \mbA^{n+1}$ is the origin.
For a subset $S \subset V$, we say that $V$ is quasi-smooth along $S$ if $S \subset \Qsm (V)$.
We say that $V$ is \textit{quasi-smooth} if it is quasi-smooth along $V$, that is, $V = \Qsm (V)$.  
\end{Def}

We sometimes write $\mbP (a_x, b_y, c_z, d_t, e_w)$, where $a, b, c, d, e$ are positive integers, and this means that it is the weighted projective space $\mbP (a, b, c, d, e)$ with homogeneous coordinates $x, y, z, t, w$ of weights $a, b, c, d, e$, respectively.

\subsubsection{Classification of Fano $3$-fold weighted hypersurfaces}

Let $a_0, \dots, a_4$ and $d$ be positive integers.
We assume that the following condition is satisfied.

\begin{Cond} \label{cond:fanowh}
\begin{itemize}
\item $a_0 \le \cdots \le a_4$, 
\item $a_4 < d$, 
\item $\gcd \{\, a_i \mid 0 \le i \le 4, i \ne j \,\} = 1$ for any $j = 0, \dots, 4$, 
\item $d < a_0 + \cdots + a_4$, and
\item there is a quasi-smooth hypersurface of degree $d$ in $\mbP (a_0, \dots, a_4)$ which has only terminal singularities.
\end{itemize}
\end{Cond}

Let $X = X_d \subset \mbP (a_0, \dots, a_4)$ be a quasi-smooth hypersurface of degree $d$ in $\mbP (a_0, \dots, a_4)$ with homogeneous coordinates $x_0, \dots, x_4$.
Then $X$ is a Fano $3$-fold with only cyclic quotient terminal singularities.
We set $\iota := \sum a_i - d$ which we call the index of $X$.
For the index 1 case (i.e.\ $\iota = 1$), the completeness of the list \cite[Table~5]{IF} is proved in \cite{CCC} while the completeness of the lists \cite[Table~1]{BS07b} and \cite[Table 1]{BS07a} (which are summarized in \cite[Table~1]{Okstrat}) for index $\iota > 1$ cases are not discussed in the literature at least in an explicit way.
We believe that the completeness has been known among specialists.
We include its proof for readers' convenience.   
We will give an effective upper bound of the weights $a_0, \dots, a_4$ by elementary arguments, which guarantees the classification of quasi-smooth Fano $3$-fold weighted hypersurfaces (of any index $\iota$).

\begin{Rem} \label{rem:quotsingind}
Let $\msp_1, \dots, \msp_n \in X$ be the singular points and we denote by $r_i$ the index of $\msp_i \in X$.
By \cite[Theorem~1.2]{KMMT} and \cite[Corollary~10.3]{Reid87}, we have 
\[
\sum_{i=1}^n \left(r _i - \frac{1}{r_i} \right) \le 24,
\]
which in particular shows $r_i \le 24$ for any $i$.
\end{Rem}

\begin{Lem} \label{lem:wtbdd1}
Suppose that $a_4 \mid d$ and $a_3 \le 24$.
Then $a_4 \le 168$.
\end{Lem}

\begin{proof}
We write $d = m a_4$ for some integer $m \ge 2$.
By the quasi-smoothness, we have $k a_3 + a_j = d$ for some $j = 0, \dots, 4$.

Suppose first that $a_j = a_4$.
Then $k a_3 = d - a_4 = (m - 1) a_4$.
We have $k \le 3$ since $k a_3 + a_4 = d < a_0 + \cdots +a_4$.
In this case $a_4 \le k a_3 \le 3 a_3 \le 72$.

Suppose that $j \le 3$.
By the inequality $d < a_0 + \cdots + a_3 + a_4$ and $d = m a_4$, we have $(m-1)d/m < a_0 + \cdots + a_3$.
Combining this with $k a_3 + a_0 \le k a_3 + a_j = d$, we obtain
\[
\frac{(m-1)k}{m} a_3 + \frac{m-1}{m} a_0 \le \frac{m-1}{m} d < a_0 + \cdots a_3.
\]
We have
\[
\frac{(m-1)k}{m} a_3 < \frac{1}{m} a_0 + a_1 + a_2 + a_3 \le \frac{3 m + 1}{m} a_3,
\]
which shows $(m-1)k \le 3 m$ and thus $k \le 3 + 3/(m-1) \le 6$ since $m \ge 2$.
In this case we have $d = k a_3 + a_j \le (k+1) a_3 \le 168$.
\end{proof}

\begin{Lem} \label{lem:wtbdd2}
If $a_3 \mid d$ and $a_4  \mid d$, then $a_4 \le 168$.
\end{Lem}

\begin{proof}
We write $d = l a_3 = m a_4$ for some positive integers $l \ge m \ge 2$.
By the inequality $d < a_0 + \cdots + a_3 + a_4 \le 4 a_3 + d/m$, we obtain 
\[
\frac{(m-1) l}{m} a_3 = \frac{(m-1)d}{m} < 4 a_3.
\]
Hence $l < 4m/(m-1) \le 8$ since $m \ge 2$.
We set $b := \gcd (a_3, a_4)$.
We claim that $b \le 24$.
Indeed, if $b > 1$, then $X$ had a cyclic quotient singular point of index $b$ along $(x_0 = x_1 = x_2 = 0)_X \ne \emptyset$, and hence $b \le 24$ by Remark~\ref{rem:quotsingind}.
By the equation $l a_3 = m a_4$, we have $a_4 \le l b$.
Therefore we have $a_4 \le 7 \cdot 24 = 168$ as desired.
\end{proof}

%
%

\begin{Thm} \label{thm:wtbdd}
Let $X = X_d \subset \mbP (a_0, \dots, a_4)$ be a quasi-smooth Fano $3$-fold weighted hypersurface in a well-formed weighted projective space which is not a linear cone $($i.e.\ $d \ne a_i$ for any $i$$)$.
Then $a_i \le 168$ for any $i$.
In particular, the lists \cite[Table~1]{BS07b} and \cite[Table 1]{BS07a} are complete.
\end{Thm}

\begin{proof}
Let $a_0, \dots, a_4$ and $d$ be as in the statement and we assume that $a_0 \le \cdots \le a_4$.
Then they satisfy Condition~\ref{cond:fanowh}.
We may assume $a_4 > 24$.
We see that $\msp_{x_4} \notin X$ because otherwise $\msp_{x_4} \in X$ is a cyclic quotient singular point of index $a_4 > 24$ which is absurd by Remark~\ref{rem:quotsingind}.
It follows that $a_4 \mid d$, and thus we obtain $a_4 \le 168$ by Lemmas~\ref{lem:wtbdd1} and \ref{lem:wtbdd2}.

By this explicit bound, we can perform a full computer search for tuples $(d, a_0, \dots, a_4)$ satisfying Condition~\ref{cond:fanowh}.
The search gives $130$ tuples among which $35 = 130 - 95$ correspond to index $\ge 2$ cases.
The $35$ tuples coincide with the lists \cite[Table~1]{BS07b} and \cite[Table 1]{BS07a}.
This completes the proof.
\end{proof}

\subsubsection{Rank $2$ toric varieties}

\begin{Def} \label{def:toricT}
Let $m < n$ be positive integers and $a_1, \dots, a_n, b_1 \dots b_n$ be integers.
Let $R = \mbC [x_1, \dots, x_n]$ be the polynomial ring endowed with a $\mbZ^2$-grading defined by the $2 \times n$ matrix
\[
\begin{pmatrix}
a_1 & \cdots & a_n \\
b_1 & \cdots & b_n 
\end{pmatrix}.
\]
This means that the bi-degree of the variable $x_i$ is $(a_i, b_i) \in \mbZ^2$.
We denote by 
\begin{equation} \label{eq:deftoricT}
\mbT \begin{pNiceArray}{ccc|ccc}[first-row]
x_1 & \dots & x_m & x_{m+1} & \dots & x_n \\
a_1 & \dots & a_m & a_{m+1} & \dots & a_n \\
b_1 & \dots & b_m & b_{m+1} & \dots & b_n
\end{pNiceArray}
\end{equation}
the toric variety whose Cox ring is $R$ and irrelevant ideal is $I = (x_1, \dots, x_m) \cap (x_{m+1}, \dots, x_n)$.
In other words, this toric variety is the geometric quotient
\[
(\mbA^n_{x_1, \dots, x_n} \setminus V (I))/(\mbC^*)^2,
\]
where the $(\mbC^*)^2$-action is defined by
\[
(\lambda, \mu) \cdot (x_1, \dots, x_n) = (\lambda^{a_1} \mu^{b_1} x_1, \dots, \lambda^{a_n} \mu^{b_n} x_n).
\]
\end{Def}

Let the notation as in Definition \ref{def:toricT} and let $\mbT$ be the toric variety \eqref{eq:deftoricT}.
Then $\mbT$ is a projective and simplicial toric variety of Picard rank $2$.
Let $\msp \in \mbT$ be a point and let $\msq = (\alpha_1, \dots, \alpha_n) \in \mbA^n$ be a preimage of $\msp$ by the morphism $\mbA^n \setminus V (I) \to \mbT$.
In this case we express $\msp$ as 
\[
\msp = (\alpha_1\!:\!\cdots\!:\!\alpha_m \, | \, \alpha_{m+1}\!:\!\cdots\!:\!\alpha_n) \in \mbT.
\]

\begin{Rem}
Let $\mbP = \mbP (a_0, \dots, a_n)$ be the weighted projective space with homogeneous coordinates $x_0, \dots, x_n$ of weights $a_0, \dots, a_n$, respectively.
Let $b_1, \dots, b_n$ be positive integers.
Then the morphism
\[
\Psi \colon \mbT := \mbT \begin{pNiceArray}{cc|cccc}[first-row]
u & x_0 & x_1 & x_2 & \dots & x_n \\
0 & a_0 & a_1 & a_2 & \dots & a_n \\
-a_0 & 0 & b_1 & b_2 & \dots & b_n
\end{pNiceArray} \to
\mbP,
\]
defined by
\[
(u\!:\!x_0 \, | \, x_1\!:\!\cdots\!:\!x_n) \mapsto (x_0\!:\!u^{b_1/a_0} x_1\!:\!u^{b_2/a_0} x_2\!:\!\cdots\!:\!u^{b_n/a_0} x_n)
\]
is the weighted blow-up of $\mbP$ at the point $\msp_{x_0}$ with $\wt (x_1, \dots, x_n) = \frac{1}{a_0} (b_1, \dots, b_n)$.
The $\Psi$-exceptional divisor is the divisor on $\mbT$ which is the zero locus of the section $u$.
\end{Rem}

\subsection{Framework of proof of birational solidity}
 
Let $X$ be a Mori-Fano variety.
By the Sarkisov program (see \cite{Corti95}, \cite{HM}), any birational map $X \ratmap Y/T$ to a Mori fiber space can be decomposed into a chain of elementary links.
Elementary links are suitable birational maps between Mori fiber spaces which fall into $4$ types (types I, II, II, IV, see \cite[Definition 3.4]{Corti95}). 
We recall $2$ types (I and II) of such links that are links from a Mori-Fano variety.

\begin{Def}
Let $X$ be a Mori-Fano variety and $V/S$ a Mori fiber space.
An \textit{elementary link} $\sigma \colon X \ratmap V/S$ is a birational map which sits in the commutative diagram
\begin{equation} \label{eq:link}
\xymatrix{
Y \ar[d]_{\varphi} \ar@{-->}[r]^{\tau} & W \ar[d]^{\psi} \\
X \ar@{-->}[r]_{\sigma} & V}
\end{equation}
where $\varphi \colon Y \to X$ is a divisorial contraction, $\tau$ is a composite of inverse flips, a flop, and flips, and $\psi$ is either an isomorphism or a divisorial contraction.
The center of $\varphi$ is called the \textit{center} of the elementary link $\sigma$, and it is denoted by $\operatorname{Cent} (\sigma)$.

We define $\EL (X)$ to be the set of elementary links from $X$.
For a subvariety $\Gamma \subset X$, we define
\[
\EL_{\Gamma} (X) := \{\, \sigma \in \EL (X) \mid \operatorname{Cent} (\sigma) = \Gamma \, \}.
\]
\end{Def}

In this paper, a birational map that is a composite of inverse flips, a flop, and flips will be simply called a \textit{pseudo isomorphism}.
Note that throughout the paper, and hence in the above definition as well, a divisorial contraction is an extremal divisorial contraction in the terminal category.
Given a divisorial contraction $\varphi \colon Y \to X$, an elementary link $\sigma \colon X \ratmap V/S$ that fits into the commutative diagram \eqref{eq:link} for some $\tau$ and $\psi$ is unique if it exists.
We call $\sigma$ the elementary link \textit{initiated by} $\varphi$.

In order to determine the set $\Pli (X)$, it is crucial to classify elementary links from $X$ and then from the Mori fiber spaces that appear as the targets of such links, and so on.
In view of the fact that any elementary link from $X$ is initiated by a divisorial contraction, it is important to detect divisorial contractions that can initiate links.  

\begin{Def}
For a movable linear system $\mcM$ on $X$, a \textit{maximal extraction} of $X$ with respect to $\mcM$ is a divisorial contraction $\varphi \colon Y \to X$ such that
\[
m_E (\mcM) > n a_E (K_X),
\]
where $E$ is the $\varphi$-exceptional divisor, $n$ is the positive rational number defined by $\mcM \sim_{\mbQ} - n K_X$, $m_E (\mcH)$ is the multiplicity of $\varphi^* \mcM$ along $E$, and $a_E (K_X)$ is the discrepancy of $K_X$ along $E$.
A divisorial contraction $\varphi$ is called a \textit{maximal extraction} of $X$ if there is a movable linear system $\mcM$ such that $\varphi$ is the maximal extraction of $X$ with respect to $\mcM$.
A subvariety $\Gamma \subset X$ is called a \textit{maximal center} if there is a maximal extraction $\varphi$ whose center is $\Gamma$.
\end{Def}

\begin{Lem}[{\cite[Lemma 2.5]{Oktri}}]
A divisorial contraction $\varphi \colon Y \to X$ that initiates an elementary link is a maximal extraction.
In particular, the center of an elementary link is a maximal center.
\end{Lem}

For Mori-Fano $3$-folds in Families \textnumero 100, 101, 102, and 103, we will be able to classify maximal extractions---and hence elementary links---completely, and Theorem \ref{thm:main} for these $4$ families will be a consequence of such classifications. 
However, for Mori-Fano $3$-folds in Family \textnumero 110, we are unable to do such a complete classification.
We will overcome this situation by classifying a Sarkisov generating set which is introduced below.

\begin{Def}
Let $f \colon V/S \ratmap W/T$ be a birational map between Mori fiber spaces.
We say that a set $\{\sigma_{\lambda} \colon V_{\lambda}/S_{\lambda} \ratmap W_{\lambda}/T_{\lambda}\}_{\lambda \in \Lambda}$ of elementary links \textit{generates} $f$ if there are finitely many $\lambda_1, \dots, \lambda_m \in \Lambda$ and an automorphism $g$ of $Y$ such that $V/S = V_{\lambda_1}/S_{\lambda_1}$, $W/T = W_{\lambda_m}/T_{\lambda_m}$, $W_{\lambda_i}/T_{\lambda_i} = V_{\lambda_{i+1}}/S_{\lambda_{i+1}}$ for $1 \le i < m$, and $f = g \circ \sigma_{\lambda_m} \circ \cdots \circ \sigma_{\lambda_1}$: 
\[
\xymatrix{
V \ar[d] \ar@<1.0ex>@{-->}^{f}@/^2pc/[rrrrrrr] \ar@{=}[r] & V_{\lambda_1} \ar[d] \ar@{-->}[r]_{\sigma_{\lambda_1}} & W_{\lambda_1} \ar[d] \ar@{=}[r] & V_{\lambda_2} \ar[d] \ar@{-->}[r]_{\sigma_{\lambda_2}} & \cdots \ar@{-->}[r]_{\sigma_{\lambda_m}} & W_{\lambda_m} \ar[d] \ar@{=}[r] & Y \ar[d] \ar[r]_{g} & Y \ar[d] \\
S \ar@{=}[r] & S_{\lambda_1} & T_{\lambda_1} \ar@{=}[r] & S_{\lambda_2} & & T_{\lambda_m} \ar@{=}[r] & T & T}
\]

Let $\mcF$ be a set of Mori fiber spaces such that any pair of Mori fiber spaces in $\mcF$ are birationally equivalent, and let $\Sigma$ be a set of elementary links between Mori fiber spaces in $\mcF$.
We say that $\Sigma$ is a \textit{Sarkisov generating set} for $\mcF$ if any birational map $f \colon V/S \ratmap W/T$ between Mori fiber spaces such that $V/S \in \mcF$ is generated by $\Sigma$.
\end{Def}

Let $X$ be a quasi-smooth member of one of the $5$ families in Theorem \ref{thm:main}.
In \cite{ACP}, an elementary link $\hat{\sigma} \colon X \ratmap \hat{X}$ to a Mori-Fano $3$-fold $\hat{X} \not\cong X$ is constructed.
For $X$ belonging to Family \textnumero 110, another elementary link $\breve{\sigma} \colon X \ratmap \breve{X}$  to a Mori-Fano $3$-fold $\breve{X}$ is constructed, where $X, \hat{X}, \breve{X}$ are mutually non-isomorphic (see Sections~\ref{sec:birig} and \ref{sec:tririg} for more details).
The aim of this paper is to prove the following from which Theorem \ref{thm:main} follows.

\begin{Thm} \label{thm:main2}
Let $X$ be a quasi-smooth member of Family \textnumero~$\msi$.
\begin{enumerate}
\item Suppose that $\msi \in \{100, 101, 102, 103\}$.
Then we have $\EL (X) = \{\hat{\sigma} \}$ and $\EL (\hat{X}) = \{\hat{\sigma}^{-1}\}$, and in particular $\Pli (X) = \{X, \hat{X}\}$.
\item Suppose that $\msi = 110$. 
Then there are an elementary link $\rho \colon \hat{X} \ratmap \breve{X}$ and a birational involution $\hat{\iota} \colon \hat{X} \ratmap \hat{X}$ which is an elementary link if it is not biregular.
The set $\{\hat{\sigma}^{\pm}, \breve{\sigma}^{\pm}, \rho^{\pm}, \hat{\iota}\}$ is a Sarkisov generating set for $\{X, \hat{X}, \breve{X}\}$.
In particular we have $\Pli (X) = \{X, \hat{X}, \breve{X}\}$.
\end{enumerate}
\end{Thm}

\begin{Rem}
In (2) of Theorem \ref{thm:main2}, it can happen that $\hat{\iota}$ is a biregular involution.
This happens only for special members, and in this case the set $\{\hat{\sigma}^{\pm}, \breve{\sigma}^{\pm}, \rho^{\pm}\}$ is a Sarkisov generating set for $\{X, \hat{X}, \breve{X}\}$.
\end{Rem}

\subsection{Methods of exclusion}

Throughout this subsection, we assume that $X$ is a Mori-Fano 3-fold.
We recall various methods of exclusion of maximal centers and also give a generalization for one of them.

\begin{Lem}[{\cite[Lemma 2.9]{OkII}}] \label{lem:mtdexclC}
Let $\Gamma \subset X$ be an irreducible and reduced curve.
If $(-K_X \cdot \Gamma) \ge (-K_X^3)$, then $\Gamma$ is not a maximal center.
\end{Lem}

\begin{Def}
Let $L$ be a Weil divisor class on $X$ and $\msp$ a smooth point of $X$. 
We say that $L$ \textit{isolates} $\msp$ if there exists a positive integer $m$ such that $\msp$ is an isolated component of the base locus of the linear system $|\mcI^m_{\msp} (m L)|$.
\end{Def}

\begin{Lem}[{cf.\ \cite[Lemma 2.15]{OkII}}] \label{lem:mtdexclsmpt}
Let $\msp \in X$ be a smooth point.
If there is a Weil divisor class $L \sim_{\mbQ} - l K_X$ on $X$ such that it isolates $\msp$ and 
\[
l \le \frac{4}{(-K_X^3)},
\] 
then $\msp$ is not a maximal center.
\end{Lem}

\begin{Lem}[{\cite[Lemma 2.18]{OkII}}] \label{lem:exclsingptG}
Let $\varphi \colon Y \to X$ be a divisorial contraction centered at a point $\msp \in X$ with exceptional divisor $E$.
Suppose that there are effective divisors $S$ and $T$ on $Y$ with the following properties.
\begin{enumerate}
\item $S \sim_{\mbQ} - a K_Y + d E$ and $T \sim_{\mbQ} - b K_Y + eE$ for some rational numbers $a, b, c, d$ such that $a, b > 0$, $0 \le e < a_E (K_X) b$, and $a e - b d \ge 0$.
\item The intersection $\Gamma = S \cap T$ is a $1$-cycle whose support consists of irreducible and reduced curves which are numerically proportional to each other.
\item $(T \cdot \Gamma) \le 0$.
\end{enumerate}
Then $\varphi$ is not a maximal extraction.
\end{Lem}

\begin{Lem}[{\cite[Lemma 2.19]{OkII}}] \label{lem:exclnegdef}
Let $\varphi \colon Y \to X$ be a divisorial contraction with exceptional divisor $E$.
Suppose that there is an effective divisor $S \sim_{\mbQ} - b K_Y + e E$ with $b > 0$ and $e \ge 0$ on $Y$ and and a normal surface $T \ne E$ on $Y$ such that the support of the $1$-cycle $S|_T$ consists of curves whose intersection matrix is negative-definite.
Then $\varphi$ is not a maximal extraction.
\end{Lem}

The following is a generalization of \cite[Lemma 2.6]{Okcodim4} which corresponds to the case $\Delta = 0$.

\begin{Lem} \label{lem:exclquotsing}
Let $\varphi \colon Y \to X$ be a divisorial contraction centered at a point $\msp \in X$ with exceptional divisor $E$.
Suppose that there exist an effective divisor $S$ on $X$ passing through $\msp$, a linear system $\mcL$ on $X$, and an effective $1$-cycle $\Delta$ on $X$ with the following properties.
\begin{enumerate}
\item $\Bs \mcL \cap \Supp S = \{\msp\} \cup \Supp \Delta$.
\item For a general member $T \in \mcL$, the $1$-cycle $S \cdot T - \Delta$ is effective and does not contain any component of $\Delta$. \item For a general $T \in \mcL$, we have
\[
(-K_Y \cdot \tilde{S} \cdot \tilde{T}) - (-K_Y \cdot \tilde{\Delta}) \le 0,
\] 
where $\tilde{S}$, $\tilde{T}$, and $\tilde{\Delta}$ are the proper transforms of $S$, $T$, and $\Delta$ on $Y$, respectively.
\end{enumerate}
Then $\msp$ is not a maximal center.
\end{Lem}

\begin{proof}
For a general $T \in \mcL$, we set $S \cdot T = \Gamma_T + \Delta$, where $\Gamma_T$ is an effective $1$-cycle which does not contain any component of $\Delta$ in its support.
Then we have $\tilde{S} \cdot \tilde{T} = \tilde{\Gamma}_T + \tilde{\Delta} + \Xi$, where $\Xi$ is an effective $1$-cycle supported on $E$.
We have $(-K_Y \cdot \Xi) \ge 0$ and hence
\[
(-K_Y \cdot \tilde{\Gamma}_T) = (-K_Y \cdot \tilde{S} \cdot \tilde{T}) - (-K_Y \cdot \tilde{\Delta}) - (-K_Y \cdot \Xi) \le 0.
\]
By \cite[Lemma 8.5]{OkIII}, there is a component $\tilde{\Gamma}^{\circ}_T$ of $\tilde{\Gamma}_T$ such that $(-K_Y \cdot \tilde{\Gamma}^{\circ}_T) \le 0$ and $(\tilde{\Gamma}^{\circ}_T \cdot E) > 0$.
By \cite[Lemma 2.10]{OkII}, $\varphi$ is not a maximal extraction.
\end{proof}

\begin{Lem}[{\cite[Lemma 2.11]{OkII}}] \label{lem:exclCint}
Let $\Gamma \subset X$ be an irreducible and reduced curve.
Assume that there is an effective divisor $S$ on $X$ containing $\Gamma$ and a movable linear system $\mcM$ on $X$ whose base locus contains $\Gamma$ with the following properties.
\begin{enumerate}
\item $S \sim_{\mbQ} - m K_X$ for some rational number $m \ge 1$.
\item For a general member $T \in \mcM$, $T$ is a normal surface, the intersection $S \cap T$ is contained in the base locus of $\mcM$ set-theoretically, and $S \cap T$ is reduced along $\Gamma$.
\item Let $T \in \mcM$ be a general member and let $\Gamma, \Gamma_1, \dots, \Gamma_l$ be the irreducible and reduced curves contained in the base locus of $\mcM$.
For each $i = 1, \dots, l$, there is an effective $1$-cycle $\Delta_i$ on $T$ such that $(\Gamma \cdot \Delta_i) \ge (-K_X \cdot \Delta_i) > 0$ and $(\Gamma_j \cdot \Delta_i) \ge 0$ for $j \ne i$.
\end{enumerate}
\end{Lem}

\subsection{Finding a Sarkisov generating set}

We explain how to obtain a Sarkisov generating set for a collection of Mori-Fano $3$-folds.

\begin{Def}
Let $X$ be a Mori-Fano $3$-fold and let $\varphi \colon Y \to X$ be a maximal extraction with respect to a movable linear system $\mcM \sim_{\mbQ} - n K_X$.
We say that a birational map $f \colon X \ratmap X'$ to a Mori-Fano $3$-fold \textit{untwists} $\mcM$ if the rational number $n'$ defined by $\mcM' \sim_{\mbQ} - n' K_{X'}$, where $\mcM'$ is the birational transform of $\mcM$ on $X'$, satisfies $n' < n$.
\end{Def}

\begin{Lem}[{\cite[Corollary 2.24]{OkII}}] \label{lem:untwist}
Let $X$ be a Mori-Fano $3$-fold and let $\varphi \colon Y \to X$ be an extraction with respect to a movable linear system $\mcM$.
Then the elementary link initiated by $\varphi$, if it exists, untwists $\mcM$.
\end{Lem}

\begin{Prop} \label{prop:Sgen}
Let $\mcF$ be a finite set of Mori-Fano $3$-folds including $X$.
For each $V \in \mcF$, let $\mcC_V$ be a finite set of irreducible and reduced curves on $V$, and define the subset $\Delta_V := \cup_{C \in \mcC_V} C$ of $V$.
We set
\[
\Sigma := \bigcup_{V \in \mcF} \left( \operatorname{EL} (V) \setminus \bigcup_{\msp \in V_{\reg} \cap \Delta_V} \operatorname{EL}_{\msp} (V) \right),
\]
where $V_{\reg}$ is the smooth locus of $V$.
Suppose that the following are satisfied.
\begin{enumerate}
\item Any pair of Mori-Fano $3$-folds in $\mcF$ are birationally equivalent but not isomorphic.
\item For each $V \in \mcF$ and $C \in \mcC_V$, there are a terminal quotient singular point $\msq \in C$ of $V$ and an elementary link $V \ratmap V'$ centered at $\msq$ such that $V' \in \mcF$.
\item Any elementary link in $\Sigma$ is a link to a Mori-Fano $3$-fold in $\mcF$.
\item For any $V \in \mcF$, any curve $C \in \mcC_V$, and any movable linear system $\mcM \sim_{\mbQ} - n K_V$ on $V$ such that $(V, \frac{1}{n} \mcM)$ is canonical at the generic point of $C$, no smooth point $\msp$ of $V$ contained in $C$ is a maximal center with respect to $\mcM$.
\end{enumerate}
Then $\Sigma$ is a Sarkisov generating set for $\mcF$.
In particular $\Pli (X) = \mcF$ and $X$ is birationally solid.
\end{Prop}

\begin{proof}
Let $f \colon V \ratmap Z/T$ be a birational map from a Mori-Fano $3$-fold $V \in \mcF$ to a Mori fiber space $Z/T$.
We fix a very ample complete linear system $\mcH$ on $Z$. 
We will construct a chain of elementary links
\begin{equation} \label{eq:Sgen}
V = V_1 \overset{\sigma_1}{\ratmap} V_2 \overset{\sigma_2}{\ratmap}  V_3 \overset{\sigma_3}{\ratmap} \cdots
\end{equation} 
such that $\sigma_i \in \Sigma$ and it untwists the movable linear system $\mcM_i := {f_i}_*^{-1} \mcH$ for each $i$, where $f_i := f \circ \sigma_1^{-1} \circ \cdots \circ \sigma_i^{-1} \colon V_i \ratmap Z$.

Suppose that we have constructed such links $\sigma_1, \dots, \sigma_{j-1}$ for some $j \ge 1$, and that the birational map $f_j \colon V_j \ratmap Z$ is not biregular.
Then there is a maximal extraction $\varphi_j \colon W_j \to V_j$ with respect to $\mcM_j = {f_j}_*^{-1} \mcH$ and there is an elementary link $\tau_j \colon V_j \ratmap V'/S'$ to a Mori fiber space initiated by $\varphi_j$.
If $\Cent (\tau_j) \not\subset {V_j}_{\reg} \cap \Delta_{V_j}$, then we set $V_{j+1} := V'$ and $\sigma_j := \tau_j$.
Note that $V_{j+1} \in \mcF$ by (3) since $\sigma_j \in \Sigma$.
The link $\sigma_j$ untwists $\mcM_j$ by Lemma \ref{lem:untwist}.
Suppose that $\Cent (\tau_j) \subset {V_j}_{\reg} \cap \Delta_{V_j}$.
Note that $\Cent (\tau_j)$ is a point since there is no divisorial contraction centered along a curve passing through a terminal quotient singular point.
Let $C \in \mcC_{V_j}$ be a curve passing through $\msp$, and let $\msq \in V_j$ be a terminal quotient singular point on $C$ as in (2).
By (4), the pair $(X, \frac{1}{n_j} \mcM_j)$, where $\mcM_j \sim_{\mbQ} - n_j K_{V_j}$, is not canonical at the generic point of $C$.
By \cite[Theorem~2.2.1]{CP}, the point $\msq$ is a maximal center with respect to $\mcM_j$.
Then, by Lemma \ref{lem:untwist}, the link $\rho \colon V \ratmap V'$ given in (2) untwists $\mcM_j$.
In this case we set $V_{j+1} := V'$ and $\sigma_j := \rho \in \Sigma$.

Associated to the sequence \eqref{eq:Sgen}, there is a strictly decreasing sequence
\[
n_1 > n_2 > n_3 > \cdots
\]
of positive rational numbers, where $\mcM_i \sim_{\mbQ} - n_i K_{V_i}$.
The above rational numbers have bounded denominators.
It follows that the construction of the sequence  of elementary links in \eqref{eq:Sgen} terminates in finitely many steps.
This implies that there is a positive integer $m$ such that $f_m \colon V_m \ratmap Z$ is biregular, that is, $f_m$ is a biregular isomorphism and $Z \cong V_m \in F$.
Therefore $\Sigma$ is a Sarkisov generating set for $\mcF$, $\Pli (X) = \mcF$, and $X$ is birationally solid.
\end{proof}

\section{Divisorial contractions to $cE$ and $cD/2$ points}
\label{sec:divcont}

The complex $n$-space with coordinates $x_1, \dots, x_n$ is denoted by $\mbC^n_{x_1, \dots, x_n}$.
Let $\mbC \{x_1, \dots, x_n\}$ be the convergent power series ring in $n$ variables $x_1, \dots, x_n$.
The germ of a (\textit{hypersurface}) \textit{singularity} is the analytic germ of the hypersurface
\[
o \in (\phi = 0) \subset \mbC^n_{x_1, \dots, x_n}, \quad (\text{or simply $(\phi = 0) \subset \mbC^n_{x_1, \dots, x_n}$}),
\]
for a power series $\phi \in \mbC \{x_1, \dots, x_n\}$, where $o$ is the origin.
We say that two singularities $(\phi = 0) \subset \mbC^n_{x_1, \dots, x_n}$ and $(\phi' = 0) \subset \mbC^4_{x_1, \dots, x_n}$ are \textit{equivalent} if there are an automorphism $\sigma$ of $\mbC \{x, y, z, u\}$ given by $x_i \mapsto \sigma_i (x_1, \dots, x_n) \in \mbC \{x_1, \dots, x_n\}$ for $i = 1, 2, \dots, n$ and an invertible element $u = u (x_1, \dots, x_n) \in \mbC \{x_1, \dots, x_n\}$ such that
\[
\phi (\sigma_1, \dots, \sigma_n) = u (x_1, \dots, x_n) \phi' (x_1, \dots, x_n).
\]

We consider $3$-dimensional singularities.

\begin{Def}
A $3$-dimensional singularity $(\phi = 0) \subset \mbC^4_{x, y, z, u}$ is called a \textit{$cDV$ singularity} if its general hyperplane section is a Du Val singularity.
We say that a $cDV$ singularity is \textit{of type} $cA_n$, $cD_n$, and $cE_n$ if its general hyperplane section is a Du Val singularity of type $A_n$, $D_n$, and $E_n$, respectively.
\end{Def}

It is well known that a (non-smooth) terminal Gorenstein singularity is an isolated $cDV$ singularity (see \cite{Reid83}).
In this paper, we only consider isolated $cDV$ singularities.
By a slight abuse of notation, we assume that a $cDV$ singularity (hence a $cE_n$ singularity, and so on) is always isolated.

\begin{Rem}
A $cE$ singularity is equivalent to the germ
\[
(x^2 + y^3 + y g (z, u) + h (z, u) = 0) \subset \mbC^4_{x, y, z, u},
\]
for some power series $g (z, u), h (z, u) \in \mbC \{z, u\}$ such that $g (z, u) \in (z, u)^3$ and $h (z, u) \in (z, u)^4$.
It is of type $cE_6$ (resp.\ $cE_7$, resp.\ $cE_8$) if $h_4 \ne 0$ (resp.\ $h_4 = 0$ and $g_4 \ne 0$, resp.\ $h_4 = g_4 = 0$ and $h_5 \ne 0$), where $g_i$ and $h_i$ are the degree $i$ parts of $g$ and $h$, respectively (see eg.\ \cite[Theorem 2.10]{Kollar98}).
\end{Rem}

\begin{Rem} \label{rem:divcontcE}
The classification of divisorial contractions to $cE$ points of discrepancy greater than $1$ is finally completed by Yamamoto \cite{Yamamoto} after the work of Kawakita \cite{Kawakita05}. 
The classification of divisorial contractions to $cE$ points of discrepancy  $1$ is also completed by Hayakawa \cite{HayakawaE} which is currently an unpublished preprint and whose contents are known only to a very limited number of specialists.
\end{Rem}

We give definitions of $cE/2$ and $cD/2$ singularities.
Let $r > 0$ and $a_1, \dots, a_n$ be integers.
We consider the action of $\mbZ_r = \mbZ/r \mbZ$ on $\mbC^n_{x_1, \dots, x_n}$ by
\[
(x_1, \dots, x_n) \mapsto (\zeta^{a_1} x_1, \dots, \zeta^{a_n} x_n),
\]
where $\zeta \in \mbC$ is a primitive $r$th root of unity.
The quotient is denoted by 
\[
\mbC^n_{x_1, \dots, x_n}/\mbZ_r (a_1, \dots, a_n).
\]
For a power series $\phi \in \mbC \{x_1, \dots, x_n\}$ which is $\mbZ_r$-(semi) invariant, the quotient space of the hypersurface $(\phi = 0) \subset \mbC^n_{x_1, \dots, x_n}$ is denoted by
\[
(\phi (x_1, \dots, x_n) = 0)/\mbZ_r (a_1, \dots, a_n).
\]
For integers $r, a, b, c, d$ with $r > 0$, we sometimes write
\[
(\phi (x, y, z, t) = 0)/\mbZ_r (a_x, b_y, c_z, d_t)
\]
in order to make the action clear. 

\begin{Rem}
We recall the Kawamata blow-up at a terminal quotient singular point.
For positive integers $a_1$, $a_2$, $a_3$, and $r$, a $3$-dimensional singularity $\msp \in X$ which is equivalent to $o \in \mbC^3/\mbZ_r (a_1, a_2, a_3)$ is called a (cyclic quotient) singularity of type $\frac{1}{r} (a_1, a_2, a_3)$.
A cyclic quotient singularity is terminal if and only if it is of type $\frac{1}{r} (1, a, r - a)$ for some integers $r$ and $a$ such that $a$ is coprime to $r$ and $0 < a < r$.
Let $\msp \in X$ be a $\frac{1}{r} (1, a, r-a)$ point.
Then we have an equivalence $\msp \in X \cong o \in \mbC^3_{x, y, z}/\mbZ_r (1, a, r-a)$.
Then the weighted blow-up $\varphi \colon Y \to X$ at $\msp \in X$ with $\wt (x, y, z) = \frac{1}{r} (1, a, r-a)$ is the unique divisorial contraction centered at $\msp$ by \cite{Kawamata}. 
We call $\varphi$ the \textit{Kawamata blow-up} of $\msp \in X$.
Note that we have $K_Y = \varphi^* K_X + \frac{1}{r} E$, where $E \cong \mbP (1, a, r-a)$ is the exceptional divisor of $\varphi$, and $(E^3) = r^2/a (r-a)$.
\end{Rem}

\begin{Def}
The germ $X$ of a $3$-dimensional non-Gorenstein singularity is a \textit{$cE/2$ singularity} if there is an embedding $X \inj \mbC^4_{x, y, z, u} /\mbZ_2 (1, 0, 1, 1)$ such that
\[
X \cong (u^2 + y^3 + y g (x, z) + h (x, z) = 0)/\mbZ_2 (1, 0, 1, 1),
\]
for some $g (x, z) \in (x, z)^4 \mbC \{x, z\}$ and $h (x, z) \in (x, z)^4 \mbC \{x, z\} \setminus (x, z)^5 \mbC \{x, z\}$.
\end{Def}

A terminal singularity of type $cD/2$ is the quotient of a $cD$ singularity by a suitable $\mbZ_2$-action, and it falls into 2 types ($cD/2$-$1$ and $cD/2$-$2$, see \cite[\S 2]{Hayakawa00} for details).
We mainly consider type $cD/2$-$2$ singularity. 

\begin{Def}
The germ of a $3$-dimensional non-Gorenstein singularity is a $cD/2$-$2$ \textit{singularity} if there is an embedding $X \inj \mbC^4_{x, y, z, u}/\mbZ_2 (1, 1, 1, 0)$ such that
\[
X \cong (u^2 + y^2 z + \lambda y x^{2 a + 1} + g (x^2, z) = 0) \subset \mbC_{x, y, z, u}/\mbZ_2 (1, 1, 0, 1),
\]
where $\lambda \in \mbC$, $a \ge 1$, and $g (x^2, z) \in (x^4, x^2 z^2, z^3) \mbC \{x, z\}$. 
\end{Def}

The classification of divisorial contractions to $cE/2$ points is completed in \cite{Hayakawa99}, \cite{Hayakawa05a}, and \cite{Kawakita05}.

We give some definitions which will be necessary in this section.

\begin{Def}
Let $f = \sum \alpha_I x^I \in \mbC \{x_1, \dots, x_n\}$ be a nonzero power series, where $\alpha_I \in \mbC$ and $x^I = x_1^{i_1} \cdots x_n^{i_n}$ for $I = (i_1, \dots, i_n) \in \mbZ_{\ge 0}^n$. 

We define $\coeff_f (x^I) := \alpha_I$ which is the coefficient of the monomial $x^I$ in $f$, and we write $x^I \in f$ if $\coeff_f (x^I) \ne 0$.

The \textit{order} of $f$ is defined to be
\[
\ord (f) := \min \{\, d \in \mbZ_{\ge 0} \mid f_d \ne 0 \,\},
\]
where $f_d \in \mbC [x_1, \dots, x_n]$ is the degree $d$ homogeneous part of $f$.

Let $\msw$ be a weight on $x_1, \dots, x_n$ defined by $\msw (x_1, \dots, x_n) = (a_1, \dots, a_n)$ for some positive rational numbers $a_1, \dots, a_n$.
This means that the weight of $x_i$, denote by $\msw (x_i)$, is $a_i$ for $i = 1, \dots, n$. 
We define
\[
\msw (f) := \min \{\, \msw (x^I) \mid \alpha_I \ne 0 \,\},
\]
where the weight $\msw (x^I)$ for a monomial $x^I$ is defined in an obvious way.
For a rational number $r$, we define
\[
\begin{split}
f_{\msw = r} &:= \sum_{\msw (x^I) = r} \alpha_I x^I, \\
f_{\msw \ge r} &:= \sum_{\msw (x^I) \ge r} \alpha_I x^I.
\end{split}
\]
\end{Def}

\begin{Def}
Let $X$ be a variety (or a complex analytic space).
For an irreducible subvariety $\Gamma \subset X$, a \textit{divisor over $\Gamma \subset X$} is a prime divisor $E$ on some normal variety $Y$ admitting a birational morphism $\varphi \colon Y \to X$ such that the closure of $\varphi (E)$ is $\Gamma$. 
By a slight abuse of notation, we say that two divisors $E$ and $F$ over $\Gamma \subset X$ are \textit{distinct} if the valuations $\nu_E$ and $\nu_F$ of the function field $\mbC (X)$ are distinct. 
\end{Def}

\subsection{Generalities on the classification of divisorial contractions}
\label{sec:genclsfcont}

Let $\msp \in X$ be a hypersurface singularity
\[
(\phi (x_1, x_2, x_3, x_4) = 0) \subset \mbC^4_{x_1, x_2, x_3, x_4},
\]
where $\phi (x_1, x_2, x_3, x_4) \in \mbC \{x_1, x_2, x_3, x_4\}$.
We assume that $\msp \in X$ is a $cDV$ singularity.
Recall that a $cDV$ singularity is assumed to be isolated in this paper, hence $\msp \in X$ is a terminal Gorenstein singularity.

We explain a general strategy for classifying divisorial contractions to $\msp \in X$ of discrepancy $1$.  

\paragraph{\bf Step 0}
We assume that we are given a divisorial contraction $\varphi_0 \colon Y_0 \to X$ centered at $\msp$ of discrepancy $a$, that is, $K_{Y_0} = \varphi_0^*K_X + a E_0$, where $E_0$ is the $\varphi_0$-exceptional divisor.
Note that $a$ is a positive integer.

\paragraph{\bf Step 1}
We compute an upper bound on the number of distinct divisors of discrepancy $1$ over $\msp \in X$ in the following manner.
Let $G$ be a divisor over $\msp \in X$ such that $a_G (K_X) = 1$.
We may assume that $G$ is a divisor on a smooth variety $Z$ admitting a birational morphism $\psi \colon Z \to Y_0$.
We have
\[
1 = a_G (K_X) = a_G (K_{Y_0}) + a \ord_G (\psi^*E_0).
\]
This implies that $a_G (K_{Y_0}) < 1$, and thus the center $\psi (G) \subset E_0$ on $Y_0$ is a non-Gorenstein singular point of $Y_0$.
Moreover $E_0$ is not Cartier at the point $\psi (G)$ since $\ord_G (\psi^*E_0) < 1$.
Let $\msq_1, \dots, \msq_n$ be the non-Gorenstein singular points of $Y_0$.
Any divisor of discrepancy $1$ over $\msp \in X$ other than $E_0$ is a  divisor of discrepancy $< 1$ over $\msq_i \in Y_0$ for some $i$.
The number of divisors of discrepancy $< 1$ over a hyperquotient terminal singularity is determined in a series of works by Kawakita and Hayakawa (see e.g. \cite{Kawakita05}, \cite{Hayakawa99}, \cite{Hayakawa00}, \cite{Hayakawa05a}).
For a terminal quotient singularity, we have the following.

\begin{Lem}[{cf. \cite[(5.7)]{Reid87}}] \label{lem:quotdiv}
Let $V = \mbC^3_{x, y, z}/\mbZ_r (1, a, r-a)$ be the germ of a terminal quotient singularity of type $\frac{1}{r} (1, a, r-a)$, where $a < r$ are coprime positive integers.
For $k = 1, 2, \dots, r-1$, let $\varphi_k \colon W_k \to V$ be the weighted blow-up with 
\[
\wt (x, y, z) = \frac{1}{r} (k, [k a]_r, [k (r-a)]_r),
\] 
where $[m]_r$ denotes the smallest positive integer that is congruent to $m$ modulo $r$ for an integer $m$.
Denote by $E_k$ the $\varphi_k$-exceptional divisor.
Then 
\[
a_{E_k} (K_V) = \frac{k}{r}
\]
for any $k$, and $E_1, \dots, E_{r-1}$ are the divisors of discrepancy smaller than $1$ over $V$.
\end{Lem}

\paragraph{\bf Step 2}
We construct divisors of discrepancy $1$ over $\msp \in X$ as many as possible.
This will be done by considering weighted blow-ups with suitable weights, possibly composed with an analytic re-embedding $\msp \in X \cong \tilde{\msp} \in \tilde{X} \subset \mbC^4$. 

\begin{Lem} \label{lem:genwblexc}
Let $\msw$ be the weight defined by $\msw (x_1, x_2, x_3, x_4) = (a_1, a_2, a_3, a_4)$ for some positive integers $a_1, \dots, a_4$, and put $d := \msw (\phi)$.
Let $\varphi \colon Y \to X$ be the weighted blow-up with $\wt (x_1, x_2, x_3, x_4) = (a_1, a_2, a_3, a_4)$.
Assume that the polynomial $\phi_{\msw = d}$ is irreducible, and set $e := \sum_{i=1}^4 a_i - d - 1 > 0$.
Then the $\varphi$-exceptional divisor $E$ is a divisor of discrepancy $e$ over $\msp \in X$.
\end{Lem} 

\begin{proof}
For the $\varphi$-exceptional divisor $E$, we have a natural isomorphism 
\begin{equation} \label{eq:genwblexc}
E \cong (\phi_{\msw = d} (x_1, x_2, x_3, x_4) = 0) \subset \mbP (a_1, a_2, a_3, a_4),
\end{equation}
which is irreducible and reduced by the assumption.
The rest is easily verified by considering adjunction.
\end{proof}

\paragraph{\bf Step 3}
Let $n_1$ be an upper bound obtained in Step 1, and let $E_1, \dots, E_{n_2}$ be distinct divisors of discrepancy $1$ over $\msp \in X$ that are constructed in Step 2.
In general $n_2 \le n_1$.
Assume that $n_2 = n_1 =: n$.
Then $E_1, \dots, E_n$ are the divisors of discrepancy $1$ over $\msp \in X$.
If there exists a divisorial contraction of discrepancy $1$ over $\msp \in X$, then for its exceptional divisor $G$, one has $\nu_G = \nu_{E_i}$ for some $i = 1, 2, \dots, n$.
Such a divisorial contraction is said to {\it realize} the divisor $E_i$.
A divisorial contraction realizing a fixed $E_i$ is unique if it exists, and not every $E_i$ is realized by a divisorial contraction in general. 

\begin{Lem} \label{lem:divcontwbl}
Let the notation and assumption as in Lemma \ref{lem:genwblexc}.
There exists a divisorial contraction $\psi \colon Z \to X$ whose exceptional divisor $G$ satisfies $\nu_G = \nu_E$ if and only if $Y$ has only terminal singularities.
In this case the divisorial contraction is unique (up to isomorphism over $X$).
\end{Lem}

\begin{proof}
This follows from \cite[Lemma 3.4]{Kawakita01}.
\end{proof}

For $1 \le i \le n$, let $\varphi_i \colon Y_i \to X$ be the weighted blow-up whose exceptional divisor is $E_i$.
Then 
\[
\{\, \varphi_i \mid \text{$Y_i$ is terminal} \, \}
\]
is precisely the divisorial contractions of discrepancy $1$ over $\msp \in X$.
The following is useful when we show that $Y$ has a non-terminal singularity for a suitable weighted blow-up $\varphi \colon Y \to X$.

\begin{Lem} \label{lem:singwbl}
Let the notation and assumption as in Lemma \ref{lem:genwblexc}.
We identify the $\varphi$-exceptional divisor $E$ with the weighted hypersurface in $\mbP := \mbP (a_1, a_2, a_3, a_4)$ defined in \eqref{eq:genwblexc}.
Then the following assertions hold.
\begin{enumerate}
\item If $\phi_{\msw = d} \in (x_1, x_2, x_3)^3$, $\phi_{\msw = d + 1} \in (x_1, x_2, x_3)^2$, and $\phi_{\msw = d + 2} \in (x_1, x_2, x_3)$, then $Y$ has a singularity worse than terminal at the point $(0\!:\!0\!:\!0\!:\!1) \in \mbP$.
\item If $\phi_{\msw = d} \in (x_1, p)^2$ and $\phi_{\msw = d + 1} \in (x_1, p)$ for some irreducible quasi-homogeneous polynomial $p = p (x_2, x_3, x_4)$ with respect to $\wt (x_2, x_3, x_4) = (a_2, a_3, a_4)$, then $Y$ is singular along the curve $C := (x_1 = p = 0) \subset \mbP$, and in particular $Y$ has singularity worse than terminal.
\item If $\phi_{\msw = d} \in (x_1^2) + (x_2, x_3)^4$, $\phi_{\msw = d+1} \in (x_2, x_3)^3$, $\phi_{\msw = d+2} \in (x_2, x_3)^2$, and $\phi_{\msw = d+3} \in (x_2, x_3)$, then $Y$ has a singularity worse than terminal at the point $(0\!:\!0\!:\!0\!:\!1) \in \mbP$.
\end{enumerate} 
\end{Lem}

\begin{proof}
We consider the $x_4$-chart $Y_{x_4}$ of $Y$:
\[
Y_{x_4} \cong (\phi_{x_4} (x_1, x_2, x_3, x_4) = 0)/\mbZ_{a_4} (a_1, a_2, a_3, -1).
\]
where
\[
\phi_{x_4} (x_1, x_2, x_3, x_4) := \phi (x_1 x_4^{a_1}, x_2 x_4^{a_2}, x_3 x_4^{a_3}, x_4^{a_4})/x_4^d.
\]
The exceptional divisor $E$ is defined by $x_4 = 0$ on this chart.
The point $\msq$ corresponds to the origin and the set $C|_{Y_{x_4}}$ corresponds to the line $(x_1 = p_{x_4} = 0)$, where $p_{x_4} = p (x_2 x_4^{a_2}, x_3 x_4^{a_3}, x_4^{a_4})/x_4^{\msw (p)}$. 
Set $f_k (x_1, x_2, x_3) := \phi_{\msw = k} (x_1, x_2, x_3, 1) \in \mbC [x_1, x_2, x_3]$.
Then
\[
\phi_{x_4} = \sum_{i \ge 0} f_{d + i} (x_1, x_2, x_3) x_4^i.
\]

The assumption (1) implies that $f_{d+i} (x_1, x_2, x_3) x_4^i \in (x_1, x_2, x_3, x_4)^3$ for any $i \ge 0$, and thus $\msq := (0\!:\!0\!:\!0\!:\!1) \in Y$ is the quotient of a hypersurface singularity of multiplicity $\ge 3$.
Thus $\msq \in Y$ is not a terminal singularity.
Similarly, the assumption (2) implies that $f_{d+i} (x_1, x_2, x_3) x_4^i \in (x_1, p_{x_4}, x_4)^2$ for any $i \ge 0$, and thus $Y$ is singular along the curve $C \subset \mbP$.
Thus $Y$ has a non-terminal singularity since a $3$-dimensional terminal singularity is isolated.
Suppose that we are in case (3).
If $Y$ has a termial singularity at $\msq$, then so is the hypersurface singularity $(\phi_{x_4} = 0) \subset \mbC^4$.
In particular, $\ord (\phi_{x_4}) = 2$.
This implies that $x_1^2 \in \phi_{x_4}$ and the singularity $(\phi_{x_4} = 0)$ is equivalent to
\[
(x_1^2 + g (x_2, x_3, x_4) = 0) \subset \mbC^4,
\]
for some $g (x_2, x_3, x_4) \in (x_2, x_3, x_4)^4 \mbC \{x_2, x_3, x_4\}$.
It is easy to see that a general hypersurface germ cannot be a Du Val singularity, and thus $\msq \in Y$ is not a terminal singularity.
This completes the proof.
\end{proof}

\subsection{Equivalence of germs}

In this subsection, for a power series $f \in \mbC \{x, y, z, u\}$ and an integer $d$, we denote by $f_d$ the degree $d$ homogeneous part of $f$.

\begin{Lem} \label{lem:crdchginv}
Let $X = (\phi (x, y, z, u) = 0) \subset \mbC^4$ and $\tilde{X} = (\tilde{\phi} (x, y, z, u) = 0) \subset \mbC^4$ be $cE$ singularities, where
\[
\begin{split}
\phi &= x^2 + y^3 + y g (z, u) + h (z, u), \\
\tilde{\phi} &= x^2 + y^3 + y \tilde{g} (z, u) + \tilde{h} (z, u).
\end{split}
\]
Assume that $X$ and $\tilde{X}$ are equivalent.
Then the following hold.
\begin{enumerate}
\item There exists a linear transformation $\theta$ of $z, u$ and a nonzero $\xi \in \mbC$ such that 
\[
\tilde{h}_4 (z, u) = \xi h_4 (\theta (z), \theta (u)).
\]
\item If $\msp \in X$ is of type $cE_7$ or $cE_8$, then there exists a linear transformation $\theta$ of $z, u$, a nonzero $\xi \in \mbC$, and a homogeneous polynomial $b_2 (z, u)$ of degree $2$ such that
\[
\begin{split}
\tilde{g}_3 (z, u) &= \xi g_3 (\theta (z), \theta (u)), \\
\tilde{h}_5 (z, u) &= \xi (h_5 (\theta (z), \theta (u)) + b_2 (z, u) g_3 (\theta (z), \theta (u))).
\end{split}
\]
\end{enumerate}
\end{Lem}

\begin{proof}
Let $\sigma \colon \mbC \{x, y, z, u\} \to \mbC \{x, y, z, u\}$ be an automorphism which induces an equivalence of $X$ and $\tilde{X}$.
We write
\[
\sigma (x) = \sum_{i \ge 1} a_i, \quad
\sigma (y) = \sum_{i \ge 1} b_i, \quad 
\sigma (z) = \sum_{i \ge 1} c_i, \quad
\sigma (u) = \sum_{\i \ge 1} d_i,
\]
where $a_i, b_i, c_i, d_i \in \mbC [x, y, z, u]$ are homogeneous polynomials of degree $i$.
Then there exists $e = e_0 + e_1 + \cdots \in \mbC \{x, y, z, u\}$, where  $e_j \in \mbC [x, y, z, u]$ is homogeneous of degree $j$, such that $e_0 \ne 0$ and $\sigma (\phi) = e \tilde{\phi}$.
Hence we have
\begin{equation} \label{eq:crdchginv-1}
\begin{split}
\sigma (x)^2 + \sigma (y)^3 & + \sigma (y) g (\sigma (z), \sigma (u)) + h (\sigma (z), \sigma (u)) \\ 
&= (e_0 + e_1 + \cdots) (x^2 + y^3 + y \tilde{g} (z, u) + \tilde{h} (z, u))
\end{split}
\end{equation}
Comparing the degree $2$ terms and then degree $3$ terms in \eqref{eq:crdchginv-1}, we see that
\begin{equation} \label{eq:crdchginv-2}
a_1 = \alpha x, \quad
b_1 = \beta y + \eta x,
\end{equation}
for some $\alpha, \beta, \eta \in \mbC$ with $\alpha \ne 0$ and $\beta \ne 0$.
Moreover we have
\begin{equation} \label{eq:crdchginv-3}
a_2 \in (x, y^2).
\end{equation}

Comparing the degree $4$ terms in \eqref{eq:crdchginv-1}, we have
\begin{equation} \label{eq:crdchginv-4}
\begin{split}
2 a_1 a_3 + a_2^2 & + 3 b_1^2 b_2 + b_1 g_3 (c_1, d_1) + h_4 (c_1, d_1) \\
&= e_0 (y \tilde{g}_3 (z, u) + \tilde{h}_4 (z, u)) + e_1 y^3 + e_2 x^2.
\end{split}
\end{equation}
We set $\bar{c}_1 = c_1 (0, 0, z, u)$ and $\bar{d}_1 = d_1 (0, 0, z, u)$.
Then, we see that $z \mapsto \bar{c}_1, u \mapsto \bar{d}_1$ gives a linear transformation which we denote by $\theta$.
By setting $x = y = 0$ in \eqref{eq:crdchginv-4}, we obtain 
\begin{equation} \label{eq:crdchginv-5}
h_4 (\theta (z), \theta (u)) = h_4 (\bar{c}_1, \bar{d}_1) = e_0 \tilde{h}_4 (z, u),
\end{equation}
where $\bar{c}_1 = c_1 (0, 0, z, u)$ and $\bar{d}_1 = d_1 (0, 0, z, u)$.
This is (1).

In the following, suppose that $X$ is of type $cE_7$ or $cE_8$.
Then $h_4 (z, u) = 0$.
We have $\tilde{h}_4 (z, u) = 0$ by (1).
By setting $x = 0$ in \eqref{eq:crdchginv-4} and then by comparing the terms not divisible by $y^2$, we have 
\begin{equation} \label{eq:crdchginv-6}
g_3 (\bar{c}_1, \bar{d}_1) = e_0 \tilde{g}_3 (z, u).
\end{equation}
By comparing the degree $5$ terms in \eqref{eq:crdchginv-1} and then by setting $x = y = 0$, we obtain
\begin{equation} \label{eq:crdchginv-7}
b_2 (0, 0, z, u) g_3 (\bar{c}_1, \bar{d}_1) + h_5 (\bar{c}_1, \bar{d}_1) = e_0 \tilde{h}_5 (z, u).
\end{equation}
The assertion (2) follows from \eqref{eq:crdchginv-6} and \eqref{eq:crdchginv-7}.
\end{proof}

The following technical results will be used in Section \ref{sec:divcontcEgt1}.

\begin{Lem} \label{lem:eqgrmsp7}
Let $X = (\phi (x, y, z, u) = 0) \subset \mbC^4$ and $\tilde{X} = (\tilde{\phi} (x, y, z, u) = 0) \subset \mbC^4$ be $cE_7$ singularities, where
\[
\begin{split}
\phi &= x^2 + y^3 + \lambda y^2 u^2 + y g (z, u) + h (z, u), \\
\tilde{\phi} &= x^2 + y^3 + \tilde{\lambda} y^2 u^2 + y \tilde{g} (z, u) + \tilde{h} (z, u),
\end{split}
\]
for some $\lambda, \tilde{\lambda} \in \mbC$ and $g, h, \tilde{g}, \tilde{h} \in \mbC \{z, u\}$ of order at least $3, 5, 3, 5$, respectively.
Assume that
\[
g_3 (z, u) = \varepsilon z^3, \quad
\tilde{g}_3 (z, u) = \tilde{\varepsilon} z^3, \quad
z^4 \mid h_5 (z, u), \quad
z^4 \mid \tilde{h}_5 (z, u),
\] 
for some $\varepsilon, \tilde{\varepsilon} \in \mbC \setminus \{0\}$.
Let $\chi \colon \mbC \{x, y, z, u\} \to \mbC \{x, y, z, u\}$ be an automorphism which induces an equivalence of $X$ and $\tilde{X}$, and write
\[
\chi (x) = \sum_{i \ge 1} a_i, \quad
\chi (y) = \sum_{i \ge 1} b_i, \quad 
\chi (z) = \sum_{i \ge 1} c_i, \quad
\chi (u) = \sum_{i \ge 1} d_i,
\]
where $a_i, b_i, c_i, d_i \in \mbC [z, u]$ are homogeneous polynomials of degree $i$.
Then we have
\[
a_1 = \alpha x, \quad
b_1 = \beta y + \eta x, \quad
c_1 = \gamma z + \ell_c (x, y), \quad
d_1 = \delta u + \ell_d (x, y, z),
\]
for some $\alpha, \beta, \gamma, \delta, \eta \in \mbC$ with $\alpha \beta \gamma \delta \ne 0$ and for some linear forms $\ell_c (x, y)$, $\ell_d (x, y, z)$, and moreover we have
\[
a_2 \in (x, y^2), \quad
a_3 \in (x, y, z^3), \quad
b_2 \in (x, y, z).
\]
\end{Lem}

\begin{proof}
There exists $e = e_0 + e_1 + \cdots \in \mbC \{x, y, z, u\}$, where $e_j \in \mbC [x, y, z, u]$ is homogeneous of degree $j$, such that $e_0 \ne 0$ and $\sigma (\phi) = e \tilde{\phi}$.
Explicitly, we have
\begin{equation} \label{eq:eqgrmsp-1}
\begin{split}
\chi (x)^2 + \chi (y)^3 & + \lambda \chi (y)^2 \chi (u)^2 + \chi (y) g (\chi (z), \chi (u)) + h (\chi (z), \chi (u)) \\
&= (e_0 + e_1 + \cdots)(x^2 + y^3 + \tilde{\lambda} y^2 u^2 + y \tilde{g} (z, u) + \tilde{h} (z, u)).
\end{split}
\end{equation}
By comparing degree $2$ and degree $3$ terms in \eqref{eq:eqgrmsp-1}, we conclude that $a_1 = \alpha x$, $b_1 = \beta y + \eta x$ for some $\alpha, \beta, \eta \in \mbC$ with $\alpha \beta \ne 0$, and $a_2 \in (x, y^2)$.
By comparing the degree $4$ terms in \eqref{eq:eqgrmsp-1}, we have
\begin{equation} \label{eq:eqgrmsp-2}
2 a_1 a_3 + a_2^2 + 3 b_1^2 b_2 + \lambda b_1^2 d_1^2 + \varepsilon b_1 c_1^3 = e_0 (\tilde{\lambda} y^2 u^2 + \tilde{\varepsilon} y z^3) + e_1 y^3 + e_2 x^2.
\end{equation}
It is easy to see that $u \notin c_1$ since the monomial $y u^3$ can only appear in $\varepsilon b_1 c_1^3$ in the equation \eqref{eq:eqgrmsp-2}, and thus we can write $c_1 = \gamma z + \ell_c (x, y)$ for some $\gamma \in \mbC$ and some linear form $\ell_c (x, y)$.
We write $d_1 = \delta u + \ell_d (x, y, z)$ for some $\delta \in \mbC$ and some linear form $\ell_d (x, y, z)$.
Then we have $\gamma \delta \ne 0$ since $\chi$ is an automorphism.
By removing terms contained in the ideal $(x^2, y)$ from \eqref{eq:eqgrmsp-2}, we have
\begin{equation} \label{eq:eqgrmsp-3}
2 \alpha x a_3 (0, 0, z, u) + \varepsilon \eta \gamma^3 x z^3 = 0,
\end{equation}
which shows that $a_3 \in (x, y, z^3)$.
By comparing the degree $5$ terms in \eqref{eq:eqgrmsp-1} and by setting $x = y = 0$, we have
\[
\varepsilon \gamma^3 b_2 (0, 0, z, u) z^3 + h_5 (\gamma z, d_1 (0, 0, z, u)) = e_0 \tilde{h}_5 (z, u).
\]
By assumption, both $h_5 (\gamma z, d_1 (0, 0, z, u))$ and $\tilde{h}_5 (z, u)$ are divisible by $z^4$.
Thus $b_2 \in (x, y, z)$ and the proof is completed.
\end{proof}

\begin{Lem} \label{lem:eqgrmsp8}
Let $X = (\phi (x, y, z, u) = 0) \subset \mbC^4$ and $\tilde{X} = (\tilde{\phi} (x, y, z, u) = 0) \subset \mbC^4$ be $cE_7$ singularities, where
\[
\begin{split}
\phi &= x^2 + y^3 + \lambda y^2 u^2 + y g (z, u) + h (z, u), \\
\tilde{\phi} &= x^2 + y^3 + \tilde{\lambda} y^2 u^2 + y \tilde{g} (z, u) + \tilde{h} (z, u),
\end{split}
\]
for some $\lambda, \tilde{\lambda} \in \mbC$ and $g, h, \tilde{g}, \tilde{h} \in \mbC \{z, u\}$ of order at least $4, 5, 4, 5$, respectively.
Assume that
\[
\begin{split}
z & \mid g_4 (z, u), \quad
z^4 \mid h_5 (z, u), \quad
z^2 \mid h_6 (z, u), \\
z & \mid \tilde{g}_4 (z, u), \quad
z^4 \mid \tilde{h}_5 (z, u), \quad
z^2 \mid \tilde{h}_6 (z, u).
\end{split}
\] 
Let $\chi \colon \mbC \{x, y, z, u\} \to \mbC \{x, y, z, u\}$ be an automorphism which induces an equivalence of $X$ and $\tilde{X}$, and write
\[
\chi (x) = \sum_{i \ge 1} a_i, \quad
\chi (y) = \sum_{i \ge 1} b_i, \quad 
\chi (z) = \sum_{i \ge 1} c_i, \quad
\chi (u) = \sum_{\i \ge 1} d_i,
\]
where $a_i, b_i, c_i, d_i \in \mbC [z, u]$ are homogeneous polynomials of degree $i$.
Then we have
\[
a_1 = \alpha x, \quad
b_1 = \beta y + \eta x, \quad
c_1 = \gamma z + \ell_c (x, y), \quad
d_1 = \delta u + \ell_d (x, y, z),
\]
for some $\alpha, \beta, \gamma, \delta, \eta \in \mbC$ with $\alpha \beta \gamma \delta \ne 0$ and for some linear forms $\ell_c (x, y)$, $\ell_d (x, y, z)$, and moreover we have
\[
a_2 \in (x, y^2), \quad
a_3 \in (x, y), \quad
b_2 \in (x, y, z).
\]
\end{Lem}

\begin{proof}
As in the proof of Lemma \ref{lem:eqgrmsp7}, there is an $e = e_0 + e_1 + \cdots \in \mbC \{x, y, z, u\}$ with $e_0 \ne 0$ such that the equation \eqref{eq:eqgrmsp-1} holds.
By comparing the degree $2$ and $3$ terms, we conclude that $a_1 = \alpha x$, $b_1 = \beta y + \eta x$ for some $\alpha, \beta, \eta \in \mbC$ with $\alpha \beta \ne 0$, and $a_2 \in (x, y^2)$.
By removing terms contained in the ideal $(x^2, y)$ from the degree $4$ part of \eqref{eq:eqgrmsp-1}, we have
\[
2 \alpha x a_3 (0, 0, z, u) = 0,
\]
which is obtained by setting $\varepsilon = 0$ in \eqref{eq:eqgrmsp-3}.
This shows $a_3 \in (x, y)$.
By comparing the degree $5$ terms in \eqref{eq:eqgrmsp-1} and by setting $x = y = 0$, we have
\[
h_5 (\bar{c}_1, \bar{d}_1) = e_0 \tilde{h}_5 (z, u),
\]
where $\bar{c}_1 = c_1 (0, 0, z, u)$ and $\bar{d}_1 = d_1 (0, 0, z, u)$.
We can write 
\[
h_5 (\bar{c}_1, \bar{d}_1) = \bar{c}_1^4 (\mu \bar{c}_1 + \nu \bar{d}_1),
\]
for some $\mu, \nu \in \mbC$.
This in particular shows that $u \notin c_1$ since $z^4 \mid \tilde{h}_5 (z, u)$, and hence we can write $c_1 = \gamma z + \ell_z (x, y)$ for some $\gamma \in \mbC \setminus \{0\}$ and a linear form $\ell_z (x, y)$.
Then we can write $d_1 = \delta u + \ell_u (x, y, z)$ for some $\delta \in \mbC \setminus \{0\}$ and a linear form $\ell_u (x, y, z)$.

It remains to show that $b_2 \in (x, y, z)$ which will be done by comparing the degree $6$ terms in \eqref{eq:eqgrmsp-1}.
We see that the degree $6$ part of $h_5 (\chi (z), \chi (u))$ is contained in the ideal $(x, y, z)$ since 
\[
h_5 (\chi (z), \chi (u)) = \chi (z)^4 (\mu \chi (z) + \nu \chi (u)),
\]
and $\chi (z) = b_1 + b_2 + \cdots$ with $b_1 \in (x, y)$.
The degree $6$ part of $h_6 (\chi (z), \chi (u))$ is $h_6 (c_1, d_1)$ and it is contained in $(x, y, z)$ since $c_1 \mid h_6 (c_1, d_1)$ and $c_1 \in (x, y, z)$.
The degree $6$ part of $\chi (y) g (\chi (z), \chi (u))$ is the sum of terms divisible by $b_1 \in (x, y)$ and $b_2 g_4 (c_1, d_1) \in (x, y, z)$, and thus it is contained in $(x, y, z)$.
By comparing the degree $6$ terms in \eqref{eq:eqgrmsp-1} and by setting $x = y = z = 0$, we have $b_2 (0, 0, 0, u)^3 = 0$.
This shows that $u^2 \notin b_2$, that is, $b_2 \in (x, y, z)$.
This completes the proof.
\end{proof}

\subsection{Divisorial contractions to $cE$ points of discrepancy greater than $1$} \label{sec:divcontcEgt1}

Divisorial contractions of discrepancy greater than $1$ to $cE$ points are classified in \cite{Kawakita05} and \cite{Yamamoto}.
According to the classification, any divisorial contraction of discrepancy greater than $1$ over $\msp \in X$ is of type $e2$ and of discrepancy $2$ when $\msp \in X$ is of type $cE_6$, is of type either $e5$ or $e9$ both of discrepancy $2$ when $\msp \in X$ is of type $cE_7$, and is of type $e9$ and of discrepancy $2$ when $\msp \in X$ is of type $cE_8$.
We refer readers to \cite{Kawakita05} and \cite{Yamamoto} for descriptions of divisorial contractions of type $e2$, $e5$, and $e9$ to a $cE$ point.
In this section, we will observe that not every $cE$ point admit such a divisorial contraction and give some necessary conditions for a $cE$ point to admit a divisorial contraction of discrepancy greater than $1$. 

\begin{Prop} \label{prop:ncd2cE6}
Let $\msp \in X$ be a $cE_6$ singularity
\[
(x^2 + y^3 + y g (z, u) + h (z, u) = 0) \subset \mbC^4_{x, y, z, u},
\]
where $g, h \in \mbC \{z, u\}$ are power series of order at least $3$ and $4$, respectively.
Suppose that $\msp \in X$ admits a divisorial contraction of discrepancy greater than $1$.
Then the following are satisfied.
\begin{enumerate}
\item The degree $4$ part $h_4 (z, u)$ of $h (z, u)$ is not the quadruple of a linear form.
\item There are at most $3$ divisors of discrepancy $1$ over $\msp \in X$.
\end{enumerate}
\end{Prop}

\begin{proof}
Let $\varphi \colon Y \to X$ be a divisorial contraction of discrepancy greater than $1$ over $\msp \in X$.
By \cite{Kawakita05} and \cite{Yamamoto}, $\varphi$ is a divisorial contraction of type $e2$ whose explicit description is given in \cite[Theorem 2.5]{Yamamoto}.
There exists an equivalence between $\msp \in X$ and the germ
\[
(x^2 + (y - p (z, u))^3 + y \hat{g} (z, u) + \hat{h} (z, u) = 0) \subset \mbA^4,
\]
where $p \in \mbC [z, u]$ is a polynomial and $\hat{g}, \hat{h} \in \mbC \{z, u\}$ are power series satisfying
\begin{enumerate}
\item[(i)] $p = \lambda z + \mu u^2$ for some $\lambda, \mu \in \mbC$ with $\lambda \ne 0$;
\item[(ii)] $\ord (\hat{g}) \ge 3$ and $u^3 \in \hat{g}$;
\item[(iii)] $\ord (\hat{h}) \ge 4$ and $\msw_2 (\hat{h}) \ge 6$, where the weight $\msw_2$ is defined by $\msw_2 (z, u) = (2, 1)$;
\end{enumerate}
and $\varphi$ is the weighted blow-up with $\wt (x, y, z, u) = (3, 3, 2, 1)$.

Rescaling $u$, we may assume that $\hat{g}_{\msw_2 = 3} (z, u) = u^3$.
By a coordinate change $y \mapsto y + p$, the germ $\msp \in X$ is equivalent to 
\[
(x^2 + y^3 + y \hat{g} (z, u) + \hat{h} (z, u) + \hat{g} (z, u) p (z, u) = 0) \subset \mbC^4.
\]
The degree $4$ part of $\hat{h} + \hat{g} p$ can be written as
\[
(\alpha z^4 + \beta z^3 u + \gamma z^2 u^2) + \lambda z u^3,
\]
for some $\alpha, \beta, \gamma \in \mbC$, which is not the quadruple of a linear form since $\lambda \ne 0$.
By Lemma \ref{lem:crdchginv}, $h_4$ is not the quadruple of a linear form, and (1) is verified.

Let $G$ be a divisor of discrepancy $1$ over $\msp \in X$.
We may assume that $G$ is a prime divisor on a smooth projective variety $Z$ admitting a birational morphism $\psi \colon Z \to Y$.
We have
\[
1 = a_G (K_X) = a_G (K_Y) + 2 \ord_G (\psi^* E).
\]  
By \cite[Table 3]{Kawakita05}, $Y$ has a unique singular point at which $E$ is not Cartier, and the unique point $\msq \in Y$ is of type $cD/3$.
It follows that $\psi (G) = \{\msq\}$.
We have $\ord_G (\psi^* E) \ge 1/3$ since $3 E$ is Cartier at $\msq$.
Thus we have $a_G (K_Y) = 1/3$.
By \cite[4.6 Corollary]{Hayakawa99}, a $cD/3$ point admits at most $3$ distinct divisors of discrepancy $1/3$ over the point.
This shows (2).
\end{proof}

\begin{Prop} \label{prop:ncd2cE7e5}
Let $\msp \in X$ be a $cE_7$ singularity
\[
(x^2 + y^3 + y g (z, u) + h (z, u) = 0) \subset \mbC^4_{x, y, z, u},
\]
where $g, h \in \mbC \{z, u\}$ are power series of order at least $3$ and $5$, respectively.
Suppose that $\msp \in X$ admits a divisorial contraction of type $e5$ and of discrepancy $2$.
Then the following are satsfied.
\begin{enumerate}
\item $\gcd \{ g_3, h_5 \} = 1$, where $g_3$ and $h_5$ are the degree $3$ and $5$ parts of $g$ and $h$, respectively.
\item There is a unique divisor of discrepancy $1$ over $\msp \in X$.
\end{enumerate}
\end{Prop}

\begin{proof}
Let $\varphi \colon Y \to X$ be a divisorial contraction of type $e5$ and of discrepancy $2$.
Then, by \cite[Theorem 2.9]{Yamamoto}, $X$ is equivalent to the germ
\[
\begin{pmatrix}
x^2 + y t + p (z, u) = 0, \\
y^2 + q (z, u) - t = 0
\end{pmatrix}
\subset \mbC^5_{x, y, z, u, t},
\]
where
\begin{itemize}
\item[(i)] $p, q \in \mbC \{z, u\}$ with $\ord p \ge 5$ and $\ord q \ge 3$;
\item[(ii)] $\gcd \{ p_5, q_3 \} = 1$ for the degree $5$ part $p_5$ and the degree $3$ part $q_3$ of $q$;
\end{itemize}
and $\varphi$ is the weighted blow-up with $\wt (x, y, z, u, t) = (5, 3, 2, 2, 7)$.

Eliminating the variable $t$, we see that $\msp \in X$ is then equivalent to the germ
\[
(x^2 + y^3 + y q (z, u) + p (z, u) = 0) \subset \mbC^4_{x, y, z, y}.
\]
By the condition (ii) and by Lemma \ref{lem:crdchginv}, the assertion (1) is verified.

By \cite[Table 1]{Yamamoto}, $Y$ admits a unique non-Gorenstein singular point and it is of type $\frac{1}{7} (1, 1, 6)$.
Let $\msq \in Y$ be the $\frac{1}{7} (1, 1, 6)$ point.
It is easy to see that $\msq$ is the origin of the $t$-chart $Y_t$ of $Y$:
\[
Y_t \cong 
\begin{pmatrix}
x^2 + y + p (z t^2, u t^2)/t^{10} = 0, \\
y^2 + q (z t^2, u t^2)/t^6 - t = 0
\end{pmatrix}
/\mbZ_7 (5_x, 3_y, 2_z, 2_u, 6_t).
\]
The $\varphi$-exceptional divisor $E$ is defined by $t = 0$ on $Y_t$, and we have a natural equivalence
\[
\msq \in Y_t \cong \mbC^3_{x, z, u}/\mbZ_7 (5, 2, 2).
\]
Eliminating $y$ and then filtering off terms divisible by $t$, we have
\begin{equation} \label{eq:ncd2cE7e5-1}
\xi t =(x^2 + p_5 (z, u))^2 + q_3 (z, u),
\end{equation}
where $\xi = \xi (x, y, z, t, u)$ does not vanish at $\msq$.
For $1 \le k \le 6$, let $\psi_k \colon Z_k \to Y$ be the weighted blow-up at $\msq \in Y$ with $\wt (x, z, u) = \frac{1}{7} ([6 k]_7, k, k)$.
We see that $a_{G_k} (K_Y) = k/7$ and, by \eqref{eq:ncd2cE7e5-1}, we have
\[
\ord_{G_k} (\psi^*_k E) = 
\begin{cases}
\frac{3 k}{7}, & (1 \le k \le 4), \\
\frac{8}{7}, & (k = 5), \\
\frac{4}{7}, & (k = 6).
\end{cases}
\]
It follows that $G_1$ is the unique divisor of discrepancy $1$ over $\msp \in X$, and (2) is satidfied.
\end{proof}

\begin{Prop} \label{prop:ncd2cE7e9}
Let $\msp \in X$ be a $cE_7$ singularity
\[
(x^2 + y^3 + y g (z, u) + h (z, u) = 0) \subset \mbC^4_{x, y, z, u},
\]
where $g, h \in \mbC \{z, u\}$ are power series of order at least $3$ and $5$, respectively.
If $\msp \in X$ admits a divisorial contraction of type $e9$ and of discrepancy $2$, then the following assertions hold.
\begin{enumerate}
\item There exists a linear form $\ell = \ell (z, u) \ne 0$ such that $g_3  = \ell^3$ and $\ell^4 \mid h_5$, where $g_3$ and $h_5$ are the degree $3$ and $5$ parts of $g$ and $h$, respectively. 
\item There are exactly $2$ divisors of discrepancy $1$ over $\msp \in X$.
\item The germ $\msp \in X$ admits a unique divisorial contraction of discrepancy greater than $1$.
\end{enumerate}
\end{Prop}

\begin{proof}
Let $\varphi \colon Y \to X$ be a divisorial contraction of type $e9$ and of discrepancy $2$.
By \cite[Theorem 2.10]{Yamamoto}, $\msp \in X$ is equivalent to the germ
\[
(x^2 + y^3 + \lambda y^2 u^2 + y \hat{g} (z, u) + \hat{h} (z, u) = 0) \subset \mbC^4_{x, y, z, u},
\] 
where $\lambda \in \mbC$ and $\hat{g}, \hat{h} \in \mbC \{z, u\}$ satisfy the following properties:
\begin{itemize}
\item[(i)] $\msw (\hat{g}) \ge 9$ and $\msw (\hat{h}) \ge 14$, where  the weight $\msw$ is defined by $\msw (z, u) = (3, 2)$;
\item[(ii)] $z^3 \in \hat{g}$ and $u^7 \in \hat{h}$;
\end{itemize}
and under the identification of germs, $\varphi$ is the weighted blow-up with $\wt (x, y, z, u) = (7, 5, 3, 2)$.

We prove (1).
By (i) and (ii), for the degree $3$ and $5$ parts $\hat{g}_3$ and $\hat{h}_5$ of $\hat{g}$ and $\hat{h}$, respectively, we see that
\[
\hat{g}_3 = \alpha z^3, \quad
\hat{h}_5 = \beta z^5 + \gamma z^4 u,
\]
for some $\alpha, \beta, \gamma \in \mbC$ with $\alpha \ne 0$.
By the coordinate change $y \mapsto y - \lambda u^2/3$, $\msp \in X$ is equivalent to the germ
\[
(x^2 + y^3 + y (\hat{g} (z, u) - \lambda^3 u^4/3) + \hat{h} (z, u) - \lambda^3 u^6/27 = 0) \subset \mbC^4_{x, y, z, u},
\]
The assertion (1) follows from Lemma \ref{lem:crdchginv}.   

We prove (2).
By \cite[Table 1]{Yamamoto}, $Y$ has exactly two non-Gorenstein singular points which consist of a $\frac{1}{3} (1, 1, 2)$ point denoted by $\msq$ and a $\frac{1}{5} (1, 1, 4)$ point denoted by $\msq'$.
The point $\msq$ is the origin of the $z$-chart $Y_z$ of $Y$:
\[
Y_z = (\phi_z (x, y, z, u) = 0)/\mbZ_3 (1_x, 2_y, 2_z, 2_u),
\]
where
\[
\phi_z := x^2 + y^3 z  + \lambda y^2 u^2 + y \hat{g} (z^3, u z^2)/z^9 + \hat{h} (z^3, u z^2)/z^{14}.
\]
The $\varphi$-exceptional divisor $E$ is defined by $z = 0$ on $Y_z$.
We have $y \in \phi_z$ since $z^3 \in \hat{g}$.
Eliminating $y$, we have a natural equivalence
\[
\msq \in Y \cong \mbC^3_{x, z, u}/\mbZ_3 (1, 2, 2).
\]
Let $\psi_1 \colon Z_1 \to Y$ and $\psi_2 \colon Z_2 \to Y$ be the weighted blow-ups of $\msq \in Y$ with $\wt (x, z, u) = \frac{1}{3} (1, 1, 2)$ and $\wt (x, z, u) = \frac{1}{3} (2, 2, 1)$, respectively, and denote by $G_i$ the $\psi_i$-exceptional divisor.
Then we compute
\[
\begin{split}
a_{G_1} (X) &= a_{G_1} (Y) + 2 \ord_{G_1} (E) = \frac{1}{3} + \frac{2}{3} = 1, \\
a_{G_2} (X) &= a_{G_2} (Y) + 2 \ord_{G_2} (E) = \frac{2}{3} + \frac{4}{3} = 2.
\end{split}
\]
The point $\msq'$ is the origin of the $y$-chart $Y_y$:
\[
Y_y = (\phi_y (x, y, z, u) = 0)/\mbZ_5 (2_x, 4_y, 3_z, 2_u),
\]
where
\[
\phi_y := x^2 + y + \lambda u^2 + \hat{g} (z y^3, u y^2)/y^9 + \hat{h} (z y^3, u y^2)/y^{14}.
\]
The exceptional divisor $E$ is defined by $y = 0$ on $Y_y$.
Eliminating $y$, we have a natural identification
\[
\msq' \in Y \cong \mbC^3_{x, z, u}/\mbZ_5 (2, 3, 2),
\]
and by filtering terms divisible by $y$ in the equation $\phi_y = 0$, we have
\[
\xi y = x^2 + \lambda u^2 + \hat{g}_{\msw = 9} (z, u) + \hat{h}_{\msw = 14} (z, u),
\]
where $\xi = \xi (x, y, z, u)$ does not vanish at $\msq'$ and the weight $\msw$ is defined by $\msw (z, u) = (3, 2)$. 
For $k = 1, 2, 3, 4$, let $\psi'_k \colon Z_k \to Y$ be the weighted blow-up of $\msq' \in Y$ with $\wt (x, z, u) = \frac{1}{5} (k, [4 k]_5, k)$.
Denote by $G'_k$ the $\psi'_k$-exceptional divisor.
We have $a_{G'_k} (Y) = k/5$ for $1 \le k \le 4$, and we compute
\[
\ord_{G_k} ({\psi'}_k^*E) =
\begin{cases}
\frac{2k}{5}, & (1 \le k \le 3), \\
\frac{3}{5}, & (k = 4).
\end{cases}
\]
Thus $G_1$ and $G'_1$ are the divisors of discrepancy $1$ over $\msp \in X$. 

We prove (3).
Let $\tilde{\varphi} \colon \tilde{Y} \to X$ be a divisorial contraction to $\msp \in X$ of discrepancy greater than $1$ with exceptional divisor $\tilde{E}$.
By the classification \cite{Yamamoto}, the discrepancy of $\tilde{\varphi}$ is $2$ and it is of type either $e5$ or $e9$.
If $\tilde{\varphi}$ is of type $e5$, then $\msp \in X$ is equivalent to a germ described in Proposition \ref{prop:ncd2cE7e5}.
By Lemma \ref{lem:crdchginv} and (1), this is impossible.
Thus $\tilde{\varphi}$ is of type $e9$.
We have an equivalence $\chi \colon \msp \in X \cong \tilde{\msp} \in \tilde{X}$, where 
\[
\tilde{\msp} \in \tilde{X} = (x^2 + y^3 + \tilde{\lambda} y^2 u^2 + y \tilde{g} (z, u) + \tilde{h} (z, u) = 0) \subset \mbC^4_{x, y, z, u}
 \]
 for some $\tilde{\lambda} \in \mbC$ and $\tilde{g}, \tilde{h} \in \mbC \{z, u\}$ satisfying the above conditions (i), (ii), and $\tilde{\varphi}$ is the composition of the weighted blow-up of $\tilde{X}$ with $\wt (x, y, z, u) = (7, 5, 3, 2)$ and $\chi$.
By Lemma \ref{lem:eqgrmsp7}, we see that
\[
\sigmawt (\chi^* x, \chi^* y, \chi^*z, \chi^* t) = (7, 5, 3, 2).
\]
This shows that $E$ and $\tilde{E}$ define the same valuation, and thus $\varphi$ is isomorphic to $\tilde{\varphi}$ over $X$ by \cite[Lemma 3.4]{Kawakita01}.
Therefore $\varphi$ is the unique divisorial contraction of discrepancy greater than $1$.
\end{proof}

\begin{Prop} \label{prop:ncd2cE8e9}
Let $\msp \in X$ be a $cE_8$ singularity
\[
(x^2 + y^3 + y g (z, u) + h (z, u) = 0) \subset \mbC^4_{x, y, z, u},
\]
where $g, h \in \mbC \{z, u\}$ are power series of order at least $4$ and $5$, respectively.
If $\msp \in X$ admits a divisorial contraction of discrepancy greater than $1$, then the following assertions hold.
\begin{enumerate}
\item For the degree $3$ and $5$ parts of $g$ and $h$, respectively, we have $g_3 = 0$ and $h_5$ is divisible by the quadruple of a linear form.
\item There are at most $2$ divisors of discrepancy $1$ over $\msp \in X$.
\item The germ $\msp \in X$ admits a unique divisorial contraction of discrepancy greater than $1$.
\end{enumerate}
\end{Prop}

\begin{proof}
Let $\varphi \colon Y \to X$ be a divisorial contraction to $\msp \in X$ of discrepancy greater than $1$ with exceptional divisor $E$.
By \cite{Kawakita05}, $\varphi$ is of type $e9$ and of discrepancy $2$.
By \cite[Theorem 2.10]{Yamamoto}, $\msp \in X$ is equivalent to the germ
\[
(x^2 + y^3 + \lambda y^2 u^2 + y \hat{g} (z, u) + \hat{h} (z, u) = 0),
\] 
where $\lambda \in \mbC$ and $\hat{g}, \hat{h} \in \mbC \{z, u\}$ are power series satisfying the following properties:
\begin{itemize}
\item[(i)] $\msw (\hat{g}) \ge 9$ and $\msw (\hat{h}) \ge 14$, where  the weight $\msw$ is defined by $\msw (z, u) = (3, 2)$;
\item[(ii)] $u^7 \in \hat{h}$, and either $z^5 \in \hat{h}$ or $z^4 u \in \hat{h}$;
\end{itemize}
and under the identification of germs, $\varphi$ is the weighted blow-up with $\wt (x, y, z, u) = (7, 5, 3, 2)$.

By \cite[Table 1]{Yamamoto}, $Y$ has exactly two non-Gorenstein singular points which consist of a $\frac{1}{3} (1, 1, 2)$ point and a $\frac{1}{5} (1, 1, 4)$ point.
The situation is almost the same as in the proof of Proposition \ref{prop:ncd2cE7e9}, and we can prove (1) and (2) by the same arguments.

The proof of (3) is completely parallel to that of Proposition \ref{prop:ncd2cE7e9}(3).
Suppose that there is a divisorial contraction $\tilde{\varphi} \colon \tilde{Y} \to X$ to $\msp \in X$ of discrepancy greater than $1$ with exceptional divisor $\tilde{E}$.
Then the discrepancy of $\tilde{\varphi}$ is $2$ and it is of type $e9$.
Thus there is an equivalence $\chi \colon \msp \in X \cong \tilde{\msp} \in \tilde{X}$, where
\[
\tilde{\msp} \in \tilde{X} =
(x^2 + y^3 + \tilde{\lambda} y^2 u^2 + y \tilde{g} (z, u) + \tilde{h} (z, u) = 0) \subset \mbC^4_{x, y, z, u},
\] 
for some $\tilde{\lambda} \in \mbC$ and $\tilde{g}, \tilde{h} \in \mbC \{z, u\}$ satisfy the same assumptions (i) and (ii) above, and $\tilde{\varphi}$ is the composition of the weighted blow-up with $\wt (x, y, z, u) = (7, 5, 3, 2)$ and $\chi$.
For the degree $6$ parts $\hat{h}_6$ and $\tilde{h}_6$ of $\hat{h}$ and $\tilde{h}$, respectively, we have $z^2 \mid \hat{h}_6 (z, u)$ and $z^2 \mid \tilde{h}_6 (z, u)$ since $\msw (\hat{h}) \ge 14$ and $\msw (\tilde{h}) \ge 14$ by assumptions.
Hence we can apply Lemma \ref{lem:eqgrmsp7} and conclude that $E$ and $\tilde{E}$ define the same valuation.
This completes the proof.
\end{proof}

\subsection{Divisors of discrepancy $1$ over a $cE$ point} \label{sec:divdisc1}

Throughout the present subsection, let 
\[
\msp \in X = (\phi (x, y, z, u) = 0) \subset \mbC^4_{x, y, z, u}
\] 
be a $cE$ singularity, where
\[
\phi = x^2 + y^3 + y^2 f (z, u) + y g (z, u) + h (z, u),
\]
for some $f, g, h \in \mbC \{z, u\}$.
For a quadruple $(a, b, c, d)$ of positive integers, we denote by 
\[
\varphi_{(a, b, c, d)} \colon Y_{(a, b, c, d)} \to X
\]
the weighted blow-up of the germ $\msp \in X$ with $\wt (x, y, z, u) = (a, b, c, d)$.
We denote by $E_{(a, b, c, d)}$ the exceptional divisor of $\varphi_{(a, b, c, d)}$.

The choices of various blow-ing up weights and choices of re-embeddings in Lemmas \ref{lem:divdisc1cEv1}--\ref{lem:divdisc1cEv4} below are due to Hayakawa \cite{HayakawaE}.

\begin{Lem} \label{lem:divdisc1cEv1}
The following are divisors of discrepancy $1$ over $\msp \in X$.  
\begin{enumerate}
\item The divisor $E_{(3, 2, 2, 1)}$ under the assumption that $\msw_2 (f) \ge 2$, $\msw_2 (g) \ge 4$, and $\msw_2 (h) \ge 6$. 
If in addition $\msw_2 (g) \ge 5$ and $\msw_2 (h) \ge 8$, then $Y_{(3, 2, 2, 1)}$ has a non-terminal singularity.

\item The divisor $E_{(3, 3, 1, 1)}$ under the assumption that $\msw_1 (g) = 3$ and $\msw_1 (h) \ge 6$.
If in addition $\msw_1 (f) \ge 2$ and there is a linear form $\ell = \ell (z, u) \ne 0$ such that $\ell^2 \mid g_{\msw_1 = 3}$, $\ell \mid g_{\msw_1 = 4}$, $\ell^2 \mid h_{\msw_1 = 6}$, and $\ell \mid h_{\msw_1 = 7}$, then $Y_{(3, 3, 1, 1)}$ has a non-terminal singularity.

\item The divisor $E_{(5, 4, 2, 1)}$ under the assumption that $\msw_2 (f) \ge 2$, $\msw_2 (g) \ge 6$, $\msw_2 (h) \ge 10$, and the polynomial $y^2 f_{\msw_2 = 2} + y g_{\msw_2 = 6} + h_{\msw_2 = 10}$ is not a square.
If in addition $\msw_2 (f) \ge 3$ and there is a $\lambda \in \mbC$ such that $\ell \mid f_{\msw_2 = 3}$, $\ell^2 \mid g_{\msw_2 = 6}$, $\ell \mid g_{\msw_2 = 7}$, $\ell^2 \mid h_{\msw_2 = 10}$, and $\ell \mid h_{\msw_2 = 11}$, where $\ell = z - \lambda u^2$, then $Y_{(5, 4, 2, 1)}$ has a non-terminal singularity.

\item The divisor $E_{(6, 4, 3, 1)}$ under the assumption that $\msw_3 (f) \ge 4$, $\msw_3 (g) \ge 8$ and $\msw_3 (h) \ge 12$.
If in addition $\msw_3 (g) \ge 9$ and $\msw_3 (h) \ge 14$, then $Y_{(6, 4, 3, 1)}$ has a non-terminal singularity.


\item The divisor $E_{(8, 5, 3, 1)}$ under the assumption that $\msw_3 (f) \ge 5$, $\msw_3 (g) \ge 10$ and $\msw_3 (h) \ge 15$.
If in addition there is a $\lambda \in \mbC$ such that $\ell \mid f_{\msw_3 = 5}$, $\ell^2 \mid g_{\msw_3 = 10}$, $\ell \mid g_{\msw_3 = 11}$, $\ell^3 \mid h_{\msw_3 = 15}$, $\ell^2 \mid h_{\msw_3 = 16}$, and $\ell \mid h_{\msw_3 = 17}$, where $\ell = z - \lambda u^3$, then $Y^{\pm}_{(8, 5, 3, 1)}$ has a non-terminal singularity. 

\item The divisor $E_{(9, 6, 4, 1)}$ under the assumption that $\msw_4 (f) \ge 6$, $\msw_4 (g) \ge 12$ and $\msw_4 (h) \ge 18$.
If in addition $\msw_4 (g) \ge 13$ and $\msw_4 (h) \ge 20$, then $Y_{(9, 6, 4, 1)}$ has a non-terminal singularity.

\end{enumerate}
\end{Lem}

\begin{proof}
Let $\msw$ be the blow-up weight, that is, $\msw (x, y, z, u) = (a, b, c, d)$.
In each case, it is straightforward to check that the polynomial $\phi_{\msw = d}$ is irreducible, where $d := \msw (\phi)$, and that $(a + b + c + d) - d - 1 = 1$ under the given assumption.
This implies that $E_{(a, b, c, d)}$ is a divisor of discrepancy $1$ over $\msp$ by Lemma~\ref{lem:genwblexc}.

It remains to show that $Y = Y_{(a, b, c, d)}$ has a non-terminal singularity under the additional assumption.
In each case (1)--(6), we can observe the following.
\begin{enumerate}
\item $d = 6$, and $\phi_{\msw = 6} \in (x, y)^2$, $\phi_{\msw = 7} \in (x, y)$.
\item $d = 6$, and $\phi_{\msw = 6} \in (x, \ell)^2$, $\phi_{\msw = 7} \in (x, \ell)$.
\item $d = 10$, and $\phi_{\msw = 10} \in (x, \ell)^2$, $\phi_{\msw = 11} \in (x, \ell)$.
\item $d = 12$, and $\phi_{\msw = 12} \in (x, y)^2$, $\phi_{\msw = 13} \in (x, y)$.
\item $d = 15$, and $\phi_{\msw = 15} \in (x, y, \ell)^3$, $\phi_{\msw = 16} \in (x, y, \ell)^2$, $\phi_{\msw = 17} \in (x, y, \ell)$.
\item $d = 18$, and $\phi_{\msw = 18} \in (x, y)^2$, $\phi_{\msw = 19} \in (x, y)$.
\end{enumerate}
By Lemma~\ref{lem:singwbl}, $Y$ has a non-terminal singularity, and the proof is completed.
\end{proof}

\begin{Lem} \label{lem:divdisc1cEv2}
Suppose that
\[
\phi = x^2 + y^3 + y^2 f (z, u) + y (2 \mu z^3 + \tilde{g} (z, u)) - \nu^2 z^4 + \tilde{h} (z, u)
\]
for some $\mu, \nu \in \mbC$ and $f, \tilde{g}, \tilde{h} \in \mbC [z, u]$ such that $\nu \ne 0$, and $\msw_1 (f) \ge 2$, $\msw_1 (\tilde{g}) \ge 4$, and $\msw_1 (\tilde{h}) = 6$.
Let $\chi^{\pm} \colon o \in X^{\pm} \to \msp \in X$ be the equivalence of germs defined by $x \mapsto x \mp (- \mu \nu^{-1} y z + \nu z^2)$ and $y \mapsto y - \hat{f}$, where 
\[
o \in X^{\pm} = (x^2 \pm 2 (-\mu y z + \nu z^2) x + y^3 + y \hat{g} + \hat{h} = 0) \subset \mbC^4_{x, y, z, u},
\]
with $\phi^{\pm} := \phi (x \pm (- \mu \nu^{-1} y z + \nu z^2, y - \hat{f}, y, z)$ and
\[
\begin{split}
\hat{f} &:= \frac{1}{3} (\mu \nu^{-1} z^2 + f), \\
\hat{g} &:= 3 \hat{f}^2 - 2 \hat{f} f + g, \\
\hat{h} &:= - \hat{f}^3 + \hat{f}^2 f - (2 \mu z^3 + g) \hat{f} + h.
\end{split}
\] 
Let $E^{\pm}_{(4, 2, 1, 1)}$ be the exceptional divisor of the composite
\[
\varphi^{\pm}_{(4, 2, 1, 1)} \colon Y_{(4, 2 , 1, 1)}^{\pm} \to X
\] 
of the weighted blow-up with $\wt (x, y, z, u) = (4, 2, 1, 1)$ of $X^{\pm}$ and $\chi^{\pm}$.
Then $E^{+}_{(4, 2, 1, 1)}$ and $E^{-}_{(4, 2, 1, 1)}$ are distinct divisors of discrepancy $1$  over $\msp$.
If in addition, 
\[
z \mid f_{\msw_1 = 2}, \ 
z^2 \mid \tilde{g}_{\msw_1 = 4}, \ 
z \mid \tilde{g}_{\msw_1 = 5}, \ 
z^3 \mid \tilde{h}_{\msw_1 = 6}, \ 
z^2 \mid \tilde{h}_{\msw_1 = 7}, \ 
z \mid \tilde{h}_{\msw_1 = 8},
\] 
then $Y^{\pm}_{(4, 2, 1, 1)}$ has a non-terminal singularity.
\end{Lem}

\begin{proof}
Let $\msw$ be the weight defined by $\msw (x, y, z, u) = (4, 2, 1, 1)$.
We have
\[
\phi^{\pm} = x^2 \pm 2 (-\mu y z + \nu z^2) x + y^3 + y \hat{g} + \hat{h},
\]
$\msw (\phi^{\pm}) = 6$, and
\[
\phi^{\pm}_{\msw = 6} = 2 \nu z^2 x + y^3 + y \hat{g}_{\msw = 4} + \hat{h}_{\msw = 6},
\]
which is clearly irreducible since $\nu \ne 0$.
Thus $E^{\pm} := E^{\pm}_{(4, 2, 1, 1)}$ is a divisor of discrepancy $(4+2+1+1) - 6 - 1 = 1$. 

Suppose that the additional assumptions are satisfied.
Then we have
\[
z^2 \mid \hat{g}_{\msw_1 =4}, \ 
z \mid \hat{g}_{\msw_1 = 5}, \ 
z^3 \mid \hat{h}_{\msw_1 = 6}, \ 
z^2 \mid \hat{h}_{\msw _1 = 7}, \ 
z \mid \hat{h}_{\msw_1 = 8}.
\]
It follows that
\[
\phi^{\pm}_{\msw = 6} \in (x, y, z)^3, \ 
\phi^{\pm}_{\msw = 7} \in (x, y, z)^2, \ 
\phi^{\pm}_{\msw = 8} \in (x, y, z).
\]
This shows that $Y$ has a non-terminal singularity by Lemma \ref{lem:singwbl}.
\end{proof}

\begin{Lem} \label{lem:divdisc1cEv3}
Suppose that
\[
\phi = x^2 + y^3 + y^2 f (z, u) + y g (z, u) + \tilde{h} (z, u) - \nu^2 z^4
\]
for some $\nu \in \mbC$ and some $f, g, \tilde{h} \in \mbC \{z, u\}$ such that $\nu \ne 0$, $\msw_2 (f) \ge 3$, $\msw_2 (g) \ge 6$, and $\msw_2 (\tilde{h}) \ge 9$. 
Let $\chi^{\pm} \colon \msp^{\pm} \in X^{\pm} \to \msp \in X$ be the equivalence defined by $x \mapsto x \mp \nu z^2$, where
\[
\msp^{\pm} \in X^{\pm} = (\phi^{\pm} (x, y, z, u) = 0) \subset \mbC^4_{x, y, z, u}
\]
with $\phi^{\pm} := \phi (x \pm \nu z^2, y, z, u)$.
Let $E^{\pm}_{(5, 3, 2, 1)}$ be the exceptional divisor of the composite 
\[
\varphi^{\pm}_{(5, 3, 2, 1)} \colon Y^{\pm}_{(5, 3, 2, 1)} \to X
\]
of the weighted blow-up $Y^{\pm}_{(5, 3, 2, 1)} \to X^{\pm}$ of $\msp^{\pm} \in X^{\pm}$ with $\wt (x, y, z, u) = (5, 3, 2, 1)$ and $\chi^{\pm}$.
Then $E^+_{(5, 3, 2, 1)}$ and $E^-_{(5, 3, 2, 1)}$ are distinct divisors of discrepancy $1$ over $\msp$.
If in addition, 
\[
z \mid f_{\msw_2 = 3}, \ 
z^2 \mid g_{\msw_2 = 6}, \ 
z \mid g_{\msw_2 = 7}, \ 
z^3 \mid \tilde{h}_{\msw_2 = 9}, \ 
z^2 \mid \tilde{h}_{\msw_2 = 10}, \ 
z \mid \tilde{h}_{\msw_2 = 11}, 
\]
then $Y^{\pm}_{(5, 3, 2, 1)}$ has a non-terminal singularity. 
\end{Lem}

\begin{proof}
Let $\msw$ be the weight defined by $\msw (x, y, z, u) = (5, 3, 2, 1)$.
We have
\[
\phi^{\pm} = x^2 \pm 2 \nu x z^2 + y^3 + y^2 f(z, u) + y g (z, u) + \tilde{h} (z, u).
\]
It is easy to check that $\msw (\phi^{\pm}) = 9$ and the polynomial
\[
\phi^{\pm}_{\msw = 9} = \pm 2 \nu x z^2 + y^3 + y^2 f_{\msw_2 = 3} + y g_{\msw_2 = 9} + h_{\msw_2 = 9}
\]
is irreducible since $\nu \ne 0$.
Hence $E^{\pm} = E^{\pm}_{(5, 3, 2, 1)}$ is a divisor of discrepancy $(5 + 3 + 2 + 1) - 9 - 1 = 1$ over $\msp \in X$.

It is also straightforward to check that, under the addition assumptions, we have
\[
\phi^{\pm}_{\msw = 9} \in (x, y, z)^3, \ 
\phi^{\pm}_{\msw = 10} \in (x, y, z)^2, \ 
\phi^{\pm}_{\msw = 11} \in (x, y, z).
\]
This shows that $Y^{\pm}_{(5, 3, 2, 1)}$ has a non-terminal singularity.
\end{proof}

\begin{Lem} \label{lem:divdisc1cEv4}
Suppose that $\msw_2 (f) = 3$, $\msw_2 (g) \ge 7$, $\msw_2 (h) = 10$, and $z^5 \in h$.
We set $e = e (z, u) := f_{\msw_2 = 3} (z, u) \ne 0$.
Let $\chi \colon \tilde{\msp} \in \tilde{X} \to \msp \in X$ be the equivalence of germs defined by $y \mapsto y + e$, where
\[
\tilde{\msp} \in \tilde{X} := (\tilde{\phi} = 0) \subset \mbC^4_{x, y, z, u}
\]
with $\tilde{\phi} := \phi (w, y - e, z, u)$.
Let $\tilde{E}_{(5, 4, 2, 1)}$ be the exceptional divisor of the composite $\tilde{\varphi}_{(5, 4, 2, 1)} \colon\tilde{Y}_{(5, 4, 2, 1)} \to X$ of the weighted blow-up of $\tilde{\msp} \in \tilde{X}$ with $\wt (x, y, z, u) = (5, 4, 2, 1)$ and $\chi \colon\tilde{X} \to X$.
Then $\tilde{E}_{(5, 4, 2, 1)}$ is a divisor of discrepancy $1$ over $\msp \in X$ which is different from $E_{(5, 4, 2, 1)}$.
If in addition there is a $\lambda \in \mbC$ such that 
\[
\ell \mid f_{\msw_2 = 3}, \ 
\ell \mid g_{\msw_2 = 7}, \ 
\ell^2 \mid h_{\msw_2 = 10}, \ 
\ell \mid h_{\msw_2 = 11},
\] 
where $\ell = z - \lambda u^2$, then $\tilde{Y}_{(5, 4, 2, 1)}$ has a non-terminal singularity.
\end{Lem}

\begin{proof}
We have
\[
\tilde{\phi} = x^2 + y^3 + y^2 \tilde{f} (z, u) + y \tilde{g} (z, u) + \tilde{h} (z, u),
\]  
where
\[
\begin{split}
\tilde{f} &= - 3 e + f, \\
\tilde{g} &= 3 e^2 - 2 e f + g, \\
\tilde{h} &= - e^3 + e^2 f - e g + h.
\end{split}
\]
We have $\msw_2 (\tilde{f}) = 3$, $\msw_2 (\tilde{g}) = 6$, and $\msw_2 (\tilde{h}) \ge 10$.
Note that $u \mid e = f_{\msw_2 = 3}$ and $z^5 \in h_{\msw_2 = 10}$.
It follows that the polynomial
\[
y^2 \tilde{f}_{\msw_2 = 2} + y \tilde{g}_{\msw_2 = 6} + \tilde{h}_{\msw_2 = 10} = y e^2 + e^2 f_{\msw_2 = 4} - e g_{\msw_2 = 7} + h_{\msw_2 = 10}
\]
is not a square since it contains $z^5$ with non-zero coefficient.
By Lemma \ref{lem:divdisc1cEv1}~(3), $\tilde{E}_{(5, 4, 2, 1)}$ is a divisor of discrepancy $1$ over $\tilde{\msp} \in \tilde{X}$ and hence over $\msp \in X$.
The divisors $\tilde{E}_{(5, 4, 2, 1)}$ and $E_{(5, 4, 2, 1)}$ define distinct valuations since the weight of $\chi (x) = x + e$ is $3$ while the weight of $x$ is $5$ with respect to $\wt (x, y, z, u) = (5, 4, 2, 1)$.

Suppose that the additional assumptions are satisfied.
Then it is straightforward to check that 
\[
\ell \mid \tilde{f}_{\msw_2 = 3}, \  
\ell^2 \mid \tilde{g}_{\msw_2 = 6}, \ 
\ell \mid \tilde{g}_{\msw_2 = 7}, \  
\ell^2 \mid \tilde{h}_{\msw_2 = 10}, \  
\ell \mid \tilde{h}_{\msw_2 = 11}
\]
hold.
Thus, by \ref{lem:divdisc1cEv1}~(3), $\tilde{Y}_{(5, 4, 2, 1)}$ has a non-terminal singularity.
\end{proof}


\subsection{Divisorial contractions to $cD/2$ points}

Let $\msp \in X$ be a singularity of type $cD/2$-$2$.
Then there is an equivalence $\msp \in X \cong o \in \overline{X}$ of germs, where
\begin{equation} \label{eq:cD2std}
o \in \overline{X} = (u^2 + y^2 z + \lambda y x^{2 a + 1} + g (x^2,z) = 0)/\mbZ_2 (1_x, 1_y, 0_z, 1_u)
\end{equation}
for some $\lambda \in \mbC$, $a \ge 1$, and $g (x^2, z) \in (x^4, x^2 z^2, z^3) \mbC \{x, z\}$.
We say that the germ $o \in \overline{X}$ is {\it standard} or is a {\it standard model} of $X$.
Let $o \in \overline{X}$ be a standard germ of a $cD/2$-$2$ singularity.
By convention, we set $a = +\infty$ when $\lambda = 0$, henceforth we always assume $\lambda \ne 0$.
Throughout this subsection, let $\msw$ and $\msw'$ be weights on $x, u$ defined by $\msw (x, z) = \frac{1}{2} (1, 4)$ and $\msw' (x, z) = \frac{1}{2} (1, 2)$, respectively.
We define $a (\overline{X}) := a$ and also define $b (\overline{X}):= \msw (g (x^2, z))$ and $b' (\overline{X}) := \msw' (g (x^2, z))$, respectively.
We define
\[
A (\overline{X}) := \{\, k \in \mbZ \mid 1 \le k \le \min \{2 a(\overline{X}) -3, b(\overline{X})-2\}, k: \text{odd} \,\}
\]
and $l (\overline{X}) = \max A (\overline{X})$ if $A (\overline{X}) \ne \emptyset$.
For a an odd integer $k$, we define $\msw_k$ be the weight on $x, y, z, u$ defined by $\msw_k (x, y, z, u) = \frac{1}{2} (1, k, 4, k + 2)$.
By \cite[Proposition~5.4]{Hayakawa00} the morphism $\varphi_{l(\overline{X})} \colon Y \to X$ which is induced from the weighted blow-up $\bar{\varphi}_{l(\overline{X})} \colon Y \to \overline{X}$ with weight $\msw_{l (\overline{X})}$ is a divisorial contraction of discrepancy $\frac{1}{2}$.

\begin{Lem} \label{lem:divcd/2pre}
Let $\msp \in X$ be a $cD/2$-$2$ singularity.
Assume that there is a standard model $o \in \overline{X}$ of $\msp \in X$ satisfying the following conditions.
\begin{enumerate}
\item If $b'(\overline{X})$ is odd $($resp.\ even$)$, then $2 a(\overline{X}) - 1 > b'(\overline{X})$ and $b(\overline{X}) > b'(\overline{X})+1$ $($resp.\ $2 a(\overline{X}) - 2 > b'(\overline{X})$ and $b(\overline{X}) > b'(\overline{X})+2$$)$.
\item $z^{b'(\overline{X})} \in g (x^2,z)$.
\end{enumerate}
Then, the weighted blow-up $\varphi_{l(\overline{X})} \colon Y \to X$ with weight $\msw_{l(\overline{X})}$ is the unique divisorial contraction of discrepancy $\frac{1}{2}$ centered at $\msp$, and $\msp \in X$ does not admit a divisorial contraction of discrepancy $1$.
\end{Lem}

\begin{proof}
We set $a = a (\overline{X})$, $b = b (\overline{X})$ and $b' = b' (\overline{X})$.
 
Suppose $b'$ is odd (resp.\ even).
Then, by (2), we have $z^{b'} \in g_{\msw' = b'} (x,z)$, which implies that $z^{b'} \in g_{\msw' = b'} (x, z)$ (resp.\ $z^{b'} \in g_{\msw' = b'} (x,z) z$) is not a square.
By (1), we have $b' < \min \{2 a -1, b-1\}$ (resp.\ $b' < \min \{2 a - 2,b-2\}$).
Thus \cite[Propositions 5.14 and 5.26]{Hayakawa00} shows that there is exactly one divisorial contraction of discrepancy $\frac{1}{2}$ centered at $\msp \in X$.

We will show that $\overline{X}$ does not admit a divisorial contraction of discrepancy $1$.
In our case, the invariants $w,w'_1,w'_2$ and $w'$ of $\overline{X}$ introduced in \cite[3.1]{Hayakawa05a} can be understood in terms of $a,b,b'$:
\[
w = \min \{ 2 a , b' \}, \ 
w'_1 = 4 a + 1, \ 
w'_2 = b'.
\]
Note that the lowest degree part of $g (x^2, z)$ is $z^{b'}$ since $\msw'$-weight of $g (x^2, z)$ is $b'$ and $z^{b'} \in g$.
It follows that the lowest degree part of $\frac{1}{4} (\lambda x^{2 a + 1})^2 - g (x,z) z$ is $- \alpha z^{b'+1}$ for some non-zero $\alpha \in \mbC$. 
In particular $\overline{X}$ satisfies the condition $(\dagger)$ in \cite[3.1]{Hayakawa05a}.
Now we define complex nunbers $a_{i, j}$ and $b_i$ as follows: $g (x^2, z) = \sum_{i,j} a_{i,j} x^{2 i} z^j$ and $h (x,z) := \lambda x^{2 a +1} = \sum_i b_i x^{2 i + 1}$.
Clearly $b_a = \lambda$ and $b_i = 0$ for $i \ne a$.

Suppose that $w' = b'$ is even.
We claim that $a_{i,j} = 0$ for $i, j$ with $i + j \le \frac{1}{2} b' + 1$.
Indeed, if $a_{i,j} \ne 0$ for some $i,j$ satisfying the above condition, then $x^{2 i} z^j \in g$ and the $\msw'$-weight of $g$ must be less than or equal to $i + j \le \frac{1}{2} b' + 1$.
This is a contradiction since the $\msw$-weight of $g$ is $b'$ and $b' > \frac{1}{2} b' + 1$.
It is clear that either (a) $b' \in 4 \mbZ, b_{b'/4} = 0$ or (b) $b' \in 4 \mbZ + 2, a_{b'/2,0} = 0$ is satisfied.
Moreover, none of the following holds:
\begin{enumerate}
\item[(i)] $w' \in 4 \mbZ$ and $a_{w'/2,1}^2 - 4 a_{w'/2+1,0} a_{w'/2-1,2} \ne 0$.
\item[(ii)] $w' \in 4 \mbZ + 2$ and $a_{w'/2,1} b_{(w'+2)/4}^2 + a_{w'/2 + 1,0}^2 \ne 0$.
\end{enumerate}
This is because all the $a_{i,j}$'s appearing the above (i) and (ii) are zero.
Thus by \cite[Proposition 3.14]{Hayakawa05a}, $\overline{X}$ and hence $X$ do not admit a divisorial contraction of discrepancy $1$ (see also the explanation given in just before the paragraph 3.11 in \cite{Hayakawa05a}).

Suppose that $w' = b'$ is odd.
Clearly we have $w'_1 = 4 a + 1 > b' = w'_2$ and $\overline{X}$ does not satisfy the condition: $w' \in 4 \mbZ + 3$ and $a_{(w'+1)/2,0} \ne 0$.
Thus by \cite[Proposition 3.17]{Hayakawa05a}, $\overline{X}$ and hence $X$ do not admit a divisorial contraction of discrepancy $1$.
\end{proof}

Let $\msp \in X$ be a $cD/2$-$2$ singularity.
We introduce the following condition on $X$ that assures the uniqueness of divisorial contraction to $\msp \in X$.

\begin{Cond} \label{cond:cd2}
There is a standard model $o \in \tilde{X}$ of $\msp \in X$ with the following properties:
\begin{enumerate}
\item $A (\tilde{X}) \ne \emptyset$ and $l (\tilde{X}) \le 7$.
\item For the weighted blowup $\varphi_{l (\tilde{X})} \colon Y \to X$ with weight $\msw_{l (\tilde{X})}$, $Y$ has only terminal quotient singularity (along the exceptional divisor).
\end{enumerate} 
\end{Cond}

\begin{Lem} \label{lem:cD2eqs}
Let $o \in \overline{X}$ be a standard germ of a $cD/2$-$2$ singularity such that $A (\overline{X}) \ne \emptyset$, and let $\bar{\varphi}_{l (\overline{X})} \colon Y \to \overline{X}$ be the weighted blow-up with weight $\msw_{l (\overline{X})}$ with exceptional divisor $E_{l (\overline{X})}$.
Suppose that $Y$ has only terminal quotient singularities.
Then the following are satisfied.
\begin{enumerate}
\item $b' (\overline{X}) = \frac{l (\overline{X}) + 3}{2}$.
\item $(2 a (\overline{X}) - 2) - b' (\overline{X}) \ge \frac{l (\overline{X}) - 1}{2}$.
\item $b (\overline{X}) - (b' (\overline{X}) + 2) \ge \frac{l (\overline{X}) - 3}{2}$.
\item $(E^3_{l (\overline{X})}) = \frac{2}{l (\overline{X})}$.
\end{enumerate}
\end{Lem}

\begin{proof}
Let $o \in \overline{X}$ be as in \eqref{eq:cD2std}.
Note that $a = a (\overline{X})$ and we set $l = l (\overline{X})$, $b = b (\overline{X})$ and $b' = b' (\overline{X})$.
We also set $E := E_{l (\overline{X})}$.
We have an isomorphism
\[
E \cong (u^2 + y^2 z + \lambda \delta_{2a-3,l} y x^{2a+1} + g_{\msw = l+2} (x^2,z) = 0) \subset \mbP (1_x, l_y, 4_z, (l+2)_u),
\]
where $\delta_{2a-3,l}$ is the Kronecker delta.
Since there is no monomial of the form $z^m x$ in $g_{\msw = l+2} (x^2,z)$, we have $z^{(l+3)/2} \in g (x^2,z)$ because otherwise $Y$  has a non-quotient singularity at $(0\!:\!0\!:\!1\!:\!0) \in E$. 
This implies $b \le l+3$ and $b' \le (l+3)/2$.

We claim $b' = (l+3)/2$.
Assume to the contrary that $b' < (l+3)/2$.
Then $b' \le (l+1)/2$ and there is a monomial $x^{2 i} z^j \in g (x^2,z)$ with $i + j \le (l+1)/2$.
Since $l \le b - 2$, we have $i + 2 j \ge b \ge l+2$. 
Thus $l + 2 \le i + 2 j = (i+j) + j \le \frac{l+1}{2} + j$, hence $j \ge (l+3)/2$.
This implies $i + j \ge \frac{l+3}{2}$.
Thus is a contradiction, and (1) is proved.

If $2 a - 3 < b-2$ (resp.\ $2 a - 3 = b-2$), then $l = 2 a - 3$ and $b = l+3$ (resp.\ $l = 2 a -3$ and $b = l +2$), and if $2 a - 3 > b-2$, then either $b = l+2$ or $l + 3$.
In any of the above cases, we have $2 a - 3 \ge l$, $b \ge l+2$ and $b' = \frac{l+3}{2}$.
Thus we have
\[
(2 a -2) - b' \ge (l+1) - \frac{l+3}{2} = \frac{l - 1}{2} > 0
\]
and
\[
b - (b'+2) \ge (l+2) - \frac{l+3}{2} -2 = \frac{l - 3}{2}.
\]
This proves (2) and (3).
Finally we have
\[
(E^3) = \frac{2^2 \cdot 2 (l + 2)}{1 \cdot l \cdot 4 \cdot (l+2)} = \frac{2}{l},
\]
and (4) is proved.
\end{proof}

\begin{Lem} \label{lem:cd/2pre2}
Suppose that a $cD/2$-$2$ singularity $\msp \in X$ satisfies \emph{Condition \ref{cond:cd2}}.
Then $l (\overline{X})$ and $b' (\overline{X})$ do not depend on the choice of a standard model $o \in \overline{X}$ of $\msp \in X$.
Moreover, for any standard model $o \in \overline{X}$ of $\msp \in X$, $b' (\overline{X}) = (l(\overline{X})+3)/2$ and the assumptions in \emph{Lemma \ref{lem:divcd/2pre}} are satisfied.
\end{Lem}

\begin{proof}
Let $o \in \overline{X}$ be be a standard model of $\msp \in X$ which is of the form \eqref{eq:cD2std}. 
The equality $b' (\overline{X}) = (l (\overline{X}) + 3)/2$ follows from (1) of Lemma~\ref{lem:cD2eqs}.
By the proof of Lemma \ref{lem:cD2eqs}, we have $z^{b' (\overline{X})} \in g (x^2, z)$.
Since $g (x^2, z) \in (x^4, x^2 z^2, z^3) \mbC \{x, z\}$, we have $b' (\overline{X}) \ge 3$, and hence $l (\overline{X}) \ge 3$.
By (2) and (3) of Lemma~\ref{lem:cD2eqs}, we see that the assumptions of Lemma~\ref{lem:divcd/2pre} are satisfied since $l (\overline{X}) \ge 3$.
By Lemma~\ref{lem:divcd/2pre}, the divisorial contraction $\varphi_{l (\overline{X})} \colon Y \to X$ is a unique divisorial contraction of discrepancy $1/2$ over $\msp \in X$ and we have $(E^3) = l (\overline{X})/2$ for its exceptional divisor.
This shows that $l (\overline{X})$ does not depend on the choice of $o \in \overline{X}$, and so is $b' (\overline{X})$.
This completes the proof.
\end{proof}

\begin{Lem} \label{lem:cd/2o3}
If a germ $\msp \in X$ of a $cD/2$ point admits a divisorial contraction $\varphi \colon Y \to X$ of discrepancy $\frac{e}{2}$ with $e \ne 1$ which is of type $o3$, then there is a standard model 
\[
o \in \overline{X} = (f (x, y, z, u) = 0)/\mbZ_2 (1_x, 1_y, 0_z, 1_u)
\]
of $\msp \in X$, where $f$ satisfies one of the following:
\begin{enumerate}
\item[(a)]
$f = u^2 + y^2 z + \lambda y x^{2 a +1} + p (x^2,z) + x^2 q (x^2,z)^2$ and there is a positive integer $r$ such that
\begin{enumerate}
\item[(1)] $e \mid r + 1$ with $e > 1$, and $e, r$ are odd.
\item[(2)] $p$ and $q$ has weight $r + 1$ and $(r-e)/2$ respectively with respect to the weight $\wt (x,z) = \frac{1}{2} (e,2)$.
\item[(3)] The polynomial $u x q_{(r-e)/2} (x^2,z) + y^2 z + p_{r+1} (x^2,z)$ is irreducible, where $p_{r+1} (x^2,z)$ and $q_{(r-e)/2} (x^2,z)$ are the weight $r+1$ and $(r-e)/2$ parts of $p (x^2,z)$ and $q (x^2,z)$ respectively with respect to the weight $\wt (x,z) = \frac{1}{2} (e,2)$.
\end{enumerate}
\item[(b)]
$f = u^2 + y^2 z + y x^{(r+2)/e} + p (x^2,z) - q (x^2,z)^2 x^2 z - q (x^2,z) x^{(r+2)/e + 1}$ and there is a positive integer $r$ such that
\begin{enumerate}
\item[(1)] $e \mid r+2$ with $e > 1$, $e \ne r+2$ and $(r+2)/e$ is odd.
\item[(2)] $p$ and $q$ have weight $r+1$ and $(r-e)/2$ respectively with respect to the weight $\wt (x,z) = \frac{1}{2} (e,2)$.
\end{enumerate}
\end{enumerate}
\end{Lem}

\begin{proof}
We need to consider the two cases: (ii)-(a) and (ii)-(b) of \cite[Theorem 1.2]{Kawakita05}.

Suppose we are in case (ii)-(a) of \cite[Theorem 1.2]{Kawakita05} and let the discrepancy of $\varphi$ is $e/2$ (with $e \ne 1$).
Then $\msp \in X$ can be identified with the germ
\[
o \in (f := u^2 + u x q (x^2,z) + y^2 z + \lambda y x^{2 a +1} + p (x^2,z) = 0)/\mbZ_2 (1_x, 1_y, 0_z, 1_u)
\]
for some $\lambda \in \mbC$ and $a \ge 1$.
Moreover there is an positive odd integer $r$ such that the conditions (a-1) and (a-2) in the statement and the condition that the weigh $r+1$ part of $f$ with respect to the weight $\bar{\msw}$ defined by $\bar{\msw} (x, y, z, u) = \frac{1}{2} (e, r, 2, r + 2)$ is irreducible.
Since 
\[
f_{\bar{\msw} = r+1} = u x q_{(r-e)/2} (x^2,z) + y^2 z + p_{r+1} (x^2,z),
\] 
the condition (a-3) is satisfied.
Replacing $u \mapsto u - \frac{1}{2} x q$ and then rescaling $q$, we may assume
\[
f = u^2 + y^2 z + \lambda y x^{2 a + 1} + p (x^2,z) + x^2 q (x^2,z)^2.
\]
Thus we are in case (a) of the statement.

Suppose we are in case (ii)-(b) of \cite[Theorem 1.2]{Kawakita05} and let the discrepancy of $\varphi$ is $\frac{e}{2}$ (with $e \ne 1$).
Then $\msp \in X$ can be identifies with the germ
\[
\begin{split}
o \in (u^2 + y v + p (x^2,z) = y z +x^{(r+2)/e} & + q (x^2,z) x z + v = 0) \\
& \subset \mbC^5_{x, y, z, u, v}/\mbZ_2 (1,1,0,1,1)
\end{split}
\]
for some $r \ge 1$, $p (x^2, z)$ and $q (x^2, z)$ satisfying the conditions (b-1) and (b-2) in the statement.
By replacing $y \mapsto - y$, $z \mapsto - z$ and eliminating $v$, $\msp \in X$ is identified with 
\[
o \in (f := u^2 + y^2 z + y x^{(r+2)/e} + q (x^2,z) x y z + p (x^2,z) = 0)/\mbZ_2 (1_x, 1_y, 0_z, 1_u).
\]
By replacing $y \mapsto y - \frac{1}{2} q x$ and then re-scaling $q$, we have
\[
f = u^2 + y^2 z + y x^{(r+2)/2} + p (x^2,z) - q (x^2,z)^2 x^2 z - q (x^2,z) x^{(r+2)/e + 1}.
\]
Therefore we are in case (b) and the proof is completed.
\end{proof}

\begin{Lem} \label{lem:cd/2e1}
If a germ $\msp \in X$ of a $cD/2$ point admits a divisorial extraction $\varphi \colon Y \to X$ of type $e1$ and of discrepancy $2$, then there is a standard model $o \in \overline{X}$
\[
o \in \overline{X} = (u^2 + y^2 z + \lambda y x^{2 a + 1} + g (x^2,z) = 0)/\mbZ_2 (1_x, 1_y, 0_z, 1_u)
\]
of $\msp \in X$, where $\lambda \in \mbC$, $a \ge 1$ and $z^j \notin g (x,z)$ for $j \le 5$.
\end{Lem}

\begin{proof}
By \cite{Kawakita12}, $X$ is identified with the germ
\[
(u^2 + v z + p (x,y,z) = y^2 + q (x,z,u) + v = 0)/\mbZ_2 (1_x, 1_y, 0_z, 1_u, 0_v).
\]
Replacing $z \mapsto - z$ and eliminating $v$, we have an identification 
\[
X \cong (f = 0)/\mbZ_2 (1_x, 1_y, 0_z, 1_u),
\] 
where $f = u^2 + y^2 z + q (x, z, u) z + p (x, y, z)$.
Moreover there is an integer $r \ge 7$, $r \equiv \pm 1 \pmod{8}$ such that $p$ has weight at least $r + 1$ and $q$ is weighted homogeneous of weight $r -1$ with respect to the weight $\wt (x,z) = \frac{1}{2} (4,2)$.
Consider the weight $\bar{\msw}$ defined by $\bar{\msw} (x, y, z, u) = \frac{1}{2} (4, r-1, 2, r+1)$.
Then
\[
f = \overbrace{(y^2 z + q (x, z, u) z)}^{\bar{\msw} = r} + \overbrace{(u^2 + p (x, y, z))}^{\bar{\msw} \ge r+1}.
\]
We write $p = y^2 p' + p''$, where $p''$ does not contain monomial divisible by $y^2$, and let $h$ be the lowest $\bar{\msw}$-weight part of $p'$.
Replacing $z \mapsto z - h$, we can eliminate $h$ and $f$ is still in the above displayed form.
Repeating this replacement, we may assume that $p$ does not contain a monomial divisible by $y^2$.
 
Thus we may assume $p (x,y,z) = y p_1 (x,z) + p_2 (x,z)$, where $p_1$ and $p_2$ have weight at least $\frac{r+3}{2}$ and $r+1$ respectively.
We can write $q (x,z,u) = u q_1 (x,z) + q_2 (x,z)$, where $q_1$ and $q_2$ are homogeneous of weights $(r-3)/2$ and $r-1$ respectively.
Replacing $u \mapsto u - \frac{1}{2} q_1 z$ then replacing $\frac{1}{2} q_1$ with $q_1$, we may assume
\[
f = u^2 + y^2 z + y p_1 (x,z) + q_2 (x,z) z + p_2 (x,z) + q_1 (x,z)^2 z^2.
\]
We write $p_1 (x,z) = z p'_1 (x,z) + \lambda x^{2 a +1}$, where $\lambda \in \mbC$, $a$ is a positive integer and $p'_1 (x,z)$ has weight at least $(r+1)/2$.
If $\lambda \ne 0$, then $a \ge (r-1)/8$ since $p_1$ has weight at least $(r+3)/2$.
Replacing $y \mapsto y - p'_1/2$ and then replacing $p'_1/2$ with $p'_1$, we may assume
\[
f = u^2 + z y^2 + \lambda y x^{2a+1} + p_2 + q_2 z + q_1^2 z^2 + {p'_1}^2 z - \lambda p'_1 x^{2 a +1}.
\]
We set $p = p_2 + q_2 z + q_1^2 z^2 + {p'_1}^2 z - \lambda p'_1 x^{2 a +1}$.
If $z^j \in p$, then $j \ge r-1 \ge 6$.
This completes the proof.
\end{proof}

\begin{Prop} \label{prop:cD2dcuni}
Suppose that a $cD/2$-$2$ singularity $\msp \in X$ satisfies \emph{Condition \ref{cond:cd2}}.
Then $\varphi_{l(\tilde{X})} \colon Y \to X$ is the unique divisorial extraction of $\msp \in X$.
\end{Prop}

\begin{proof}
By the classification due to Kawakita \cite{Kawakita05}, we need to show that $X$ does not admit a divisorial contraction of discrepancy $\frac{e}{2}$ with $e \ne 1$ and of type $o3$ or of discrepancy $2$ of type $e1$.

Suppose that $X$ admits a divisorial extraction of type $o3$ or of type $e1$.
Then there is an equivalence $\msp \in X \cong o \in \overline{X}$, where $o \in \overline{X}$ is a standard germ whose description is given in (a), (b) in Lemma \ref{lem:cd/2o3} or in Lemma \ref{lem:cd/2e1}.
We set $l = l (\overline{X})$, $a = a(\overline{X})$, $b = b (\overline{X})$ and $b' = b' (\overline{X})$.

Suppose that $\overline{X}$ is of the form given in (a) of Lemma \ref{lem:cd/2o3}.
Then $z^{b'} \in p (x^2,z)$ by (2) of Lemma \ref{lem:divcd/2pre}.
By considering the weight with respect to $\wt (x,z) = \frac{1}{2} (e,2)$, we have $b' \ge r+1$.
Hence $b \ge b' + 3 \ge r + 4$.
We will show that both $p_{(r-e)/2} (x^2,z)$ and $q_{r+1} (x^2,z)$ are divisible by $z$.
Suppose $z \nmid p_{r+1} (x^2,z)$.
Then $x^{2 (r+1)/e} \in p (x^2,z)$. 
It follows that $b \le (r+1)/e$ and combining this with $b \ge r+4$ we have $0 < (e-1)r \le 1 -4 e < 0$.
This is a contradiction.
Suppose $z \nmid q_{(r-e)/2} (x^2,z)$.
Then $x^{(r-e)/e} \in q (x^2,z)$.
In particular $e \mid r-e$, which is impossible since $e \mid r+1$.
Thus both $p_{r+1} (x^2,z)$ and $q_{(r-e)/2} (x^2,z)$ are divisible by $z$.
But then the polynomial $u x q_{(r-e)/2} (x^2,z) + y^2 z + p_{r+1} (x^2,z)$ is reducible, contradicting (a-3) of Lemma \ref{lem:cd/2o3}.

Suppose that $\overline{X}$ is of the form given in (b) of Lemma \ref{lem:cd/2o3}.
Then $a = a (\overline{X}) = (r+2)/e$.
We have $z^{b'} \in g$, where $b' = b' (\overline{X})$, hence $z^{b'} \in p (x^2,z)$.
By considering the weight, this implies $b' \ge r+1$. 
But then, since $e \ge 3$, we have
\[
(2 a - 2) - b' \le \left(\frac{r+2}{3} - 3\right) - (r+1) \le \frac{r+2}{3} - r -4 = - \frac{2 r + 10}{3} < 0.
\]
This is a contradiction.

Finally, suppose that $\overline{X}$ is of the form given in Lemma \ref{lem:cd/2e1}.
Then $b' \ge 6$ since $z^{b'} \in g (x^2,z)$ and $z^j \notin g (x^2,z)$ for $j \le 5$.
This is a contradiction since $b' = (l+3)/2$ and $l \le 7$.
\end{proof}

\begin{Cor} \label{cor:cD2dcuni}
Let $\msp \in X$ be a terminal singularity of type $cD/2$.
Suppose that there is a divisorial contraction $\varphi \colon Y \to X$ of discrepancy $1/2$ centered at $\msp \in X$ with the following property$:$ $(E^3) = 2/l$ for some odd integer $3 \le l \le 7$, where $E$ is the $\varphi$-exceptional divisor, and the set of non-Gorenstein singular points of $Y$ along $E$ consists of $2$ terminal quotients singular points of type $\frac{1}{4} (1, 1, 3)$ and $\frac{1}{l} (1, 2, -2)$.
Then $\varphi$ is the unique divisorial extraction of $\msp \in X$.
\end{Cor}

\begin{proof}
By the classification of divisorial contractions of discrepancy $1/2$ to a $cD/2$ point, the pair $(\tilde{\varphi}, \tilde{E})$ of a divisorial contraction $\tilde{\varphi} \colon \tilde{Y} \to \tilde{X}$ to a $cD/2$ singularity $\tilde{\msp} \in \tilde{X}$ such that $(\tilde{E}^3) = 2/\tilde{l}$ for some odd positive integer $\tilde{l}$ is one of the following:
\begin{enumerate}
\item $\tilde{\msp} \in \tilde{X}$ is of type $cD/2$-$1$, and ($\tilde{\varphi}, \tilde{E}) = (\pi_3, E_3)$, where $(\pi_3, E_3)$ is the one given in \cite[Proposition~4.12]{Hayakawa00}, and it is proved in loc.\ cit.\ that $\tilde{X}$ has a non-quotient singularity along $\tilde{E}$.
\item $\tilde{\msp} \in \tilde{X}$ is of type $cD/2$-$2$, and $(\tilde{\varphi}, \tilde{E}) = (\pi_l, E_l)$, where $(\pi_l, E_l)$ is the one given in \cite[Proposition~5.4]{Hayakawa00}. 
Note that $\tilde{\varphi}$ coincides with the divisorial contraction $\varphi_{l (\overline{X})}$ to $\overline{X}$, where $\overline{X}$ is a standard model of $X$.
\item $\tilde{\msp} \in \tilde{X}$ is of type $cD/2$-$2$, and $(\tilde{\varphi}, \tilde{E}) = (\pi'_{\pm}, E'_{\pm})$ or $(\pi', E')$, where $(\pi'_{\pm}, E'_{\pm})$ (resp.\ $(\pi', E')$) is the one given in \cite[Proposition~5.18]{Hayakawa00} (resp.\ \cite[Proposition~5.25]{Hayakawa00}), and in either case it is proved in loc.\ cit. that $(\tilde{E}^3) = 4/m$ for some even integer $m > 0$ and $\tilde{X}$ has a terminal quotient singular of index $2 m$ along $\tilde{E}$.
\end{enumerate}

Let $\varphi \colon Y \to X$ be as in the statement and let $E$ be its exceptional divisor.
By the assumption, this corresponds to (2).
By the assumption on $l$ and the singularity of $Y$ along its exceptional divisor, a standard model $\overline{X}$ of $X$ clearly satisfy Condition~\ref{cond:cd2}.
The assertion follows from Proposition~\ref{prop:cD2dcuni}.
\end{proof}

\section{Exclusions for some Fano $3$-fold hypersurfaces} \label{sec:exclmethod}

Let $X$ be a $3$-fold weighted hypersurface
\[
X = X_d \subset \mbP (a_0, a_1, a_2, a_3, a_4)
\]
defined by a quasi-homogeneous polynomial $F = F (x, y, z, t, w)$ of degree $d$, where $x, y, z, t, w$ are homogeneous polynomials of weights $a_0, a_1, a_2, a_3, a_4$, respectively.
We denote by $A$ the Weil divisor class on $X$ such that $\mcO_X (A) \cong \mcO_X (1)$.
Throughout this section, we assume that $X$ is a normal $\mbQ$-factorial variety with only terminal singularities, the set $X \cap \Sing \mbP (a_0, \dots, a_4)$ is of codimension at least $2$ in $X$, and $\iota_X := \sum_{i=0}^4 a_i - d > 0$.
Note that we have $-K_X \sim \iota_X A$ by adjunction.

\begin{Def} \label{def:isolset}
Let $V$ be a subvariety of a weighted projective space $\mbP (a_0, \dots, a_n)$ with homogeneous coordinates $x_0, \dots, x_n$ and let $\msp \in V$ be a point.
We say that a finite set $\{g_1, \dots, g_N\}$ of quasi-homogeneous polynomials in variables $x_0, \dots, x_n$ \textit{isolates} $\msp$ or is a $\msp$-\textit{isolating set} if $\msp$ is an isolated component of $(g_1 = \cdots = g_N = 0) \cap V$.
\end{Def}

\begin{Lem}[{\cite[Lemma 5.6.4]{CPR}}]
Under the notation and assumption as in \emph{Definition~\ref{def:isolset}}, if $\{g_1, \dots, g_N\}$ isolates $\msp$, then $m A \sim_{\mbQ} - (m/\iota_X) K_X$ isolates $\msp$, where $m = \max \{\deg g_i\}$.
\end{Lem}

\subsection{Curves and smooth points}

\begin{Lem} \label{lem:exclcurves}
Assume that $(-K_X^3) \le \iota_X$, and that $X$ is quasi-smooth along $X \cap \Sing \mbP (a_0, \dots, a_4)$.
Then no curve on $X$ is a maximal center.
\end{Lem}

\begin{proof}
Let $\Gamma \subset X$ be an irreducible and reduced curve.
We set $\Sigma := X \cap \Sing \mbP (a_0, \dots, a_4)$.
Note that the divisor $A$ is a Cartier divisor outside $\Sigma$.

Suppose that $\Gamma \cap \Sigma = \emptyset$.
Then $(-K_X \cdot \Gamma) = \iota_X (A \cdot \Gamma) \ge \iota_X$, and thus we have $(-K_X \cdot \Gamma) \ge (-K_X)^3$.
It follows from \cite[Lemma 2.9]{OkII} that $\Gamma$ is not a maximal center.

Suppose that $\Gamma \cap \Sigma \ne \emptyset$.
By the assumption that $X$ is quasi-smooth along $\Sigma$, any point in $\Sigma$ is a terminal quotient singularity of $X$.
Thus $\Gamma$ is not a maximal center since a curve passing through a $3$-dimensional terminal quotient singular point cannot the center of a divisorial contraction (\cite{Kawamata}).
\end{proof}

\begin{Lem} \label{lem:exclsmpts}
Let $X = X_{2 b} \subset \mbP (a_0, \dots, a_3, b)$ be a Mori-Fano weighted hypersurface defined by a quasi-homogeneous polynomial $F = F (x, y, z, t, w)$ of degree $2 b$, where $a_0 \le a_1 \le a_2 \le a_3$.
If $w^2 \in F$ and the inequality
\[
\iota_X^2 \le 2 a_0 a_1
\]
holds, then no smooth point of $X$ is a maximal center.
\end{Lem}

\begin{proof}
Let $\msp = (\alpha_0\!:\!\alpha_1\!:\!\alpha_2\!:\!\alpha_3\!:\!\beta) \in X$ be a smooth point, and let $\msq := \pi (\msp) = (\alpha_0\!:\!\alpha_1\!:\!\alpha_2\!:\!\alpha_3)$ be the image of $\msp$ under the double cover $\pi \colon X \to \mbP (a_0, \dots, a_3) =: \mbP$.

If $\alpha_0 = \alpha_1 = \alpha_2 = 0$, then the set $\{x_0, x_1, x_2\}$ isolates $\msp$ since $w^2 \in F$, and hence $a_2 A$ isolates $\msp$.
We assume that $\alpha_k \ne 0$ for some $k \in \{0, 1, 2\}$.
Then the common zero locus in $\mbP$ of the set
\[
\Lambda := \{\, \alpha_k^{a_i} x_i^{a_k} - \alpha_i^k x_k^{a_i} \mid i \in \{0, 1, 2, 3\} \setminus \{k\} \, \}
\]
is a finite set of points, denoted by $\Sigma \subset \mbP$ including $\msq$.
The common zero locus in $X$ of the set $\Lambda$ is the inverse image $\pi^{-1} (\Sigma)$ and it is a finite set of points.
In particular the set $\Lambda$ isolates $\msp$, and thus 
\[
\max \{\, a_i a_k \mid i \in \{0, 1, 2, 3\} \setminus \{k\} \, \} A
\]
isolates $\msp$.
Since $k \in \{0, 1, 2\}$, we have
\[
a_2 \le \max \{\, a_i a_k \mid i \in \{0, 1, 2, 3\} \setminus \{k\} \, \} \le a_2 a_3,
\]
and thus $m A$ isolates $\msp$ for some positive integer $m$ satisfying $m \le a_2 a_3$.
By the assumption $\iota_X^2 \le 2 a_0 a_1$, we have
\[
\frac{m}{\iota_X} \le \frac{a_2 a_3}{\iota_X} \le \frac{4}{\iota_X^3 (A^3)} = \frac{4}{(-K_X^3)}
\]
since
\[
(A^3) = \frac{2 b}{a_0 a_1 a_2 a_3 b}.
\]
By Lemma \ref{lem:mtdexclsmpt}, $\msp$ is not a maximal center.
\end{proof}

\begin{Lem} \label{lem:excl110nspt}
Let $X = X_{7} \subset \mbP (1_x, 1_y, 1_z, 2_t, 3_w)$ be a Mori-Fano hypersurface defined by a quasi-homogeneous polynomial $F = F (x, y, z, t, w)$ of degree $7$.
Suppose that $X$ is quasi-smooth at $\msp_t$ and $\msp_w$.
If $X$ contains the curve $\Gamma = (x = y = z = 0)$, then we suppose in addition that $X$ is quasi-smooth along $\Gamma$.
Then no smooth point on $X$ is a maximal center.
\end{Lem}

\begin{proof}
This proof is taken from that of \cite[Lemma 2.1.8]{CP}.
The task of this proof is to observe that the same argument indeed works in this generalized setting.

Let $\msp \in X$ be a smooth point.
Suppose that $\msp \notin \Gamma$.
Then it is easy to see that $3 A$ isolates $\msp$.
By Lemma \ref{lem:mtdexclsmpt}, $\msp$ is not a maximal center since $3 < 4/(A^3) = 24/7$.
If $\Gamma$ is not contained in $X$, then $\Gamma \cap X$ consists of two singular points $\msp_t, \msp_w$.
Hence it remains to exclude a smooth point $\msp$ of $X$ contained in $\Gamma$ assuming that $X$ contains $\Gamma$.

Let $\msp \in \Gamma$ be a smooth point.
By the quasi-smoothness of $X$ along $\Gamma$, we can write
\[
F = w^2 x + t^3 y + w t a_2 + t^2 b_3 + w c_4 + t d_5 + e_7,
\]
where $a_2, \dots, e_7 \in \mbC [x, y, z]$ are quasi-homogeneous polynomials of indicated degree.
Let $S \in |A|$ be a general member which is defined by $z - \lambda x - \mu y = 0$ on $X$ for general $\lambda, \mu \in \mbC$.
The base locus of $|A|$ is $\Gamma$, and hence $S$ is smooth outside $\Gamma$.
Let $J_{C_S}$ be the Jacobi matrix of the affine cone $C_S \subset \mbA^5$ of $S$.
Then, by setting $x = y = z = 0$, we have
\[
J_{C_S}|_{\Gamma} =
\begin{pmatrix}
w^2 & t^3 & 0 & 0 & 0 \\
-\lambda & - \mu & 1 & 0 & 0
\end{pmatrix}
\]
This shows that $S$ is quasi-smooth along $\Gamma$.
We see that $S$ is a K3 surface with two singular points $\msp_t$ and $\msp_w$ of type $\frac{1}{2} (1, 1)$ and $\frac{1}{3} (1, 2)$, respectively.
Since $\Gamma \subset S$ is a smooth rational curve passing through the above two singular points, we compute 
\[
(\Gamma^2) = -2 + \frac{1}{2} + \frac{2}{3} = - \frac{1}{6}.
\]
We set $T = (y - x = 0)_X$.
Then $T|_S = \Gamma + \Delta$, where
\[
\Delta = (z - \lambda x - \mu y = y - x = F (x, x, (\lambda + \mu) x, t, w)/x = 0).
\]
We set $h_6 (x, t, w) = F (x, x, (\lambda + \mu) x, t, w)/x$.
Then
\[
h_6 = w^2 + t^3 + \alpha w t x + \beta t^2 x^2 + \gamma w x^3 + \delta t x^4 + \varepsilon x^6,
\]
where $\alpha, \dots, \varepsilon \in \mbC$ are defined by is the $a_2 (x, x, (\lambda + \mu) x) = \alpha x^2, \dots, e_7 (x, x, (\lambda + \mu) x) = \varepsilon x^7$, respectively.
Note that $h_6$ is an irreducible polynomial (for any $\alpha, \dots, \varepsilon \in \mbC$) and thus $\Delta$ is an irreducible and reduced curve of degree $1$.
We have
\[
\frac{1}{6} = (\Gamma \cdot T|_S) = (\Gamma^2) + (\Gamma \cdot \Delta) = - \frac{5}{6} + (\Gamma \cdot \Delta),
\]
and hence $(\Gamma \cdot \Delta) = 1$. 

Now assume that $\msp$ is a maximal center.
Then there is a movable linear system $\mcM \sim_{\mbQ} n A$ on $X$ such that $(X, \frac{1}{n} \mcM)$ is not canonical at $\msp$.
By inversion of adjunction, $(S, \frac{1}{n} \mcM|_S)$ is not log canonical at $\msp$.
Let $M \in \mcM$ be a general member and we write 
\[
D_S := \frac{1}{n} M|_S = \gamma \Gamma + \delta \Delta + C,
\]
where $\gamma, \delta \ge 0$ are rational numbers and $C$ is an effective divisor on $S$ whose support contains neither $\Gamma$ nor $\Delta$.
By taking intersection number of the divisors
\[
(1-\delta)A|_S \sim_{\mbQ} D_s - \delta T|_S = (\gamma - \delta)\Gamma + C
\]
and $\Delta$, we obtain $1-\delta \ge \gamma - \delta$ since $(\Delta \cdot C) \ge 0$.
This shows $\gamma \le 1$.
We see that the pair $(S, D_S) = (\gamma \Gamma + \delta \Delta + C)$ is not log canonical at $\msp$, and hence $(S, \Gamma + \delta \Delta + C)$ is not log canonical at $\msp$.
By the inversion of adjunction, $(\Gamma, (\delta \Delta + C)|_{\Gamma})$ is not log canonical at $\msp$, which implies $\mult_{\msp} (\delta \Delta + C)|_{\Gamma} > 1$.
However we have
\[
(\delta \Delta + C \cdot \Gamma) = (D_S - \gamma \Gamma \cdot \Gamma) = \frac{1}{6} + \frac{5}{6} \gamma \le 1.
\]
This is a contradiction and $\msp$ is not a maximal center. 
\end{proof}

\subsection{Cyclic quotient singular points}

We exclude terminal quotient singular points on suitable Mori-Fano weighted hypersurfaces as a maximal center.
Note that, for the exclusion of a terminal quotient singular point $\msp$ on a Mori-Fano $3$-fold $X$ as a maximal center, it is enough to exclude the Kawamata blow-up $\varphi \colon Y \to X$ at $\msp$ as a maximal extraction since $\varphi$ is the unique divisorial contraction centered at $\msp$.

\begin{Lem} \label{lem:exclqsing1}
Let $r$ and $a$ be positive integers such that
\[
(r, a) \in \{ (3, 1), (4, 1), (7, 3)\}.
\]
Let $X = X_{4 r - 2 a} \subset \mbP (1, a, r - 2a, r, 2 r - a)$ be a Mori-Fano $3$-fold weighted hypersurface of index $1$ defined by a quasi-homogeneous polynomial $F = F (x, y, z, t, w)$ of degree $4 r - 2 a$, where $x, y, z, t, w$ are homogeneous coordinates of weight $1, a, r-2a, r, 2 r - a$, respectively.
Suppose that $w^2 \in F$ and that $X$ is quasi-smooth at the point $\msp_t \in X$.
Then the $\frac{1}{r} (1, a, r-a)$ point $\msp_t$ of $X$ is not a maximal center.
\end{Lem}

\begin{proof}
By the quasi-smoothness of $X$ at $\msp := \msp_t$, we can choose homogeneous coordinates such that
\[
F = t^3 z + t^2 f_{2 r - 2 a} + t f_{3 r - 2 a} + f_{4 r - 2 a},
\]
where $f_i = f_i (x, y, z, w)$ is a quasi-homogeneous polynomial of degree $i$.
Let $\varphi \colon Y \to X$ be the Kawamata blow-up at $\msp$.
Locally around $\msp$, we can take $x, y, w$ as local orbifold coordinates of $X$, and by filtering off terms divisible by $z$ in the equation $F = 0$, we have
\begin{equation} \label{eq:No100exclsing-1}
\xi z = t^2 f_{2 r - 2 a} (x, y, 0, w) + t f_{3 r - 2 a} (x, y, 0, w) + f_{4 r - 2 a} (x, y, 0, w),
\end{equation}
where the section $\xi = - t^3 + \cdots$ does not vanish at $\msp$.
The Kawamata blow-up $\varphi$ is defined as the weighted blowup with $\wt (x, y, w) = \frac{1}{r} (1, a, r - a)$.
The minimum of the weights of the monomials in the RHS of \eqref{eq:No100exclsing-1} with respect to the weight $\wt (y, z, w) = \frac{1}{r} (1, a, r - a)$ is $\frac{2 (r-a)}{r}$, which is attained by $w^2 \in f_{4 r - 2 a} (x, y, 0, w)$.
Hence the section $z$ vanishes along $E$ to order $\frac{2 (r - a)}{r}$. 
We set $S := (z = 0)_X$ and let $\mcL \subset |a A|$ be the pencil on $X$ generated $x^a$ and $y$.
We have
\[
\Supp (S) \cap \Bs \mcL = (x = y = z = 0) \cap X = \{\msp\}
\]
since $w^2 \in F$.
For the proper transforms $\tilde{S}$ and $\tilde{T}$ of $S$ and a general member $T \in \mcL$ on $Y$, we have $\tilde{S} \sim_{\mbQ} (r - 2 a)\varphi^* A - \frac{2 (r - a)}{r} E$ and $\tilde{T} \sim_{\mbQ} a \varphi^*A - \frac{a}{r} E$.
Then we compute
\[
\begin{split}
(-K_Y \cdot \tilde{S} \cdot \tilde{T}) &= a (r - 2 a) (A^3) - \frac{2 a (r - a)}{r^3} (E^3) = 0 \\
&= \frac{a(r - 2 a)(4 r - 2 a)}{a (r - 2 a) r (2 r - a)} - \frac{2 a (r - a)}{r^3} \cdot \frac{r^2}{a (r - a)} \\
&= 0.
\end{split}
\]
By Lemma \ref{lem:exclquotsing}, $\msp$ is not a maximal center.
\end{proof}

\begin{Lem} \label{lem:exclsingpt2}
Let $X = X_{12} \subset \mbP (1_x, 1_y, 1_z, 4_t, 6_w)$ be a Mori-Fano hypersurface defined by a quasi-homogeneous polynomial $F = F (x, y, z, t, w)$ of degree $12$.
Suppose that $w^2, t^3 \in F$.
Then the singular point of type $\frac{1}{2} (1, 1, 1)$ is not a maximal center.
\end{Lem}

\begin{proof}
Let $\msp \in X$ be the singular point of type $\frac{1}{2} (1, 1, 1)$.
Note that $(x = y = z = 0) \cap X = \{\msp\}$.
We set $\mcL := |A|$ which is generated by $x, y, z$, and let $S, T \in \mcL$ be general members.
We have $\Supp S \cap \Bs \mcL = \{\msp\}$.
Let $\varphi \colon Y \to X$ be the Kawamata blow-up at $\msp \in X$.
We can take $x, y, z$ as local orbifold coordinates of $X$ at $\msp$ and $\varphi$ is the weighted blow-up with $\wt (x, y, z) = \frac{1}{2} (1, 1, 1)$.
For the proper transforms $\tilde{S}$ and $\tilde{T}$ of $S$ and $T$ on $Y$, respectively, we have $\tilde{S} \sim_{\mbQ} \tilde{T} \sim_{\mbQ} \varphi^*A - \frac{1}{2} E$.
We compute
\[
(-K_Y \cdot \tilde{S} \cdot \tilde{T}) = (A^3) - \frac{1}{2^3} (E^3) = 0.
\]
By Lemma \ref{lem:exclquotsing}, $\msp$ is not a maximal center.
\end{proof}

\begin{Lem} \label{lem:exclsingpt3}
Let $X = X_{22} \subset \mbP (1_x, 1_y, 3_z, 7_t, 11_w)$ be a Mori-Fano hypersurface defined by a quasi-homogeneous polynomial $F = F (x, y, z, t, w)$ of degree $22$.
Suppose that $w^2 \in F$ and $X$ is quasi-smooth at $\msp_z$.
Then the $\frac{1}{3} (1, 1, 2)$ point $\msp_z$ is not a maximal center.
\end{Lem}

\begin{proof}
Let $\varphi \colon Y \to X$ be the Kawamata blow-up at $\msp := \msp_t$ with exceptional divisor $E$.

We first consider the case where $z^5 t \in F$.
We can choose homogeneous coordinates such that
\[
F = z^5 t + z^4 f_{10} + z^3 f_{13} + z^2 f_{16} + z f_{19} + f_{22},
\]
where $f_i = f_i (x, y, t, w)$ is a quasi-homogeneous polynomial of degree $i$.
We can choose $x, y, w$ as a local orbifold coordinates around the point $\msp$, and $\varphi$ is the weighted blow-up with weight $\wt (x, y, w) = \frac{1}{3} (1, 1, 2)$.
We have
\[
\xi t = z^4 f_{10} (x, y, 0, w) + \cdots + f_{22} (x, y, 0, w),
\]
where $\xi = - z^5 + \cdots$ does not vanish at $\msp$.
It follows that the section $t$ vanishes along $E$ to order $\frac{4}{3}$ since $w^2 \in f_{22} (x, y, 0, w)$.
We set $S := (t = 0)_X$ and $\mcL := |A|$.
We have $\Supp S \cap \Bs \mcL = \{\msp\}$ since $w^2 \in F$.
For the proper transforms $\tilde{S}$ and $\tilde{T}$ of $S$ and a general member $T \in \mcL$, we have $\tilde{S} \sim_{\mbQ} 7 \varphi^* A - \frac{4}{3} E$ and $\tilde{T} \sim_{\mbQ} \varphi^* A - \frac{1}{3} E$.
We compute
\[
(-K_Y \cdot \tilde{S} \cdot \tilde{T}) = 7 (A^3) - \frac{4}{3^3} (E^3) = 0.
\]
By Lemma~\ref{lem:exclquotsing}, $\msp$ is not a maximal center.

We next consider the case where $z^5 t \notin F$.
We have either $z^7 x \in F$ or $z^7 y \in F$ by the quasi-smoothness of $X$ at $\msp$.
By a choice of homogeneous coordinates, we may assume $z^7 x \in F$ and we can write
\[
F = z^7 x + z^6 f_4 + z^5 f_7 + z^4 f_{10} + z^3 f_{13} + z^2 f_{16} + z f_{19} + f_{22},
\]
where $f_i = f_i (y, z, t, w)$ is a quasi-homogeneous polynomial of degree $i$.
We can choose $y, t, w$ as local orbifold coordinates at $\msp$ and $\varphi$ is the weighted blow-up with $\wt (y, t, w) = \frac{1}{3} (1, 1, 2)$.
The section $x$ vanishes along $E$ to order $\frac{4}{3}$ since $t \notin f_7$ by the assumption.
We set $S := (x = 0)_X$ and $T := (y = 0)_X$.
Then we have $\tilde{S} \cdot \tilde{T} = 2 \tilde{\Gamma}$, where $\tilde{\Gamma}$ is the strict transform of the curve $\Gamma = (x = y = w = 0) \subset X$.
We have
\[
\begin{split}
\tilde{S} &\sim_{\mbQ} - \varphi^* K_X - \frac{4}{3} E \sim_{\mbQ} - K_Y - E, \\
\tilde{T} &\sim_{\mbQ} - \varphi^* K_X - \frac{1}{3} E \sim_{\mbQ} - K_Y.
\end{split}
\]  
Moreover
\[
(\tilde{T} \cdot 2 \tilde{\Gamma}) = (\tilde{S} \cdot \tilde{T}^2) = (A^3) - \frac{4}{3^3} (E^3) = \frac{2}{21} - \frac{2}{3} < 0.
\]
By Lemma \ref{lem:exclsingptG}, the point $\msp$ is not a maximal center.
\end{proof}

\begin{Lem} \label{lem:exclsingpt4}
Let $X = X_{7} \subset \mbP (1_x, 1_y, 1_z, 2_t, 3_w)$ be a Mori-Fano hypersurface defined by a quasi-homogeneous polynomial $F = F (x, y, z, t, w)$ of degree $7$.
Suppose that $X$ contains the curve $\Gamma = (x = y = z = 0)$ and that $X$ is quasi-smooth along $\Gamma$.
Then the $\frac{1}{2} (1, 1, 1)$ point $\msp_t$ is not a maximal center.
\end{Lem}

\begin{proof}
Note that $t^2 w \notin F$ since $\Gamma \subset X$.
By the quasi-smoothness of $X$ at $\msp_t$, we can write
\[
F = t^3 x + t^2 g_3 + t (w g_2 + g_5) + w^2 g_1 + w g_4 + g_7,
\]
where $g_i = g_i (x, y, z)$ is a quasi-homogeneous polynomial of degree $i$.
We set $S = (x = 0)_X$ and $\mcL = |A|$.
We have $\Supp S \cap \Bs \mcL = \Gamma$.
By the quasi-smoothness of $X$ along $\Gamma$, $g_1 (0, y, z) \ne 0$.
It follows that, for a general member $T \in \mcL$, $\Gamma$ appears in the $1$-cycle $S \cdot T$ with coefficient $1$, and we can write $S \cdot T = \Gamma + \Delta_T$, where $\Delta_T$ is an effective $1$-cycle on $X$ such that $\Gamma \not\subset \Supp \Delta_T$.
Let $\varphi \colon Y \to X$ be the Kawamata blow-up at the $\frac{1}{2} (1, 1, 1)$ point $\msp_t$ and denote by $E$ its exceptional divisor.
Let $\tilde{S}, \tilde{T}$ and $\tilde{\Gamma}$ be proper transforms of $S, T$ and $\Gamma$, respectively, on $Y$.
We see that $\tilde{\Gamma}$ intersects $E$ transversally at one point, so that $(E \cdot \tilde{\Gamma}) = 1$ and thus 
\[
(-K_Y \cdot \tilde{\Gamma}) = \frac{1}{6} - \frac{1}{2} = - \frac{1}{3}.
\]
We have $\tilde{S} \sim_{\mbQ} \varphi^*A_{\mbQ} - \frac{3}{2} E$ and $\tilde{T} \sim \varphi^*A - \frac{1}{2} E$.
Hence we have
\[
(-K_Y \cdot \tilde{S} \cdot \tilde{T}) - (-K_Y \cdot \tilde{\Gamma}) = \left( \frac{7}{6} - \frac{3}{2} \right) + \frac{1}{3} = 0.
\]
By Lemma~\ref{lem:exclquotsing}, $\msp_t$ is not a maximal center.
\end{proof}

\section{Families \textnumero 100, 101, 102, and 103} \label{sec:birig}

Let $X = X_d \subset \mbP := \mbP (a_0, \dots, a_4)$ be a quasi-smooth member of Family \textnumero~$\msi$, where $\msi \in \{100, 101, 102, 103, 110\}$.

We briefly recall the construction of an elementary link $\hat{\sigma} \colon X \ratmap \hat{X}$ to a (non-quasi-smooth) Mori-Fano $3$-fold hypersurface $\hat{X}$ of index $1$.
We refer readers to \cite[Section 3]{ACP} for details.
The link $\hat{\sigma} = \hat{\varphi} \circ \hat{\theta} \circ \varphi^{-1}$ is of the form
\[
\xymatrix{
\mcY \ar[d]_{\varphi} \ar@{-->}[r]^{\hat{\theta}} & \hat{\mcY} \ar[d]^{\hat{\varphi}} \\
X & X',}
\]
where $\varphi, \hat{\theta}, \hat{\varphi}$ are birational maps explained below.
Let $\msq \in X$ be the cyclic quotient singular point of the highest index.
Then $\varphi \colon \mcY \to X$ is the Kawamata blow-up of $X$ at $\msq$.
The variety $\mcY$ can be naturally embedded into a suitable rank $2$ toric variety $\mbT$.
Note that there is a birational toric morphism $\Phi \colon \mbT \to \mbP$, which is a weighted blow-up at $\msq \in \mbP$, such that $\varphi = \Phi|_{\mcY}$.
We can run a $2$-ray game on $\mbT$ and get 
\[
\mbT \overset{\hat{\Theta}}{\ratmap} \hat{\mbT} \overset{\hat{\Phi}}{\to} \hat{\mbP},
\] 
where $\hat{\Theta} \colon \mbT \ratmap \hat{\mbT}$ is a birational map to a rank $2$ toric variety $\hat{\mbT}$ which is an isomorphism in codimension $1$, and $\hat{\Phi} \colon \hat{\mbT} \to \hat{\mbP}$ is a toric divisorial contraction to a suitable weighted projective space $\hat{\mbP}$ with center a point $\hat{\msq} \in \hat{\mbP}$.
Note that $\hat{\Phi}$ is a weighted blow-up at $\hat{\msq} \in \hat{\mbP}$.
Let $\hat{\mcY} \subset \hat{\mbT}$ be the birational transform of $\mcY$ on $\hat{\mbT}$.
Then $\hat{\theta} := \hat{\Theta}|_{\mcY} \colon \mcY \ratmap \hat{\mcY}$ is a composite of inverse flips, a flop, and flips.
Moreover the image $\hat{X} = \hat{\Phi} (\hat{\mcY})$ of $\hat{Y}$ is a Mori-Fano $3$-fold hypersurface of index $1$ and $\hat{\varphi} = \hat{\Phi}|_{\hat{\mcY}} \colon \hat{\mcY} \to \hat{X}$ is a divisorial contraction with center $\hat{\msq} \in \hat{X}$.

In this and next sections, we prove Theorems \ref{thm:main2} and \ref{thm:main} by considering each family separately.
In the beginning of each subsections below, we give more details on the above link $\hat{\sigma}$, especially on the descriptions of $\hat{\varphi}$ and $\hat{X}$, and on relevant information on $X$ and $\hat{X}$. 
We fix the following notation:
\begin{itemize}
\item $F = F (x, y, z, t, u)$ is the polynomial defining the weighted hypersurface $X$.
\item $A \in \Cl (X)$ is the Weil divisor class which is the positive generator of $\Cl (X) \cong \mbZ$.
\item $\iota_X := \sum_{i = 0}^4 a_i - d$ is the Fano index of $X$ so that we have $-K_X \sim \iota_X A$. 
\item $\msq \in X$ is the cyclic quotient singular point of the highest index, and $\varphi \colon \mcY \to X$ is the Kawamata blow-up at $\msq$ with exceptional divisor $\mcE$.
\item $\hat{F}$ is the polynomial defining the weighted hypersurface $\hat{X}$
\item $\hat{A} := -K_{\hat{X}}$.
\item $\hat{\varphi} \colon \hat{\mcY} \to \hat{X}$ is the divisorial contraction obtained as above with exceptional divisor $\hat{\mcE}$, and $\hat{\msq} \in \hat{X}$ is the center $\hat{\varphi}$.
\end{itemize}
 
\subsection{Family \textnumero 100}

Let $X = X_{18} \subset \mbP (1_x,2_y,3_z,5_t,9_w)$ be a quasi-smooth member of Family \textnumero $100$, which is defined by  a quasi-homogeneous polynomial $F = F (x, y, z, t, u)$ of degree $18$, and whose basic information is given as follows:
\begin{itemize}
\item $\iota_X = 2$.
\item $(A^3) = 1/15$.
\item $\Sing (X) = \{\frac{1}{3} (1,1,2), \frac{1}{5} (1,2,3)\}$, and $\msq = \msp_t$ is the $\frac{1}{5} (1, 2, 3)$ point.
\end{itemize}
The elementary link $\hat{\sigma} \colon X \ratmap \hat{X}$ is embedded into the following toric diagram.
\[
\xymatrix{
\text{$\mbT := \mbT \begin{pNiceArray}{cc|cccc}[first-row]
u & t & w & y & z & x \\
0 & 5 & 9 & 2 & 3 & 1 \\
-5 & 0 & 2 & 1 & 4 & 3
\end{pNiceArray}$} \ar@{-->}[r]^{\hat{\Theta}} \ar[d]_{\Phi} &
\text{$\mbT \begin{pNiceArray}{cccc|cc}[first-row]
u & t & w & y & z & x \\
3 & 4 & 6 & 1 & 0 & -1 \\
1 & 3 & 5 & 1 & 1 & 0
\end{pNiceArray} =: \hat{\mbT}$} \ar[d]^{\hat{\Phi}} \\
\mbP (5_t, 9_w, 2_y, 3_z, 1_x) & \mbP (1_u, 3_t, 5_w, 1_y, 1_z)}
\]
The birational morphisms $\Phi, \hat{\Phi}$, and the birational map $\hat{\Theta}$ are defined as follows:
\[
\begin{split}
\Phi & \colon (u\!:\!t \, | \, w\!:\!y\!:\!z\!:\!x) \mapsto (t\!:\!w u^{2/5}\!:\!y u^{1/5}\!:\!z u^{4/5}\!:\!x u^{3/5}), \\
\hat{\Phi} & \colon (u\!:\!t\!:\!w\!:\!y \, | \,  z\!:\!x) \mapsto (u x^3\!:\!t x^4\!:\!w x^6\!:\!y x\!:\!z), \\
\hat{\Theta} & \colon (u\!:\!t \, | \, w\!:\!y\!:\!z\!:\!x) \mapsto (u\!:\!t\!:\!w\!:\!y \, | \,  z\!:\!x).
\end{split}
\]  
The varieties $\mcY$ and $\hat{\mcY}$ are the hypersurfaces in $\mbT$ and $\hat{\mbT}$, respectively, which are defined by the same equation
\[
G (u, x, y, z, t, w) := u^{- 4/5} F (u^{3/5} x, u^{1/5} y, u^{4/5} z, t, u^{2/5} w) = 0.
\]
The variety $X' = X'_{10} \subset \mbP (1_u, 1_y, 1_z, 3_t, 5_w)$ is the hypersurface defined by
\[
\hat{F} (u, y, z, t, w) := G (u, 1, y, z, t, w) = 0,
\]
where $\hat{F}$ is a quasi-homogeneous polynomial of degree $10$, and we have the following basic information of $X'$:
\begin{itemize}
\item $\iota_{\hat{X}} = 1$.
\item $(\hat{A}^3) = 2/3$.
\item $\Sing (\hat{X}) = \{\frac{1}{3} (1, 1, 2), cE_6\}$, and $\hat{\msq} = \hat{\msp}_z \in \hat{X}$ is the $cE_6$ point (see Lemma \ref{lem:loceqcENo100}).
\end{itemize}

\subsubsection{Maximal extractions of $X$}

\begin{Lem} \label{lem:No100sing1}
The $\frac{1}{3} (1,1,2)$ point of $X$ is not a maximal center.
\end{Lem}

\begin{proof}
Let $\msp = \msp_z$ be the $\frac{1}{3} (1,1,2)$ point, and let $\psi \colon Y \to X$ be the Kawamata blow-up of $X$ at $\msp$ with exceptional divisor $E$.
By a choice of homogeneous coordinates and by the quasi-smoothness of $X$, we can write
\[
F = z^3 w + z^2 f_{12} + z f_{15} + f_{18},
\]
where $f_i = f_i (x, y, t, w)$ is a quasi-homogeneous polynomial of degree $i$.
By the quasi-smoothness of $X$, we may assume that $\coeff_{f_{18}} (w^2) = \coeff_{f_{15}} (t^3 z) = 1$ after rescaling $t, z, w$.
The open set $(z \ne 0) \subset X$ is naturally isomorphic to the affine hyperquotient
\[
(F (x, y, 1, t, w) = 0)/\mbZ_3 (1_x, 2_y, 2_t, 0_w),
\]
and the point $\msp$ corresponds to the origin.
By eliminating $w$, we can choose $x, y, t$ as local orbifold coordinates of $X$ at $\msp$, and $\psi$ is locally the weighted blow-up with $\wt (x, y, t) = \frac{1}{3} (2, 1, 1)$.
Filtering off terms divisible by $w$ in the equation $F (x, y, 1, t, w) = 0$, we have
\[
w (-1 + \cdots) = f_{12} (x, y, t, 0) + f_{15} (x, y, t, 0) + f_{18} (x, y, t, 0),
\]
where the omitted term $\cdots$ in the left-hand side is a polynomial which vanishes at the origin.
The lowest weight of monomials in the right-hand side of the above equation with respect to $\wt (x, y, t) = \frac{1}{3} (2,1,1)$ is $\frac{3}{3}$, which is attained by the monomial $t^3 \in f_{15} (x, y, t, 0)$.
Hence $\ord_E (w) = \frac{3}{3}$.

We set $S := (x = 0)_X$ and $T := (y = 0)_X$.
For the proper transforms $\tilde{S}$ and $\tilde{T}$ of $S$ and $T$ on $Y$, respectively, we have 
\[
\begin{split}
\tilde{S} &\sim_{\mbQ} \psi^*A - \frac{2}{3} E \sim_{\mbQ} \frac{1}{2} (-K_Y) - \frac{1}{2} E, \\  
\tilde{T} &\sim_{\mbQ} 2 \psi^*A - \frac{1}{3} E \sim -K_Y.
\end{split}
\] 
We have
\[
\Gamma := S \cap T \cong (w^2 + w z^3 + t^3 z = 0) \subset \mbP (3_z,5_t,9_w)
\]
and $\Gamma$ is an irreducible and reduced curve.
Since $\tilde{S} \cap \tilde{T} \cap E$ does not contain a curve (in fact, it consists of the singular point of $E \cong \mbP (1,1,2)$), we see that $\tilde{S} \cap \tilde{T} = \tilde{\Gamma}$, where $\tilde{\Gamma}$ is the proper transform of $\Gamma$ on $Y$.
We compute
\[
\begin{split}
(\tilde{T} \cdot \tilde{\Gamma}) &= (\tilde{T}^2 \cdot \tilde{S}) 
= (2 \psi^* A - \frac{1}{3} E)^2 \cdot (\psi^* A - \frac{2}{3} E) \\
&= \frac{2^2}{15} - \frac{2}{3^3} \cdot \frac{3^2}{2} < 0.
\end{split}
\]
Thus, by Lemma~\ref{lem:exclsingptG}, $\msp$ is not a maximal center.
\end{proof}

By Lemma \ref{lem:No100sing1} and \cite[Lemmas 4.1 and 4.2]{ACP}, we have the following.

\begin{Prop} \label{prop:No100main1}
The Kawamata blow-up $\varphi \colon \mcY \to X$ at the $\frac{1}{5} (1, 2, 3)$ point $\msq \in X$ is the unique maximal extraction of $X$, and $\EL (X) = \{\hat{\sigma}\}$.
\end{Prop}

\subsubsection{Divisorial contractions centered at $\hat{\msq} \in \hat{X}$}

For a positive integer $i$, we denote by $\msw_i$ the weight on $u, y$ defined by $\msw_i (u, y) = (i, 1)$.

\begin{Lem} \label{lem:loceqcENo100}
The singularity $\hat{\msq} \in \hat{X}$ is of type $cE_6$, and it is equivalent to $(\phi = 0) \subset \mbC^4_{w, t, u, y}$, where
\begin{equation} \label{eq:cEeqNo100}
\phi (w, t, u, y) = w^2 + t^3 + t^2 f (u, y) + t g (u, y) + h (u, y),
\end{equation}
for some $f, g, h \in \mbC [u, y]$ satisfying the following properties.
\begin{enumerate}
\item $g = 2 \mu u^3 + \tilde{g}$, where $\mu \in \mbC$ and $\tilde{g} = \tilde{g} (z, u) := g_{\msw_1 \ge 4} (z, u)$. 
\item $h = - u^4 + \tilde{h}$, where $\tilde{h} = \tilde{h} (z, u) := h_{\msw_1 \ge 5} (z, u)$.
\item $\msw_1 (f) \ge 2$, $\msw_1 (\tilde{g}) \ge 4$, and $\msw_1 (\tilde{h}) \ge 6$.
\item $u \mid f_{\msw_1 = 2}$.
\item $u^2 \mid \tilde{g}_{\msw_1 = 4}$ and $u \mid \tilde{g}_{\msw_1 = 5}$. 
\item $u^3 \mid \tilde{h}_{\msw_1 = 6}$, $u^2 \mid \tilde{h}_{\msw_1 = 7}$ and $u \mid \tilde{h}_{\msw_1 = 8}$.
\item $\msw_2 (f) \ge 3$, $\msw_2 (g) \ge 6$ and $\msw_2 (h) = 8$.
\item $u \mid f_{\msw_2 = 3}$.
\item $u^2 \mid g_{\msw_2 = 6}$ and $u \mid g_{\msw_2 = 7}$.
\item $h_{\msw_2 = 8} = - u^4$, $u^3 \mid \tilde{h}_{\msw_2 = 9}$, $u^2 \mid \tilde{h}_{\msw_2 = 10}$ and $u \mid \tilde{h}_{\msw_2 = 11}$.
\end{enumerate}
\end{Lem}

\begin{proof}
By a choice of homogeneous coordinates $x, y, z, t, w$, we can write
\[
F = t^3 z + t^2 e_8 + t e_{13} + w^2 + e_{18},
\]
where $e_i = e_i (x, y, z)$ is a quasi-homogeneous polynomial of degree $i$. 
By the quasi-smoothness of $X$, we may assume $\coeff_{e_{18}} (z^6) = -1$ after rescaling $z$.
We set $\mu := \coeff_{e_{13}} (z^4 x)/2$.
Then we have
\[
\begin{split}
\hat{F} &= u^{-4/5} F (u^{3/5}, u^{1/5} y, u^{4/5} z, t, u^{2/5} w) \\
&= t^3 z + t^2 \hat{e}_8 + t \hat{e}_{13} + w^2 + \hat{e}_{18},
\end{split}
\] 
where
\[
\hat{e}_i = \hat{e}_i (u, y, z) := u^{-4/5} e_i (u^{3/5}, u^{1/5} y, u^{4/5} z)
\]
for $i = 8, 13, 18$.
By setting $z = 1$, the germ $\hat{\msq} \in \hat{X}$ is naturally equivalent to the hypersurface germ $(\phi = 0) \subset \mbC^4_{w, t, u, y}$, where
\[
\begin{split}
\phi &:= \hat{F} (u, y, 1, t, w) \\
&= t^3 + t^2 \hat{e}_8 (u, y, 1) + t \hat{e}_{13} (u, y, 1) + w^2 + \hat{e}_{18} (u, y, 1).
\end{split}
\]
Note that $\coeff_{\hat{e}_{13} (u, y, 1)} (u^3) = 2 \mu$ and $\coeff_{\hat{e}_{18} (u, y, 1)} (u^4) = -1$. 
By setting 
\[
f (u, y) := \hat{e}_8 (u, y, 1), \ 
g (u, y) := \hat{e}_{13} (u, y, 1), \ 
h (u, y) := \hat{e}_{18} (u, y, 1),
\]
it is straightforward to check that the conditions (1)--(10) are all satisfied. 
Finally it is easily seen that the singularity $\hat{\msq} \in \hat{X}$ is of type $cE_6$ since $u^4 \in h (z, u)$.
\end{proof}

\begin{Lem} \label{lem:boundNo100}
There are at most $6$ distinct divisors of discrepancy $1$ over $\hat{\msq} \in \hat{X}$.
\end{Lem}

\begin{proof}
We see that the non-Gorenstein singular points of $\hat{\mcY}$ along $\hat{\mcE}$ consist of $2$ points of type $\frac{1}{3} (1, 1, 2)$ and $1$ point of type $\frac{1}{2} (1, 1, 1)$.
By the arguments in Section \ref{sec:genclsfcont} (especially by Lemma \ref{lem:quotdiv}), there are at most $5$ distinct exceptional divisors of discrepancy $1$ over $\hat{\msq} \in \hat{X}$ other than $\hat{\mcE}$.
\end{proof}

\begin{Lem} \label{lem:No100divcont}
The morphism $\hat{\varphi} \colon \hat{\mcY} \to \hat{X}$ is the unique divisorial contraction centered at the $cE$ point $\hat{\msq} \in \hat{X}$.
\end{Lem}

\begin{proof}
We define various weighted blow-ups of $\hat{X}$ centered at $\hat{\msq}$ as follows:
\begin{itemize}
\item $\psi_1 \colon W_1 \to \hat{X}$ is the weighted blow-up with $\wt (w, t, u, y) = (3, 2, 2, 1)$ which is described in Lemma \ref{lem:divdisc1cEv1}(1).
\item $\psi_2^{\pm} \colon W_2^{\pm} \to \hat{X}$ is the weighted blow-up with $\wt (w, t, u, y) = (4, 2, 1, 1)$ composed with a suitable isomorphism which is described in Lemma \ref{lem:divdisc1cEv2}.
\item $\psi_3^{\pm} \colon W_3^{\pm} \to \hat{X}$ is the weighted blow-up with weight $\wt (w, t, u, y) = (5, 3, 2, 1)$ composed with a suitable isomorphism which is described in Lemma \ref{lem:divdisc1cEv3}.
\end{itemize}
Let $E_1$, $E_2^{\pm}$, and $E_3^{\pm}$ be the exceptional divisors of $\psi_1$, $\psi_2^{\pm}$, and $\psi_3^{\pm}$, respectively.
By Lemmas \ref{lem:divdisc1cEv1}(1), \ref{lem:divdisc1cEv2}, \ref{lem:divdisc1cEv3}, and \ref{lem:loceqcENo100}, $E_1$, $E_2^{\pm}$, and $E_3^{\pm}$ are divisors of discrepancy $1$ over $\hat{\msq} \in \hat{X}$, and the varieties $W_1$, $W_2^{\pm}$, and $W_3^{\pm}$ have a non-terminal singularity.
By Lemma \ref{lem:boundNo100}, $E_1$, $E_2^{\pm}$, $E_3^{\pm}$ and $\hat{\mcE}$ are all the divisors of discrepancy $1$ over $\hat{\msq} \in \hat{X}$.
By Lemma \ref{lem:divcontwbl}, none of $E_1, E_2^{\pm}$, $E_3^{\pm}$ is realized by the exceptional divisor of a divisorial contraction of $\hat{\msq} \in \hat{X}$.
Finally it follows from the existence of $6$ divisors of discrepancy $1$ over the point $\hat{\msq} \in \hat{X}$ of type $cE_6$ and Proposition~\ref{prop:ncd2cE6} that there is no divisor of discrepancy greater than $1$ over $\hat{\msq} \in \hat{X}$.  
This completes the proof.
\end{proof}


\subsubsection{Maximal extractions of $\hat{X}$}

\begin{Prop} \label{prop:No100main2}
The divisorial contraction $\hat{\varphi} \colon \hat{\mcY} \to \hat{X}$ centered at the $cE$ point $\hat{\msq}$ is the unique maximal extraction of $\hat{X}$, and $\EL (\hat{X}) = \{\hat{\sigma}^{-1}\}$.
\end{Prop}

\begin{proof}
By Lemmas~\ref{lem:exclsmpts} and \ref{lem:exclcurves}, no smooth point and no curve on $\hat{X}$ is a maximal center. 
By Lemma~\ref{lem:exclqsing1}, the $\frac{1}{3} (1, 1, 2)$ point is not a maximal center.
Finally $\hat{\varphi}$ is the unique divisorial contraction centered at $\hat{\msq}$ by Lemma~\ref{lem:No100divcont}, and thus the assertion follows.
\end{proof}

Theorem~\ref{thm:main2}, and hence Theorem~\ref{thm:main}, for Family \textnumero~100 follow from Propositions~\ref{prop:No100main1} and \ref{prop:No100main2}.

\subsection{Family \textnumero 101}

Let $X = X_{22} \subset \mbP (1_x,2_y,3_z,7_t,11_w)$ be a quasi-smooth member of Family \textnumero $102$, which is defined by a quasi-homogeneous polynomial $F = F (x, y, z, t, w)$ of degree $22$, and whose basic information is given as follows:
\begin{itemize}
\item $\iota_X = 2$.
\item $(A^3) = 1/21$.
\item $\Sing (X) = \{\frac{1}{3} (1,1,2), \frac{1}{7} (1,2,5)\}$, and $\msq = \msp_t$ is the $\frac{1}{7} (1, 2, 5)$ point.
\end{itemize}
The elementary link $\hat{\sigma} \colon X \ratmap \hat{X}$ is embedded into the following toric diagram.

\[
\xymatrix{
\text{$\mbT := \mbT \begin{pNiceArray}{cc|cccc}[first-row]
u & t & w & y & z & x \\
0 & 7 & 11 & 2 & 3 & 1 \\
-7 & 0 & 2 & 1 & 5 & 4
\end{pNiceArray}$} \ar@{-->}[r]^{\hat{\Theta}} \ar[d]_{\Phi} &
\text{$\mbT \begin{pNiceArray}{cccc|cc}[first-row]
u & t & w & y & z & x \\
3 & 5 & 7 & 1 & 0 & -1 \\
1 & 4 & 6 & 1 & 1 & 0
\end{pNiceArray} =: \hat{\mbT}$} \ar[d]^{\hat{\Phi}} \\
\mbP (7_t, {11}_w, 2_y, 3_z, 1_x) & \mbP (1_u, 4_t, 6_w, 1_y, 1_z)}
\]
The birational morphisms $\Phi, \hat{\Phi}$, and the birational map $\hat{\Theta}$ are defined as follows:
\[
\begin{split}
\Phi & \colon (u\!:\!t \, | \, w\!:\!y\!:\!z\!:\!x) \mapsto (t\!:\!w u^{2/7}\!:\!y u^{1/7}\!:\!z u^{5/7}\!:\!x u^{4/7}), \\
\hat{\Phi} & \colon (u\!:\!t\!:\!w\!:\!y \, | \,  z\!:\!x) \mapsto (u x^3\!:\!t x^5\!:\!w x^7\!:\!y x\!:\!z), \\
\hat{\Theta} & \colon (u\!:\!t \, | \, w\!:\!y\!:\!z\!:\!x) \mapsto (u\!:\!t\!:\!w\!:\!y \, | \,  z\!:\!x).
\end{split}
\]  
The varieties $\mcY$ and $\hat{\mcY}$ are the hypersurfaces in $\mbT$ and $\hat{\mbT}$, respectively, which are defined by the same equation
\[
G (u, x, y, z, t, w) := u^{- 4/7} F (u^{4/7} x, u^{1/7} y, u^{5/7} z, t, u^{2/7} w) = 0.
\]
The variety $\hat{X} = \hat{X}_{12} \subset \mbP (1_u, 1_y, 1_z, 4_t, 6_w)$ is the hypersurface defined by
\[
\hat{F} (u, y, z, t, w) := G (u, 1, y, z, t, w) = 0,
\]
where $\hat{F}$ is a quasi-homogeneous polynomial of degree $12$, and we have the following basic information of $\hat{X}$:
\begin{itemize}
\item $\iota_{\hat{X}} = 1$.
\item $(\hat{A}^3) = 1/2$.
\item $\Sing (\hat{X}) = \{\frac{1}{2} (1, 1, 2), cE_{7, 8}\}$, where $\hat{\msq} = \hat{\msp}_z \in \hat{X}$ is the $cE_{7,8}$ point (note that $cE_{7, 8}$ means that $\hat{\msq}$ is of type either $cE_7$ or $cE_8$, see Lemma ~\ref{lem:loceqcENo101}).
\end{itemize}

\subsubsection{Maximal extractions of $X$}

\begin{Lem} \label{lem:No101sing1}
The $\frac{1}{3} (1,1,2)$ point of $X$ is not a maximal center.
\end{Lem}

\begin{proof}
Let $\msp = \msp_z \in X$ be the $\frac{1}{3} (1,1,2)$ point, and let $\psi \colon Y \to X$ be the Kawamata blow-up of $X$ at $\msp$ with exceptional divisor $E$.
By a choice of homogeneous coordinates and by the quasi-smoothness of $X$, we can write
\[
F =
\begin{cases}
z^5 t + z^4 f_{10} + z^3 f_{13} + z^2 f_{16} + z f_{19} + f_{22}, & \text{if $z^5 t \in F$}, \\
z^7 x + z^6 f_4 + z^5 f_7 + z^4 f_{10} + z^3 f_{13} + z^2 f_{16} + z f_{19} + f_{22}, & \text{if $z^5 t \notin F$},
\end{cases}
\]
where $f_i = f_i (x, y, t, w)$ is a quasi-homogeneous polynomial of degree $i$.
By the quasi-smoothness of $X$, we may assume that $\coeff_{f_{22}} (w^2) = 1$ by rescaling $w$.

The open set $(z \ne 0) \subset X$ is naturally isomorphic to the affine hyperquotient
\[
(F (x, y, 1, t, w) = 0)/\mbZ_3 (1_x, 2_y, 1_t, 2_w),
\]
and the point $\msp$ corresponds to the origin.
If $z^5 t \in F$ (resp.\ $z^5 t \notin F$), then we can choose $x, y, w$ (resp.\ $y, t, w$) as local orbifold coordinates of $X$ at $\msp$ and $\psi$ is the weighted blow-up with $\wt (x, y, w) = \frac{1}{3} (2, 1, 1)$ (resp.\ $\wt (y, t, w) = \frac{1}{3} (1, 2, 1)$).
Suppose that $z^5 t \in F$. 
Then filtering off terms divisible by $t$ in the equation $F (x, y, 1, t, w) = 0$, we have
\[
(-1 + \cdots) t = f_{10} (x, y, 0, w) + f_{13} (x, y, 0, w) + \cdots + f_{22} (x, y, 0, w),
\]
where the omitted term $\cdots$ in the left-hand side is a polynomial which vanishes at the origin.
The lowest weight of the monomials in the right-hand side of the above equation is $\frac{2}{3}$, which is attained by $w^2 \in f_{22} (x, y, 0, w)$, and thus $\ord_E (t) = \frac{2}{3}$.
Suppose that $z^5 t \notin F$.
Then filtering off terms divisible by $x$ in the equation $F (x, y, 1, t, w) = 0$, we have
\[
(-1+ \cdots) x = f_4 (0, y, t, w) + f_7 (0, y, t, w) + \cdots + f_{22} (0, y, t, w).
\]
The lowest weight of the monomials in the right-hand side of the above equation with respect to $\wt (y, t, w) = \frac{1}{3} (1, 2, 1)$ is $\frac{2}{3}$, which is attained by $w^2 \in f_{22} (0, y, t, w)$.
Hence in both cases we have $\ord_E (x, y, t, w) = \frac{1}{3} (2, 1, 2, 1)$.

We set $S := (x = 0)_X$ and $T := (y = 0)_X$.
Then, for the proper transforms $\tilde{S}$ and $\tilde{T}$ os $S$ and $T$ on $Y$, we have
\[
\begin{split}
\tilde{S} &\sim_{\mbQ} \psi^*A - \frac{2}{3} E \sim_{\mbQ} - \frac{1}{2} K_Y - \frac{1}{2} E, \\
\tilde{T} &\sim_{\mbQ} 2 \psi^*A - \frac{1}{3} E \sim_{\mbQ} - K_Y.
\end{split}
\]
We have
\[
S \cap T \cong
\begin{cases}
(z^5 t + w^2 = 0) \subset \mbP (3_z, 7_t, 11_w), & \text{if $z^5 t \in F$}, \\
(w^2 = 0) \subset \mbP (3_z, 7_t, 11_z), & \text{if $z^5 t \notin F$}.
\end{cases}
\]
Let $\Gamma$ be the support of $S \cap T$, which is an irreducible and reduced curve on $X$.
Note that we have $S \cap T = m \Gamma$. where
\[
m = 
\begin{cases}
1, & \text{if $z^5 t \in F$}, \\
2, & \text{if $z^5 t \notin F$}.
\end{cases}
\]
It is easy to see that $\tilde{S} \cap \tilde{T} \cap E$ consists of $1$ point.
It follows that $\tilde{S} \cap \tilde{T} = m \tilde{\Gamma}$, where $\tilde{\Gamma}$ is the proper transform of $\Gamma$.
We compute
\[
\begin{split}
m (\tilde{T} \cdot \tilde{\Gamma}) &= (\tilde{T}^2 \cdot \tilde{S}) = (2 \psi^*A - \frac{2}{3})^3 \cdot (\psi^*A - \frac{1}{3}E)  \\
&= \frac{2^2}{21} - \frac{2^2}{3^3} \cdot \frac{3^2}{2} = - \frac{1}{7} < 0.
\end{split}
\] 
By Lemma~\ref{lem:exclsingptG}, $\msp$ is not a maximal center.
\end{proof}

The above Lemma \ref{lem:No101sing1} combined with \cite[Lemma 4.1 and 4.2]{ACP}, we have the following.

\begin{Prop} \label{prop:No101main1}
The Kawamata blow-up $\varphi \colon \mcY \to X$ at the $\frac{1}{7} (1, 2, 5)$ point $\msq$ is the unique maximal extraction of $X$, and $\EL (X) = \{\hat{\sigma}\}$.
\end{Prop}

\subsubsection{Divisorial contractions centered at $\hat{\msq} \in \hat{X}$}

For a positive integer $i$, we denote by $\msw_i$ the weight on $u, y$ defined by $\msw_i (u, y) = (i, 1)$.

\begin{Lem} \label{lem:loceqcENo101}
The singularity $\hat{\msq} \in \hat{X}$ is of type $cE_7$ or $cE_8$, and it is equivalent to $(\phi = 0) \subset \mbC^4_{w, t, u, y}$, where
\begin{equation} \label{eq:cEeqNo101}
\phi (w, t, u, y) = w^2 + t^3 + t^2 f (u, y) + t g (u, y) + h (u, y),
\end{equation}
for some $f (u, y), g (u, y), h (u, y) \in \mbC [u, y]$ satisfying the following properties.
\begin{enumerate}
\item $g_{\msw_1 = 3} = \mu u^3$ and $h_{\msw_1 = 5} = \nu u^5$ for some $\mu, \nu \in \mbC$ such that $(\mu, \nu) \ne (0, 0)$.
Moreover if $\mu \ne 0$, then $\nu = 0$.
\item $\msw_1 (f) \ge 2$, $\msw_1 (g) \ge 3$ and $\msw_1 (h) \ge 5$.
\item $u^3 \mid g_{\msw_1 = 3}$ and $u^3 \mid g_{\msw_1 = 4}$.
\item $u^4 \mid h_{\msw_1 = 6}$ and $u^4 \mid h_{\msw_1 = 7}$.
\item $\msw_2 (f) \ge 3$, $\msw_2 (g) \ge 6$ and $\msw_2 (h) \ge 10$.
\item $u \mid f_{\msw_2 = 3}$.
\item $u^3 \mid g_{\msw_2 = 6}$ and $u^2 \mid g_{\msw_2 = 7}$.
\item $u^4 \mid h_{\msw_2 = 10}$ and $u^3 \mid h_{\msw_2 = 11}$.
\item $\msw_3 (f) \ge 4$, $\msw_3 (g) \ge 9$ and $\msw_3 (h) \ge 14$.
\end{enumerate}
The singularity $\hat{\msq} \in \hat{X}$ is of type $cE_7$ if $u^3 \in g$, and of type $cE_8$ if $u^3 \notin g$.
\end{Lem}

\begin{proof}
By a choice of homogeneous coordinates $x, y, z, t, w$ of $\mbP (1, 2, 3, 7, 11)$ and by the quasi-smoothness of $X$, we can write
\[
F = w^2 + t^3 x + t^2 e_8 (x, y, z) + t e_{15} (x, y, z) + e_{22} (x, y, z),
\]
where $e_i = e_i (x, y, z)$ is a quasi-homogeneous polynomial of degree $i$.
Set $\mu = \coeff_{e_{15}} (z^5)$ and $\nu = \coeff_{e_{22}} (z^7 x)$.
By the quasi-smoothness of $X$ at the point $\msp_x \in X$, we have $(\lambda, \mu) \ne (0, 0)$.
Moreover, if $\mu \ne 0$, then by replacing $t \mapsto t - \mu z^2 x$, we may assume $\nu = 0$. 
Then we have
\[
\begin{split}
\hat{F} &= u^{-4/7} F (u^{4/7}, u^{1/7} y, u^{5/7} z, t, u^{2/7} w) \\
&= w^2 + t^3 + t^2 \hat{e}_8 (u, y, z) + t \hat{e}_{15} (u, y, z) + \hat{e}_{22} (u, y, z),
\end{split}
\]
where
\[
\hat{e}_i = \hat{e}_i (u, y, z) := u^{-4/7} e_i (u^{4/7}, u^{1/7} y, u^{5/7} z)
\]
for $i = 8, 15, 22$.
By setting $z = 1$, the germ $\hat{\msp} \in \hat{X}$ is naturally isomorphic to the origin of the hypersurface germ in $\mbA^4_{u, y, t, w}$ defined by $\phi = 0$, where
\[
\phi := \hat{F} (u, y, 1, t, w) = w^2 + t^3 + t^2 \hat{e}_8 (u, y, 1) + t \hat{e}_{15} (u, y, 1) + \hat{e}_{22} (u, y, 1).
\]
Note that $\coeff_{\hat{e}_{15}} (u^3) = \lambda$ and $\coeff_{\hat{e}_{22}} (u^5) = \mu$.
By setting
\[
f := \hat{e}_8 (u, y, 1), \ 
g := \hat{e}_{15} (u, y, 1), \ 
h := \hat{e}_{22} (u, y, 1),
\]
it is straightforward to check that the conditions (1)--(9) are all satisfied.
Finally we see that $\hat{\msq} \in \hat{X}$ is not of type $cE_6$ since $\msw_1 (h) \ge 5$, and that it is of type $cE_7$ if and only if $\msw_1 (g) = 3$, which is equivalent to the condition $u^3 \in g$.
This completes the proof.
\end{proof}

\begin{Lem} \label{prop:No101divcont}
The number of distinct divisors of discrepancy $1$ over $\hat{\msq} \in \hat{X}$ is at most $6$ $($resp.\ $5$$)$ if $\hat{\msq} \in \hat{X}$ is of type $cE_7$ $($resp.\ $cE_8$$)$.
\end{Lem}

\begin{proof}
We see that the non-Gorenstein singular points of $\hat{\mcY}$ along $\hat{\mcE}$ consist of $1$ point of type $\frac{1}{5} (1, 2, 3)$ and $1$ point of type $\frac{1}{3} (1, 1, 2)$.
This implies that the number of divisors of discrepancy $1$ over $\hat{\msq} \in \hat{X}$ are at most $7$ including $\hat{\mcE}$.

Let $\msp \in \hat{\mcY}$ be the $\frac{1}{5} (1, 2, 3)$ point and let $G$ be the exceptional divisor of discrepancy $4/5$ over $\msp \in \hat{\mcY}$.
We will show that $a_G (K_{\hat{X}}) = 2$.
Locally around $\hat{\msq} \in \hat{X}$, the divisorial contraction $\hat{\varphi} \colon \hat{\mcY} \to \hat{X}$ is the weighted blow-up of the germ $(\phi = 0) \subset \mbC^4$ with $\wt (w, t, u, y) = (7, 5, 3, 1)$, where $\phi$ is as in \eqref{eq:cEeqNo101}.
The $t$-chart $\hat{\mcY}_t$ is the affine hyperquotient
\[
(\phi (w t^5, t^5, u t^3, y t)/t^{14} = 0)/\mbZ_5 (7_w, 4_t, 3_u, 1_y),
\]
where we can write
\[
\phi (w t^5, t^5, u t^3, y t)/t^{14} = t (1 + \xi) + w^2 + h_{\msw_2 = 14}
\]
for some polynomial $\xi = \xi (w, t, u, y) \in (w, t, u, y)$.
Note that the point $\msp$ correspond to the origin, and the divisor $\hat{\mcE}$ is cut out by $t$ on $\hat{\mcY}_t$.
In a neighborhood of $\msp \in \hat{\mcY}$ we can eliminate the valuable $t$ and we can take $w, t, u$ as a local orbifold coordinates of $\hat{\mcY}$.
Moreover the divisor $\hat{\mcE}$ is defined by $w^2 + h_{\msw_3 = 14}$.
Then $G$ is the exceptional divisor of the weighted blow-up of $\hat{\mcY}$ at $\msp$ with weight $\wt (w, u, y) = \frac{1}{5} (3, 2, 4)$.
We have $a_G (K_{\hat{\mcY}}) = \frac{4}{5}$ and $\ord_G (\hat{\mcE}) = \frac{6}{5}$ so that 
\[
a_G (K_{\hat{X}}) = \frac{4}{5} + 1 \cdot \frac{6}{5} = 2.
\]

Suppose that $\hat{\msq} \in \hat{X}$ is of type $cE_8$.
In this case $\mu = 0$ and $\lambda \ne 0$.
We show that the exceptional divisor $G'$ of discrepancy $2/3$ over the $\frac{1}{3} (1, 1, 2)$ point $\msp' \in \hat{\mcY}$ is not of discrepancy $1$ over $\hat{\msq} \in \hat{X}$.
The $u$-chart $\hat{\mcY}_u$ is the affine hyperquotient
\[
(\phi (w u^7, t u^5, u^3, y u)/u^{14} = 0)/\mbZ_3 (1_w, 2_t, 2_u, 1_y),
\]
where we can write
\[
\begin{split}
\phi (w u^7, t u^5, u^3, y u)/ & u^{14} \\
= u (\nu + \xi) +  w^2 & + t^2 (\lambda y + f_{\msw_3 = 4} (1, y)) + t g_{\msw_3 = 9} (1, y) + h_{\msw_3 = 14} (1, y)
\end{split}
\]
for some polynomial $\xi = \xi (w, t, u, y) \in (w, t, u, y)$.
The point $\msp' \in \hat{\mcY}$ is the origin of $\hat{\mcY}_u$, and we can choose $w, t, y$ as local orbifold coordinates of $\hat{\mcY}$ at $\msp'$ by eliminating $u$.
The divisor $G'$ is the exceptional divisor of the weighted blow-up of $\msp' \in \hat{\mcY}$ with $\wt (w, t, y) = \frac{1}{3} (2, 1, 2)$.
We have $a_{G'} (K_{\hat{\mcY}}) = 2/3$ and $\ord_{G'} (\hat{\mcE}) = 4/3$ so that
\[
a_{G'} (K_{\hat{X}}) = \frac{2}{3} + 1 \cdot \frac{4}{3} = 2.
\]
This completes the proof.
\end{proof}

\begin{Lem} \label{lem:No101divcont}
The morphism $\hat{\varphi} \colon \hat{\mcY} \to \hat{X}$ is the unique divisorial contraction centered at the $cE$ point $\hat{\msq} \in \hat{X}$.
\end{Lem}

\begin{proof}
We define various weighted blow-ups of $\hat{X}$ centered at $\hat{\msq}$ as follows.
\begin{itemize}
\item $\psi_1 \colon W_1 \to \hat{X}$ is the weighted blow-up with $\wt (w, t, u, y) = (3, 2, 2, 1)$ which is described in Lemma \ref{lem:divdisc1cEv1}(1).
\item $\psi_2 \colon W_2 \to \hat{X}$ is the weighted blow-up with $\wt (w, t, u, y) = (3, 3, 1, 1)$ which is described in Lemma \ref{lem:divdisc1cEv1}(2).
\item $\psi_3 \colon W_3 \to \hat{X}$ is the weighted blow-up with $\wt (w, t, u, y) = (5, 4, 2, 1)$ which is described in Lemma \ref{lem:divdisc1cEv1}(3).
\item $\psi'_3 \colon W'_3 \to \hat{X}$ is the weighted blow-up with $\wt (w, t, u, y) = (5, 4, 2, 1)$ composed with a suitable isomorphism which is described in Lemma \ref{lem:divdisc1cEv4}.
\item $\psi_4 \colon W_4 \to \hat{X}$ is the weighted blow-up with $\wt (w, t, u, y) = (5, 4, 2, 1)$ which is described in Lemma \ref{lem:divdisc1cEv1}(4).
\end{itemize}
Let $E_1, \dots, E_4$ and $E'_3$ be the exceptional divisor of $\psi_1, \dots, \psi_4$ and $\psi'_3$, respectively.
By Lemmas \ref{lem:loceqcENo101}, \ref{lem:divdisc1cEv1}, and \ref{lem:divdisc1cEv4} each of $E_1, E_3, E'_3, E_4$ is a prime exceptional divisor of discrepancy $1$ over $\hat{\msq} \in \hat{X}$ and none of them is realized by the exceptional divisor of a divisorial contraction.
In case $\hat{\msq} \in \hat{X}$ is of type $cE_7$, then we have $\tauonewt (h) \ge 6$ by Lemma \ref{lem:loceqcENo101}, and hence $E_2$ is also a prime exceptional divisor of discrepancy $1$ over $\hat{\msq} \in \hat{X}$ and it is not realized by the exceptional divisor of a divisorial contraction.
By Lemma \ref{prop:No101divcont}, $\varphi'$ is the unique divisorial contraction of discrepancy $1$ over $\hat{\msq} \in \hat{X}$.
It follows from the existence of at least $5$ divisors of discrepancy $1$ over $\hat{\msq} \in \hat{X}$ and Propositions \ref{prop:ncd2cE7e5}, \ref{prop:ncd2cE7e9} that $\hat{\msq} \in \hat{X}$ admits no divisor of discrepancy greater than $1$.
This completes the proof.
\end{proof}

\subsubsection{Maximal extractions of $\hat{X}$}

\begin{Prop} \label{prop:No101main2}
The divisorial contraction $\hat{\varphi} \colon \hat{\mcY} \to \hat{X}$ is the unique maximal extraction of $\hat{X}$, and $\EL (\hat{X}) = \{\hat{\sigma}^{-1}\}$.
\end{Prop}

\begin{proof}
This follows from Lemmas \ref{lem:exclsmpts}, \ref{lem:exclcurves}, \ref{lem:exclsingpt2}, and \ref{lem:No101divcont}.
\end{proof}

Theorem~\ref{thm:main2}, and hence Theorem~\ref{thm:main}, for Family \textnumero~101 follow from Propositions~\ref{prop:No101main1} and \ref{prop:No101main2}.

\subsection{Family \textnumero 102}

Let $X = X_{26} \subset \mbP (1_x,2_y,5_z,7_t,13_w)$ be a quasi-smooth member of Family \textnumero $102$, which is defined by a quasi-homogeneous polynomial $F = F (x, y, z, t, w)$ of degree $26$, and whose basic information is given as follows:
\begin{itemize}
\item $\iota_X = 2$.
\item $(A^3) = 1/35$.
\item $\Sing (X) = \{\frac{1}{5} (1,1,4), \frac{1}{7} (1,3,4)\}$, and $\msq = \msp_t$ is the $\frac{1}{7} (1, 3, 4)$ point.
\end{itemize}
The elementary link $\hat{\sigma} \colon X \ratmap \hat{X}$ is embedded into the following toric diagram.
\[
\xymatrix{
\text{$\mbT := \mbT \begin{pNiceArray}{cc|cccc}[first-row]
u & t & w & y & z & x \\
0 & 7 & 13 & 2 & 5 & 1 \\
-7 & 0 & 3 & 1 & 6 & 4
\end{pNiceArray}$} \ar@{-->}[r]^{\hat{\Theta}} \ar[d]_{\Phi} &
\text{$\mbT \begin{pNiceArray}{cccc|cc}[first-row]
u & t & w & y & z & x \\
5 & 6 & 9 & 1 & 0 & -2 \\
1 & 4 & 7 & 1 & 2 & 0
\end{pNiceArray} =: \hat{\mbT}$} \ar[d]^{\hat{\Phi}} \\
\mbP (7_t, {13}_w, 2_y, 5_z, 1_x) & \mbP (1_u, 4_t, 7_w, 1_y, 2_z)}
\]
The birational morphisms $\Phi, \hat{\Phi}$, and the birational map $\hat{\Theta}$ are defined as follows:
\[
\begin{split}
\Phi & \colon (u\!:\!t \, | \, w\!:\!y\!:\!z\!:\!x) \mapsto (t\!:\!w u^{3/7}\!:\!y u^{1/7}\!:\!z u^{6/7}\!:\!x u^{4/7}), \\
\hat{\Phi} & \colon (u\!:\!t\!:\!w\!:\!y \, | \,  z\!:\!x) \mapsto (u x^{5/2}\!:\!t x^3\!:\!w x^{9/2}\!:\!y x^{1/2}\!:\!z), \\
\hat{\Theta} & \colon (u\!:\!t \, | \, w\!:\!y\!:\!z\!:\!x) \mapsto (u\!:\!t\!:\!w\!:\!y \, | \,  z\!:\!x).
\end{split}
\]
The varieties $\mcY$ and $\hat{\mcY}$ are the hypersurfaces in $\mbT$ and $\hat{\mbT}$, respectively, which are defined by the same equation
\[
G (u, x, y, z, t, w) := u^{- 6/7} F (u^{4/7} x, u^{1/7} y, u^{6/7} z, t, u^{3/7} w) = 0.
\]
The variety $\hat{X}= \hat{X}_{14} \subset \mbP (1_u, 1_y, 2_z, 4_t, 7_w)$ is the hypersurface defined by
\[
\hat{F} (u, y, z, t, w) := G (u, 1, y, z, t, w) = 0,
\]
where $\hat{F}$ is a quasi-homogeneous polynomial of degree $14$, and we have the following basic information of $\hat{X}$:
\begin{itemize}
\item $\iota_{\hat{X}} = 1$.
\item $(\hat{A}^3) = 1/4$.
\item $\Sing (\hat{X}) = \{\frac{1}{4} (1, 1, 3), cE/2\}$, and $\hat{\msq} = \hat{\msp}_z \in \hat{X}$ is the $cE/2$ point (see Lemma \ref{lem:loceqcENo102}). 
\end{itemize}

\subsubsection{Maximal extractions of $X$}

\begin{Lem} \label{lem:No102sing1}
The $\frac{1}{5} (1,1,4)$ point of $X$ is not a maximal center.
\end{Lem}

\begin{proof}
Let $\msp = \msp_z \in X$ be the $\frac{1}{5} (1,1,4)$ point, and let $\psi \colon Y \to X$ be the Kawamata blow-up of $X$ at $\msp$ with exceptional divisor $E$.
By a choice of homogeneous coordinates and by the quasi-smoothness of $X$, we can write
\[
F = z^5 x + z^4 f_6 + z^3 f_{11} + z^2 f_{16} + z f_{21} + f_{26},
\]
where $f_i = f_i (x, y, t, w)$ is a quasi-homogeneous polynomial of degree $i$.
By the quasi-smoothness of $X$, we may assume that $\coeff_{f_{21}} (t^3) = 1$ after rescaling $t$.

The open subset $(z \ne 0) \subset X$ is naturally isomorphic to the affine hyperquotient
\[
(F (x, y, 1, t, w) = 0)/\mbZ_5 (1_x, 2_y, 2_t, 3_w),
\] 
and the point $\msp$ corresponds to the origin.
By eliminating $x$, we can choose $y, t, w$ as local orbifold coordinates of $X$ at $\msp$, and $\psi$ is the weighted blow-up with $\wt (y, t, w) = \frac{1}{5} (1, 1, 4)$.
Filtering off terms divisible by $x$ in the equation $F (x, y, 1, t, w) = 0$, we have
\[
x (-1 + \cdots) = f_6 (0, y, t, w) + f_{11} (0, y, t, w) + \cdots + f_{26} (0, y, t, w).
\]
The lowest weight of monomials in the right-hand side of the above equation with respect to $\wt (y, t, w) = \frac{1}{5} (1, 1, 4)$ is $\frac{3}{5}$, which is attained by the monomial $t^3 \in f_{21} (0, y, t, w)$.
Hence $\ord_E (x) = \frac{3}{5}$.

We set $S = (x = 0)_X$ and $T = (y = 0)_X$, and let $\tilde{S}, \tilde{T}$ be their proper transforms on $Y$.
Let $E$ be the $\psi$-exceptional divisor.
We have
\[
\begin{split}
\tilde{S} &\sim_{\mbQ} \psi^*A - \frac{3}{5} E \sim \frac{1}{2} (-K_Y) - \frac{1}{2} E, \\ 
\tilde{T} &\sim_{\mbQ} 2 \psi^*A - \frac{1}{5} E \sim -K_Y.
\end{split}
\]
We see that 
\[
\Gamma := S \cap T \cong (w^2 + t^3 z = 0) \subset \mbP (5_z, 7_t, 13_w)
\]
so that $\Gamma$ is an irreducible and reduced curve on $X$.
It is straightforward to see that $\tilde{S} \cap \tilde{T} \cap E$ consists of the unique point (which is the singular point of $E \cong \mbP (1,1,4)$).
This implies that $\tilde{S} \cap \tilde{T} = \tilde{\Gamma}$, where $\tilde{\Gamma}$ is the proper transform of $\Gamma$ on $Y$.
Now we compute
\[
\begin{split}
(\tilde{T} \cdot \tilde{\Gamma}) &= (\tilde{T}^2 \cdot \tilde{S}) = (2 \psi^*A - \frac{1}{5} E)^2 \cdot (\psi^* A - \frac{3}{5} E) \\
&= \frac{2^2}{35} - \frac{3}{5^3} \cdot \frac{5^2}{4} < 0.
\end{split}
\]
By Lemma~\ref{lem:exclsingptG}, $\msp$ is not a maximal center.
\end{proof}

By the above Lemma \ref{lem:No102sing1} combined with \cite[Lemma 4.1 and 4.2]{ACP}, we have the following.

\begin{Prop} \label{prop:No102main1}
The Kawamata blow-up $\varphi \colon \mcY \to X$ at the $\frac{1}{7} (1, 3, 4)$ point $\msq$ is the unique maximal extraction of $X$, and $\EL (X) = \{\hat{\sigma}\}$.
\end{Prop}

\subsubsection{Divisorial contractions centered at $\hat{\msq} \in \hat{X}$}

We denote by $\msw$ the weight on $u, y$ defined by $\msw (u, y) = \frac{1}{2} (3, 1)$.

\begin{Lem} \label{lem:loceqcENo102}
The singularity $\hat{\msq} \in \hat{X}$ is of type $cE/2$, and it is equivalent to $(\phi (w, t, u, y) = 0)/\mbZ_2 (1_w, 0_t, 1_u, 1_y)$, where
\[
\phi (w, t, u, y) = w^2 + t^3 + t \tilde{g} (u, y) + \tilde{h} (u, y),
\]
for some $\tilde{g} (u, y), \tilde{h} (u, y) \in \mbC \{u, y\}$ that are  invariant under the $\mbZ_2$-action and that satisfy the following properties.
\begin{enumerate}
\item The degree $4$ part of $\tilde{h}$ is $u^4$.
\item $\msw (\tilde{g}) \ge 4$ and $\msw (\tilde{h}) = 6$.
\item $u^2 \mid \tilde{g}_{\msw = 4}$ and $u \mid \tilde{g}_{\msw = 5}$.
\item $u^3 \mid \tilde{h}_{\msw = 6}$, $u^2 \mid \tilde{h}_{\msw = 7}$, and $u \mid \tilde{h}_{\msw = 8}$.
\end{enumerate} 
\end{Lem}

\begin{proof}
By a choice of homogeneous coordinates of $x, y, z, t, w$ of $\mbP (1, 2, 5, 7, 13)$, we can write
\[
F = w^2 + t^3 z + t^2 e_{12} + t e_{19} + e_{26},
\]
where $e_i = e_i (x, y, z)$ is a quasi-homogeneous polynomial of degree $i$.
By the quasi-smoothness of $X$, we have $z^5 x \in e_{26}$.
Hence we may assume that $\coeff_{e_{26}} (z^5 x) = 1$ by rescaling $x$.
Then we have
\[
\begin{split}
\hat{F} &= u^{-6/7} F (u^{4/7}, u^{1/7} y, u^{6/7} z, t, u^{3/7} w) \\
&= w^2 + t^3 z + 3 t^2 \hat{e}_{12} (u, y, z) + t \hat{e}_{19} (u, y, z) + \hat{e}_{26} (u, y, z),
\end{split}
\]
where
\[
\hat{e}_i = \hat{e}_i (u, y, z) := u^{-6/7} e_i (u^{4/7}, u^{1/7} y, u^{6/7} z)
\]
for $i = 6, 19, 26$.
We set 
\[
f (u, y) := \hat{e}_{12} (u, y, 1), \ 
g (u, y) := \hat{e}_{19} (u, y, 1), \ 
h (z, u) := \hat{e}_{26} (u, y, 1),
\]
By setting $z = 1$, $\hat{\msq} \in \hat{X}$ is equivalent to 
\begin{equation} \label{eq:loceqcENo102}
o \in (w^2 + t^3 + 3 t^2 f + t g + h = 0)/\mbZ_2 (1, 0, 1, 1).
\end{equation} 
Note that $\coeff_{h} (u^4) = 1$.
It can be checked that
\begin{enumerate}
\item The degree $4$ part of $h$ is $u^4$.
\item $\msw (f) \ge 2$, $\msw (g) \ge 4$, and $\msw (h) \ge 6$.
\item $u \mid f_{\msw = 2}$.
\item $u^2 \mid g_{\msw = 4}$ and $u \mid g_{\msw = 5}$.
\item $u^3 \mid h_{\msw = 6}$, $u^2 \mid h_{\msw = 7}$, and $u \mid h_{\msw = 8}$.
\end{enumerate}
By the coordinate change $y \mapsto y - f$, the germ \eqref{eq:loceqcENo102} is equivalent to the germ $o \in (\phi (w, t, u, y) = 0)/\mbZ_2 (1_w, 0_t, 1_u, 1_y)$, where
\[
\phi = x^2 + y^3 + y \tilde{g} (z, u) + \tilde{h} (z, u)
\]
with
\[
\tilde{g} = g - 3 f^2, \quad
\tilde{h} = h - f^3 - f g.
\]
It is then straightforward to see that $\hat{\msq} \in \hat{X}$ is of type $cE/2$ and the conditions (1)--(5) in the statement are satisfied.
\end{proof}

\begin{Lem} \label{lem:No102divcont}
The morphism $\hat{\varphi} \colon \hat{\mcY} \to \hat{X}$ is the unique divisorial contraction centered at the $cE/2$ point $\hat{\msq} \in \hat{X}$.
\end{Lem}

\begin{proof}
We identify $\hat{\msq} \in \hat{X}$ with the germ 
\[
(\phi (w, t, u, y) = 0)/\mbZ_2 (1_w, 0_t, 1_u, 1_y)
\]
given in Lemma~\ref{lem:loceqcENo102}.
We will show that the singularity $\hat{\msq} \in \hat{X}$ corresponds to Case (F) of \cite[Section 7.8]{Hayakawa05b}.

We write
\[
\tilde{g} (u, y) = \sum_{i, j} a_{i j} u^i y^j, \quad
\tilde{h} (u, y) = \sum_{i, j} b_{i j} u^i y^j,
\]
where $a_{i j}, b_{i j} \in \mbC$.
Then 
\[
a_{04} = a_{13} = a_{06} = b_{06} = b_{15} = b_{08} = 0.
\]
It follows that the system of equations
\[
x^3 + a_{04} x + b_{06} = 0, \ 
3 x^2 + a_{04} = 0, \ 
a_{13} x + b_{15} = 0, \ 
a_{06} x + b_{08} = 0,
\]
has a solution $x = 0$.
For the polynomial
\[
\Phi := x^2 + y g_{\msw = 3} + h_{\msw = 5},
\]
we have $\Phi = x^2$.
Note that $\msw (\tilde{g}) \ge 4$ and $\msw (\tilde{h}) \ge 6$.
Hence it remains to check that the weighted blow-up $\psi \colon W \to \hat{X}$ at $\hat{\msq}$ with $\wt (w, t, u, y) = \frac{1}{2} (7, 4, 3, 1)$ has a non-terminal singularity along its exceptional divisor.
We consider the $y$-chart $W_y$ that is isomorphic to the hypersurface $(\phi_y = 0) \subset \mbA^4_{w, t, u, y}$, where
\[
\phi_y (w, t, u, y) := \phi (w y^{7/2}, t y^{4/2}, u y^{3/2}, y^{1/2})/y^7.
\]
We have
\[
\begin{split}
\phi_y = w^2 y + t^3 + t (\tilde{g}^{[\tau]}_4 (u, 1) &+ y \tilde{g}^{[\tau]}_5 (u, 1)) + \\
+ \tilde{h}^{[\tau]}_7 (u, 1) &+ y \tilde{h}^{[\tau]}_8 (u, 1) + y^2 \tilde{h}^{[\tau]}_9 (u, 1) + \cdots,
\end{split}  
\]
where the omitted term is the sum of monomials divisible by $y^3$.
It follows that $W$ has multiplicity $3$ at the origin, and thus $W$ has a non-terminal singularity along the exceptional divisor.
This shows that $\hat{\msq} \in \hat{X}$ corresponds to Case (F) of \cite[Section 7.8]{Hayakawa05b} (Note also that $\hat{\msq} \in \hat{X}$ corresponds to (10.E.e) in \cite[10.71]{Hayakawa99}).
By the proof of \cite[Theorem 7.9]{Hayakawa05b} (see (6) of the proof), $\hat{\msq} \in \hat{X}$ admits a unique divisorial contraction of discrepancy $1/2$ and it does not admit a divisorial contraction of discrepancy greater than $1/2$.
\end{proof}

\subsubsection{Maximal extractions of $\hat{X}$}

\begin{Lem} \label{lem:No102curves}
No curve on $\hat{X}$ is a maximal center.
\end{Lem}

\begin{proof}
We are unable to apply Lemma \ref{lem:exclcurves} since $\hat{X}$ is not quasi-smooth at the point $\hat{\msq} = \hat{\msp}_z \in \hat{X}$ at which the ambient weighted projective space is singular.

Let $\Gamma \subset \hat{X}$ be an irreducible and reduced curve.
Then we have 
\[
(\hat{A} \cdot \Gamma) \ge \frac{1}{4} = (\hat{A})^3,
\]
where the first inequality follows since $4 \hat{A}$ is a Cartier divisor on $\hat{X}$. 
By Lemma~\ref{lem:mtdexclC}, $\Gamma$ is not a maximal center.
\end{proof}

\begin{Prop} \label{prop:No102main2}
The divisorial contraction $\hat{\varphi} \colon \hat{\mcY} \to \hat{X}$ centered at the $cE/2$ point $\hat{\msq}$ is the unique maximal extraction of $\hat{X}$, and $\EL (\hat{X}) = \{\hat{\sigma}^{-1}\}$.
\end{Prop}

\begin{proof}
This follows from Lemmas \ref{lem:exclsmpts}, \ref{lem:No102curves}, \ref{lem:exclqsing1}, and \ref{lem:No102divcont}.
\end{proof}

Theorem~\ref{thm:main2}, and hence Theorem~\ref{thm:main}, for Family \textnumero~102 follow from Propositions~\ref{prop:No102main1} and \ref{prop:No102main2}.

\subsection{Family \textnumero 103}

Let $X = X_{38} \subset \mbP (2_x, 3_y, 5_z, 11_t, 19_w)$ be a quasi-smooth member of Family \textnumero $103$, which is defined by a quasi-homogeneous polynomial $F = F (x, y, z, t, w)$ of degree $38$, and whose basic information is given as follows:
\begin{itemize}
\item $\iota_X = 2$.
\item $(A^3) = 1/165$.
\item $\Sing (X) = \{\frac{1}{3} (1,1,2), \frac{1}{5} (1, 2, 3), \frac{1}{11} (1, 4, 7)\}$, and $\msq =\msp_t$ is the $\frac{1}{11} (1, 4, 7)$ point.
\end{itemize}
The elementary link $\hat{\chi} \colon X \ratmap \hat{X}$ is embedded into the following toric diagram.
\[
\xymatrix{
\text{$\mbT := \mbT \begin{pNiceArray}{cc|cccc}[first-row]
u & t & w & x & z & y \\
0 & 11 & 19 & 2 & 5 & 3 \\
-11 & 0 & 4 & 1 & 8 & 7
\end{pNiceArray}$} \ar@{-->}[r]^{\hat{\Theta}} \ar[d]_{\Phi} &
\text{$\mbT \begin{pNiceArray}{cccc|cc}[first-row]
u & t & w & x & z & y \\
5 & 8 & 12 & 1 & 0 & -1 \\
3 & 7 & 11 & 1 & 1 & 0
\end{pNiceArray} =: \hat{\mbT}$} \ar[d]^{\hat{\Phi}} \\
\mbP ({11}_t, {19}_w, 2_x, 5_z, 3_y) & \mbP (3_u, 7_t, {11}_w, 1_x, 1_z)}
\]
The birational morphisms $\Phi, \hat{\Phi}$, and the birational map $\hat{\Theta}$ are defined as follows:
\[
\begin{split}
\Phi & \colon (u\!:\!t \, | \, w\!:\!x\!:\!z\!:\!y) \mapsto (t\!:\!w u^{4/11}\!:\!x u^{1/11}\!:\!z u^{8/11}\!:\!y u^{7/11}), \\
\hat{\Phi} & \colon (u\!:\!t\!:\!w\!:\!x \, | \,  z\!:\!y) \mapsto (u y^5\!:\!t y^8\!:\!w y^{12}\!:\!x y\!:\!z), \\
\hat{\Theta} & \colon (u\!:\!t \, | \, w\!:\!x\!:\!z\!:\!y) \mapsto (u\!:\!t\!:\!w\!:\!x \, | \,  z\!:\!y).
\end{split}
\]
The varieties $\mcY$ and $\hat{\mcY}$ are the hypersurfaces in $\mbT$ and $\hat{\mbT}$, respectively, which are defined by the same equation
\[
G (u, x, y, z, t, w) := u^{- 8/11} F (u^{1/11} x, u^{7/11} y, u^{8/11} z, t, u^{4/11} w) = 0.
\]
The variety $\hat{X} = \hat{X}_{22} \subset \mbP (1_x, 1_z, 3_u, 7_t, {11}_w)$ is the hypersurface defined by
\[
\hat{F} (x, z, u, t, w) := G (u, x, 1, z, t, w) = 0.
\]
where $\hat{F}$ is a quasi-homogeneous polynomial of degree $22$, and we have the following basic information of $\hat{X}$:
\begin{itemize}
\item $\iota_{\hat{X}} = 1$.
\item $(\hat{A}^3) = 2/21$.
\item $\Sing (\hat{X}) = \{\frac{1}{3} (1, 1, 2), \frac{1}{7} (1, 3, 4), cE_8\}$, and $\hat{\msq} = \hat{\msp}_z \in \hat{X}$ is the $cE_8$ point (see Lemma \ref{lem:loceqcENo103}).
\end{itemize}

\subsubsection{Maximal extractions of $X$}

\begin{Lem} \label{lem:No103sing1}
The $\frac{1}{3} (1,1,2)$ point of $X$ is not a maximal center.
\end{Lem}

\begin{proof}
Let $\msp = \msp_y \in X$ be the $\frac{1}{3} (1, 1, 2)$ point, and let $\psi \colon Y \to X$ be the Kawamata blow-up of $X$ at $\msp$ wth exceptional divisor $E$.
By a choice of homogeneous coordinates and by the quasi-smoothness of $X$, we can write
\[
F =
\begin{cases}
y^9 t + y^8 f_{14} + \cdots + y f_{35} + f_{38}, & \text{if $y^9 t \in F$}, \\
y^{11} z + y^{10} f_8 + \cdots + y f_{35} + f_{38}, & \text{if $y^9 t \notin F$ and $y^{11} z \in F$}, \\
y^{12} x + y^{11} f_5 + \cdots + y f_{35} + f_{38}, & \text{if $y^9 t, y^{11} z \notin F$},
\end{cases}
\]
where $f_i = f_i (x, z, t, w)$ is a quasi-homogeneous polynomial of degree $i$.
The open set $(y \ne 0) \subset X$ is naturally isomorphic to the affine hyperquotient
\[
(F (x, 1, z, t, w) = 0)/\mbZ_3 (2_x, 2_z, 2_t, 1_w),
\]
and the point $\msp$ corresponds to the origin.
The Kawamata blow-up $\psi$ is the weighted blow-up with the following weight:
\[
\begin{cases}
\wt (x, z, w) = \frac{1}{3} (1, 1, 2), & \text{if $y^9 t \in F$}, \\
\wt (x, t, w) = \frac{1}{3} (1, 1, 2), & \text{if $y^9 t \notin F$ and $y^{11} z \in F$}, \\
\wt (z, t, w) = \frac{1}{3} (1, 1, 2), & \text{if $y^9 t, y^{11} z \notin F$}.
\end{cases}
\]
Moreover, we have
\[
\begin{cases}
\ord_E (t) = \frac{4}{3}, & \text{if $y^9 t \in F$}, \\
\ord_E (z) = \frac{4}{3}, & \text{if $y^9 t \notin $ and $y^{11} z \in F$}, \\
\ord_E (x) = \frac{4}{3}, & \text{if $y^9 t, y^{11} z \notin F$}.
\end{cases}
\]

We first consider the case where either $y^9 t \in F$ or $y^9t, y^{11} z \notin F$.
We set
\[
\begin{cases}
\text{$S := (x = 0)_X$ and $T := (t = 0)_X$}, & \text{if $y^9 t \in F$}, \\
\text{$S := (x = 0)_X$ and $T := (z = 0)_X$}, & \text{if $y^9 t, y^{11} z \notin F$}.
\end{cases}
\]
Then we have
\[
\begin{cases}
\tilde{S} \sim_{\mbQ} \psi^*A - \frac{1}{3} E, \ \tilde{T} \sim_{\mbQ} 11 \psi^* A - \frac{4}{3} E, & \text{if $y^9 t \in F$}, \\
\tilde{S} \sim_{\mbQ} \psi^*A - \frac{4}{3} E, \ \tilde{T} \sim_{\mbQ} 5 \psi^*A - \frac{1}{3} E, & \text{if $y^9 t, y^{11} z \notin F$}.
\end{cases}
\]
We see that
\[
S \cap T \cong
\begin{cases}
(\alpha y^6 z^4 + \beta y^3 w z^2 + y z^7 + w^2 = 0) \subset \mbP (3_y, 5_z, 19_w), & \text{if $y^9 t \in F$}, \\
(w^2 = 0) \subset \mbP (3_y, 11_t, 19_w), & \text{if $y^9 t \notin F$},
\end{cases}
\]
where $\alpha, \beta \in \mbC$.
Let $\Gamma$ be the support of $S \cap T$, which is an irreducible and reduced curve on $X$.
Note that we have $S \cap T = m \Gamma$, where
\[
m = 
\begin{cases}
1, & \text{if $y^9 t \in F$}, \\
2, & \text{if $y^9 t \notin F$}.
\end{cases}
\]
It is straightforward to check that $\tilde{S} \cap \tilde{T} \cap E$ consists of $2$ points (counting with multiplicity).
 It follows that $\tilde{S} \cap \tilde{T} = m \tilde{\Gamma}$, where $\tilde{\Gamma}$ is the proper transform of $\Gamma$.
We compute
\[
m (\tilde{T} \cdot \tilde{\Gamma}) = (\tilde{T}^2 \cdot \tilde{S}) =
\begin{cases}
- \frac{6}{5}, & \text{if $y^9 t \in F$}, \\
- \frac{4}{11}, & \text{if $y^9 t, y^{11} t \notin F$}.
\end{cases}
\]
By Lemma~\ref{lem:exclsingptG}, $\msp$ is not a maximal center.

We next consider the case where $y^9 t \notin F$ and $y^{11} z \in F$.
We set $\mcL:= |22 A|$, $S := (z = 0)_X $, and let $L \in \mcL$ be a general member.
Then we have
\[
\tilde{S} \sim 5 \psi^*A - \frac{4}{3} E, \ 
\tilde{L} \sim 22 \psi^*A - \frac{2}{3} E.
\]
We compute
\[
\begin{split}
(-K_Y \cdot \tilde{S} \cdot \tilde{L}) &= (2 \psi^*A' - \frac{1}{3} E) \cdot (5 \psi^*A - \frac{4}{3} E) \cdot (22 \psi^*A - \frac{2}{3} E) \\
&= \frac{2 \cdot 5 \cdot 22}{165} - \frac{4 \cdot 2}{3^3} \cdot \frac{3^2}{2} = 0.
\end{split}
\]
By Lemma \ref{lem:exclquotsing}, $\msp$ is not a maximal center.
\end{proof}

\begin{Lem} \label{lem:No103sing2}
The $\frac{1}{5} (1,2,3)$ point of $X$ is not a maximal center.
\end{Lem}

\begin{proof}
Let $\msp = \msp_z$ be the $\frac{1}{5} (1,2,3)$ point, and let $\psi \colon Y \to X$ be the Kawamata blow-up of $X$ at $\msp$ with exceptional divisor $E$.
By a choice of homogeneous coordinates and by the quasi-smoothness of $X$, we can write
\[
F = z^7 y + z^6 f_{8} + z^5 f_{13} + z^4 f_{18} + z^3 f_{23} + z^2 f_{28} + z f_{33} + f_{38},
\]
where $f_i = f_i (x, y, t, w)$ is a quasi-homogeneous polynomial of degree $i$.
The open set $(z \ne 0) \subset X$ is naturally isomorphic to the affine hyperquotient
\[
(F (x, y, 1, t, w) = 0)/\mbZ_5 (2_x, 3_y, 1_t, 4_w),
\]
and the point $\msp$ corresponds to the origin.
By eliminating $y$, we can take $x, t, w$ as local orbifold coordinates of $X$ at $\msp$, and $\psi$ is locally the weighted blow-up with weight $\wt (x, t, w) = \frac{1}{5} (1, 3, 2)$.
Filtering off terms divisible by $z$ in the equation $F (x, y, 1, t, w) = 0$, we obtain
\[
(-1 + \cdots) y = f_8 (x, 0, t, w) + f_{13} (x, 0, t, w) + \cdots + f_{38} (x, 0, t, w).
\]
The lowest weight of the monomials in the right-hand side of the above equation with respect to $\wt (x, y, t) = \frac{1}{5} (1, 3, 2)$ is $\frac{4}{5}$, which is attained by the monomial $w^2 \in f_{38} (x, 0, t, w)$.
Hence we have $\ord_E (y) = \frac{4}{5}$.

We set $S := (y= 0)_X$ and $T = (x = 0)_X$, and let $\tilde{S}, \tilde{T}$ be their proper transforms on $Y$.
We have 
\[
\begin{split}
\tilde{S} &\sim_{\mbQ} 3 \psi^*A - \frac{4}{5} E \sim \frac{3}{2} (-K_Y) - \frac{1}{2} E, \\ 
\tilde{T} &\sim_{\mbQ} 2 \psi^*A - \frac{1}{5} E \sim -K_Y,
\end{split}
\]
where $E$ is the $\psi$-exceptional divisor.
We have
\[
\Gamma := S \cap T \cong (w^2 + t^3 z = 0) \subset \mbP (3_z,5_t,9_w)
\]
and $\Gamma$ is an irreducible and reduced curve.
Since $\tilde{S} \cap \tilde{T} \cap E$ does not contain a curve (in fact, it consists of the singular point of $E \cong \mbP (1,2,3)$ of typpe $\frac{1}{3} (1,2)$), we see that $\tilde{S} \cap \tilde{T} = \tilde{\Gamma}$, where $\tilde{\Gamma}$ is the proper transform of $\Gamma$ on $Y$.
We compute
\[
\begin{split}
(\tilde{T} \cdot \tilde{\Gamma}) &= (\tilde{T}^2 \cdot \tilde{S}) = (2 \psi^* A - \frac{1}{5} E)^2 \cdot (3 \psi^* A - \frac{4}{5} E) \\
&= \frac{2^2 \cdot 3}{165} - \frac{4}{5^3} \cdot \frac{5^2}{2 \cdot 3} < 0.
\end{split}
\]
Thus, by Lemma~\ref{lem:exclsingptG}, $\msp$ is not a maximal center.
\end{proof}

By the above Lemmas \ref{lem:No103sing1} and \ref{lem:No103sing2} combined with \cite[Lemmas 4.1 and 4.2]{ACP}, we have the following.

\begin{Prop} \label{prop:No103main1}
The Kawamata blow-up $\varphi \colon \mcY \to X$ at the $\frac{1}{11} (1, 4, 7)$ point $\msq \in X$ is the unique maximal extraction of $X$, and $\EL (X) = \{\hat{\sigma}\}$.
\end{Prop}

\subsubsection{Divisorial contractions centered at $\hat{\msq} \in \hat{X}$}

For a positive integer $i$, we denote by $\msw_i$ the weight on $u, x$ defined by $\msw_i (u, x) = (i, 1)$.

\begin{Lem} \label{lem:loceqcENo103}
The singularity $\hat{\msq} \in \hat{X}$ is of type $cE_8$, and it is equivalent to $(\phi = 0) \subset \mbC^4_{u, y, t, w}$, where
\begin{equation} \label{eq:cEeqNo103}
\phi (w, t, u, x) = w^2 + t^3 + t^2 f + t g + h,
\end{equation}
for some $f = f (u, x), g = g (u, x), h = h (u, x) \in \mbC [u, x]$ satisfying the following properties.
\begin{enumerate}
\item $\msw_2 (f) \ge 4$, $\msw_2 (g) \ge 7$, and $\msw_2 (h) \ge 10$.
\item $u^3 \mid g_{\msw_2 = 7}$.
\item $h_{\msw_2 = 10} = u^5$, and $u^5 \mid h_{\msw_2 = 11}$.
\item $\msw_3 (f) \ge 6$, $\msw_3 (g) \ge 10$, and $\msw_3 (h) \ge 15$.
\item $u^3 \mid g_{\msw_3 = 10}$ and $u^3 \mid g_{\msw_3 = 11}$.
\item $h_{\msw_3 = 15} = u^5$, $u^4 \mid h_{\msw_3 = 16}$, and $u^4 \mid h_{\msw_3 = 17}$.
\item $\msw_4 (f) \ge 7$, $\msw_4 (g) \ge 14$, $\msw_4 (h) \ge 20$.
\item $u \mid f_{\msw_4 = 7}$.
\item $u^2 \mid g_{\msw_4 = 14}$ and $u \mid g_{\msw_4 = 15}$.
\item $u^4 \mid h_{\msw_4 = 20}$, $u^3 \mid h_{\msw_4 = 21}$, $u^2 \mid h_{\msw_4 = 22}$ and $u \mid h_{\msw_4 = 23}$.
\end{enumerate}
\end{Lem}

\begin{proof}
By a choice of homogeneous coordinates $x, y, z, t, w$ of $\mbP (2, 3, 5, 11, 19)$, we can write
\[
F = w^2 + t^3 z + t^2 e_{16} (x, y, z) + t e_{27} (x, y, z) + e_{38} (x, y, z),
\]
where $e_i = e_i (x, y, z)$ is a quasi-homogeneous polynomial of degree $i$.
By the quasi-smoothness of $X$ at $\msp_z$, we have $z^7 y \in e_{38}$.
By re-scaling $y$, we may assume that $\coeff_{e_{38}} (z^7 y) = 1$.
For the defining polynomial $\hat{F} = \hat{F} (u, x, z, t, w)$ of $\hat{X}$, we have
\[
\begin{split}
\hat{F} &= u^{-8/11} F (u^{1/11} x, u^{7/11}, u^{8/11} z, t, u^{4/11} w) \\
&= w^2 + t^3 + t^2 \hat{e}_{16} (u, x, z) + t \hat{e}_{27} (u, x, z) + \hat{e}_{38} (u, x, z),
\end{split}
\]
where
\[
\hat{e}_i = \hat{e}_i (u, x, z) := u^{-8/11} e_i (u^{1/11} x, u^{7/11}, u^{8/11} z)
\]
for $i = 16, 27, 38$.
By setting $z = 1$, the germ $\hat{\msp} \in \hat{X}$ is naturally isomorphic to the origin of the hypersurface germ in $\mbA^4_{w, t, u, x}$ defined by $\phi = 0$, where
\[
\phi := \hat{F} (u, x, 1, t, w) = w^2 + t^3 + t^2 \hat{e}_{16} (u, x, 1) + \hat{e}_{27} (u, x, 1) + \hat{e}_{38} (u, x, 1).
\]
Note that $\coeff_{\hat{e}_{38}} (u^5) = 1$.
By setting
\[
f := \hat{e}_{16} (u, x, 1), g := \hat{e}_{27} (u, x, 1), h := \hat{e}_{38} (u, x, 1),
\]
it is straightforward to check that the conditions (1)--(10) are all satisfied.
Finally, since we have $\msw_1 (g) \ge 4$ and $\msw_1 (h) \ge 5$ by (1), it follows that the singularity $\hat{\msq} \in \hat{X}$ is of type $cE_8$.
\end{proof}

\begin{Lem} \label{lem:boundNo103}
There are at most $7$ divisors of discrepancy $1$ over $\hat{\msq} \in \hat{X}$.
\end{Lem}

\begin{proof}
Under the identification of $\hat{\msq} \in \hat{X}$ with $(\phi = 0) \subset \mbA^4_{w, t, u, x}$, the morphism $\hat{\varphi} \colon \hat{\mcY} \to \hat{X}$ is the weighted blow-up with $\wt (w, t, u, x) = (12, 8, 5, 1)$.
Let $\msw$ be the weight on $w, t, u, x$ such that $\msw (w, t, u, x) = (12, 8, 5, 1)$.
The $\varphi$-exceptional divisor is isomorphic to
\[
(\phi_{\msw = 24} = 0) \subset \mbP (12_w, 8_t, 5_u, 1_x).
\]
We see that the non-Gorenstein singularities of $\mcY'$ along $\mcE'$ consist of $1$ point of type $\frac{1}{4} (1, 1, 3)$ and $1$ point of type $\frac{1}{5} (1, 2, 3)$.
Hence there are at most $8$ divisors of discrepancy $1$ over $\hat{\msq} \in \hat{X}$ including $\hat{\mcE}$.

Let $\msp \in \hat{\mcY}$ be the $\frac{1}{5} (1, 2, 3)$ point and let $G$ be the  divisor of discrepancy $4/5$ over $\msp \in \hat{\mcY}$.
The $u$-chart $\hat{\mcY}_u$ is the hyperquotient
\[
(\phi (w u^{12}, t u^8, u^5, x u)/u^{24} = 0)/\mbZ_5 (12_w, 8_t, -1_u, 1_x).
\] 
Consider the weight $\msw_5$ on $u, x$ defined by $\msw_5 (u, x) = (5, 1)$.
Then, by filtering off terms divisible by $u$, we can write
\[
\begin{split}
& \phi (w u^{12},  t u^8, u^5, x u)/u^{24} \\
&= u (1 + \theta) + w^2 + t^3 + t^2 f_{\msw_5 = 8} (1, x) + t g_{\msw_5 = 16} (1, x) + h_{\msw_5 = 24} (1, x),
\end{split}
\]
where $\theta = \theta (w, t, u, x)$ vanishes at the origin.
Note that $\msp$ corresponds to the origin of $\hat{\mcY}_u$, and the divisor $\hat{\mcE}$ is cut out by $u$ on $\hat{\mcY}_u$.
In a neighborhood of $\msp \in \hat{\mcY}_u$, we can take $w, t, x$ as local orbifold coordinates of $\hat{\mcY}$, and $\hat{\mcE}$ is defined by 
\[
w^2 + t^3 + t^2 f_{\msw_5 = 8} (1, x) + t g_{\msw_5 = 16} (1, x) + h_{\msw_5 = 24} (1, x).
\]
We see that $G$ is the exceptional divisor of the weighted blow-up of $\hat{\mcY}$ at $\msp$ with $\wt (w, t, x) = \frac{1}{5} (3, 2, 4)$.
We have $a_G (K_{\hat{\mcY}}) = 4/5$ and $\ord_G (\hat{\mcE}) = 6/5$ so that
\[
a_G (K_{\hat{X}}) = a_G (K_{\hat{\mcY}}) + \ord_G (\hat{\mcE}) = 2.
\]
This shows that there are at most $7$ divisors of discrepancy $1$ over $\hat{\msq} \in \hat{X}$.
\end{proof}

\begin{Lem} \label{lem:No103divcont}
The morphism $\hat{\varphi} \colon \hat{\mcY} \to \hat{X}$ is the unique divisorial contraction centered at the $cE$ point $\hat{\msq} \in \hat{X}$.
\end{Lem}

\begin{proof}
We define various weighted blow-ups of $\hat{X}$ centered at $\hat{\msq}$ as follows.
\begin{itemize}
\item $\psi_1 \colon W_1 \to \hat{X}$ is the weighted blow-up with $\wt (w, t, u, x) = (3, 2, 2, 1)$ which is described in Lemma \ref{lem:divdisc1cEv1}(1).
\item $\psi_2 \colon W_2 \to \hat{X}$ is the weighted blow-up with $\wt (w, t, u, x) = (5, 4, 2, 1)$ which is described in Lemma \ref{lem:divdisc1cEv1}(3).
\item $\psi_3 \colon W_3 \to \hat{X}$ is the weighted blow-up with $\wt (w, t, u, x) = (6, 4, 3, 1)$ which is described in Lemma \ref{lem:divdisc1cEv1}(4).
\item $\psi^{\pm}_4 \colon W^{\pm}_4 \to \hat{X}$ is the weighted blow-up with $\wt (w, t, u, x) = (8, 5, 3, 1)$ which is described in Lemma \ref{lem:divdisc1cEv1}(5).
\item $\psi_5 \colon W_5 \to \hat{X}$ is the weighted blow-up with $\wt (w, t, u, x) = (9, 6, 4, 1)$ which described in Lemma \ref{lem:divdisc1cEv1}(6).
\item $\psi_6 \colon W_6 \to \hat{X}$ is the weighted blow-up with $\wt (w, t, u, x) = (10, 7, 4, 1)$ which is described in Lemma \ref{lem:divdisc1cEv1}(7).
\end{itemize}
Let $E_1, \dots, E_6$ be the exceptional divisors of $\psi_1, \dots, \psi_6$, respectively.
By Lemmas~\ref{lem:loceqcENo103} and \ref{lem:divdisc1cEv1}, each of them is a divisor of discrepancy $1$ over $\hat{\msq} \in \hat{X}$ and none of them can be realized by the exceptional divisor of a divisorial contraction.
Thus $\hat{\varphi}$ is the unique divisorial contraction of discrepancy $1$ over $\hat{\msq} \in \hat{X}$ by Lemma~\ref{lem:boundNo103}.
It follows from the existence of $7$ divisors of discrepancy $1$ over the point $\hat{\msq} \in \hat{X}$ of type $cE_8$ and Proposition~\ref{prop:ncd2cE8e9} that $\hat{\msq} \in \hat{X}$ admits no divisor of discrepancy greater than $1$.
\end{proof}

\subsubsection{Maximal extractions of $\hat{X}$}

\begin{Prop} \label{prop:No103main2}
The divisorial contraction $\hat{\varphi} \colon \hat{\mcY} \to \hat{X}$ centered at the $cE$ point $\hat{\msq}$ is the unique maximal extraction of $\hat{X}$, and $\EL (\hat{X}) = \{\hat{\sigma}^{-1}\}$.
\end{Prop}

\begin{proof}
This follows from Lemmas \ref{lem:exclsmpts}, \ref{lem:exclcurves}, \ref{lem:exclqsing1}, \ref{lem:exclsingpt4}, and \ref{lem:No103divcont}.
\end{proof}

Theorem~\ref{thm:main2}, and hence Theorem~\ref{thm:main}, for Family \textnumero~103 follow from Propositions~\ref{prop:No103main1} and \ref{prop:No103main2}.

\section{Family \textnumero 110} \label{sec:tririg}
Let $X = X_{21} \subset \mbP (1_x, 3_y, 5_z, 7_t, 8_w)$ be a quasi-smooth member of Family \textnumero $110$, which is defined by a quasi-homogeneous polynomial $F = F (x, y, z, t, w)$ of degree $21$, and whose basic information is given as follows:
\begin{itemize}
\item $\iota_X = 3$.
\item $(A^3) = 1/40$.
\item $\Sing (X) = \{\frac{1}{5} (1,1,4), \frac{1}{8} (1, 3, 5)\}$, and $\msq = \msp_w$ is the $\frac{1}{8} (1, 3, 5)$ point.
We set $\msr := \msp_z \in X$ which  is the $\frac{1}{5} (1, 1, 4)$ point.
\end{itemize}

\subsection{Elementary links}

Unlike the $4$ families treated in the previous section, $X$ admits two elementary links $\hat{\sigma} \colon X \ratmap \hat{X}$ and $\breve{\sigma} \colon X \ratmap \breve{X}$ to Mori-Fano $3$-folds $\hat{X}$ and $\breve{X}$.
These two links together with an elementary link $\rho \colon \breve{X} \ratmap \hat{X}$ form a triangle of elementary links:
\[
\xymatrix{
& \ar@{-->}[ld]_{\hat{\sigma}} X \ar@{-->}[rd]^{\breve{\sigma}} \ & \\
\hat{X} & & \ar@{-->}[ll]^{\rho} \breve{X}}
\]
The aim of this subsection is to explain the links $\hat{\sigma}$, $\breve{\sigma}$, and $\rho$.

\subsubsection{The link $\hat{\sigma} \colon X \ratmap \hat{X}$ centered at the $\frac{1}{8} (1, 3, 5)$ point $\msq$}

The elementary link $\hat{\sigma} \colon X \ratmap \hat{X}$ is embedded into the following toric diagram.
\[
\xymatrix{
\text{$\mbT := \mbT \begin{pNiceArray}{cc|cccc}[first-row]
u & w & y & t & z & x \\
0 & 8 & 3 & 7 & 5 & 1 \\
-8 & 0 & 1 & 5 & 7 & 3
\end{pNiceArray}$} \ar@{-->}[r]^{\hat{\Theta}} \ar[d]_{\Phi} &
\text{$\mbT \begin{pNiceArray}{cccc|cc}[first-row]
u & w & y & t & z & x \\
5 & 7 & 2 & 3 & 0 & -1 \\
1 & 3 & 1 & 2 & 1 & 0
\end{pNiceArray} =: \hat{\mbT}$} \ar[d]^{\hat{\Phi}} \\
\mbP (8_w, 3_y, 7_t, 5_z, 1_x) & \mbP (1_u, 3_w, 1_y, 2_t, 1_z)}
\]
The morphisms $\Phi, \hat{\Phi}$, and the rational map $\hat{\Theta}$ are defined as follows:
\[
\begin{split}
\Phi & \colon (u\!:\!w \, | \, y\!:\!t\!:\!z\!:\!x) \mapsto (w\!:\!y u^{1/8} \!:\!t u^{5/8}\!:\!z u^{7/8}\!:\!x u^{3/8}), \\
\hat{\Phi} & \colon (u\!:\!w\!:\!y\!:\!t \, | \,  z\!:\!x) \mapsto (u x^5\!:\!w x^7\!:\!y x^2\!:\!t x^3\!:\!z), \\
\hat{\Theta} & \colon (u\!:\!w \, | \, y\!:\!t\!:\!z\!:\!x) \mapsto (u\!:\!w\!:\!y\!:\!t \, | \,  z\!:\!x).
\end{split}
\]  
The varieties $\mcY$ and $\hat{\mcY}$ are the hypersurfaces in $\mbT$ and $\hat{\mbT}$, respectively, which are defined by the same equation
\[
G (u, x, y, z, t, w) := u^{- 7/8} F (u^{3/8} x, u^{1/8} y, u^{7/8} z, u^{5/8} t, w) = 0.
\]
The variety $\hat{X} = \hat{X}_7 \subset \mbP (1_u, 1_y, 1_z, 2_t, 3_w)$ is the hypersurface defined by 
\begin{equation} \label{eq:110defeqhat}
\hat{F} (u, y, z, t, w) := G (u, 1, y, z, t, w) = 0,
\end{equation}
where $\hat{F}$ is a quasi-homogeneous polynomial of degree $7$, and we have the following basic information of $\hat{X}$:
\begin{itemize}
\item $\iota_{\hat{X}} = 1$.
\item $(\hat{A}^3) = 7/6$.
\item $\Sing (\hat{X}) = \{\frac{1}{2} (1, 1, 1), \frac{1}{3} (1, 1, 2), cE_7\}$, where $\hat{\msq} = \hat{\msp}_z \in \hat{X}$ is the $cE_7$ point (see Lemma~\ref{lem:No110cEdet}).
\end{itemize}

\subsubsection{The link $\breve{\sigma} \colon X \ratmap \breve{X}$ centered at the $\frac{1}{5} (1, 1, 4)$ point $\msr$}

We recall the construction of the elementary link $\breve{\sigma} \colon X \ratmap \breve{X}$ initiated by the Kawamata blow-up $\upsilon \colon \mcZ \to X$ at the  $\frac{1}{5} (1, 1, 4)$ point $\mathsf{r} \in X$.
We define
\[
\Upsilon_0 \colon
\text{$\mbU_0 := \mbT \begin{pNiceArray}{cc|cccc}[first-row]
u & z & w & y & t & x \\
0 & 5 & 8 & 3 & 7 & 1 \\
-5 & 0 & 1 & 1 & 4 & 2
\end{pNiceArray}$} \to \mbP (5_z, 8_w, 3_y, 7_t, 1_x)
\]
by
\[
(u\!:\!z \, | \, w\!:\!y\!:\!t\!:\!x) \mapsto (z\!:\!w u^{1/5} \!:\!y u^{1/5}\!:\!t u^{4/5}\!:\!x u^{2/5}),
\]
which is the weighted blow-up with $\wt (w, y, t, x) = \frac{1}{5} (1, 1, 4, 2)$.
Let $\mcZ_0 \subset \mbU_0$ be the proper transform of $X$ so that the restriction $\upsilon_0 := \Upsilon|_{\mcZ_0} \colon \mcZ_0 \to X$ gives the Kawamata blow-up of $X$ at $\msr$.
The variety $\mcZ_0$ is defined by the equation $G_0 (u, z, w, y, t, x) = 0$, where
\[
G_0 (u, z, w, y, t, x) := x^{-2/5} F (x u^{2/5}, y u^{1/5}, z, t u^{4/5}, w u^{1/5}).
\]
By filtering off terms divisible by $z$, we write $F = z F_1 + F_2$, where $F_1 \in \mbC [x, y, z, t, w]$ and $F_2 \in \mbC [x, y, t, w]$.
The weights of $G_1$ and $G_2$ with respect to $\wt (w, y, t, x) = \frac{1}{5} (1, 1, 4, 2)$ are $2/5$ and $7/5$, respectively, and $F_1, F_2$ are both contained in the ideal $(x, y, t, w)$.
We define
\[
\begin{split}
G_1 (u, x, y, z, t, w) &:= u^{-2/5} F_1 (x u^{2/5}, y u^{1/5}, z, t u^{4/5}, w u^{1/5}), \\
G_2 (u, x, y, t, w) &:= u^{-7/5} F_2 (x u^{2/5}, y u^{1/5}, t u^{4/5}, w u^{1/5}),
\end{split}
\]
so that $G_0 = z G_1 + u G_2$.
Note that $G_1, G_2 \in (x, y, t, w)$, and $G_0$ is contained in the irrelevant ideal $(u, z) \cap (w, y, t, x)$ of $\mbU_0$.
We consider unprojection as follows.
We introduce the new coordinate $v$ by the cross ratio obtained from the equation $z G_1 + u G_2 = 0$:
\[
v := \frac{G_1}{u} = - \frac{G_2}{z}.
\]
We define a new rank $2$ toric variety
\[
\text{$\mbU := \mbT \begin{pNiceArray}{cc|ccccc}[first-row]
u & z & w & v & y & t & x \\
0 & 5 & 8 & 16 & 3 & 7 & 1 \\
-5 & 0 & 1 & 7 & 1 & 4 & 2
\end{pNiceArray}$}
\]
and its codimension $2$ complete intersection subvariety $\mcZ \subset \mbU$ defined by
\[
\begin{cases}
v u - G_1 (u, x, y, z, t, w) = 0, \\
v z + G_2 (u, x, y, z, t, w) = 0.
\end{cases}
\]
The natural projection $\mbU \ratmap \mbU_0$ obtained by removing the variable $v$ restricts to an isomorphism $\mcZ \to \mcZ_0$ since $G_1, G_2 \in (x, y, t, w)$.
Let $\Upsilon \colon \mbU \ratmap \mbP (5, 8, 3, 7, 1)$ be the composite of $\mbU \ratmap \mbU_0$ and $\Upsilon_0$.
Then we get the diagram.
\[
\xymatrix{
\text{$\mbU := \mbT \begin{pNiceArray}{cc|ccccc}[first-row]
u & z & w & v & y & t & x \\
0 & 5 & 8 & 16 & 3 & 7 & 1 \\
-5 & 0 & 1 & 7 & 1 & 4 & 2
\end{pNiceArray}$} \ar@{-->}[r]^{\breve{\Theta}} \ar@{-->}[d]_{\Upsilon} &
\text{$\mbT \begin{pNiceArray}{ccccc|cc}[first-row]
u & z & w & v & y & t & x \\
7 & 4 & 5 & 3 & 1 & 0 & -2 \\
1 & 2 & 3 & 5 & 1 & 2 & 0
\end{pNiceArray} =: \breve{\mbU}$} \ar[d]^{\breve{\Upsilon}} \\
\mbP (5_z, 8_w, 3_y, 7_t, 1_x) & \mbP (1_u, 2_z, 3_w, 5_v, 1_y, 2_t)}
\]
The birational maps $\Upsilon, \Omega$, and the birational morphism $\breve{\Upsilon}$ are defined as follows:
\[
\begin{split}
\Upsilon & \colon (u\!:\!z \, | \, w\!:\!v\!:\!y\!:\!t\!:\!x) \mapsto (z\!:\!w u^{1/5} \!:\!y u^{1/5}\!:\!t u^{4/5}\!:\!x u^{2/5}), \\
\breve{\Upsilon} & \colon (u\!:\!z\!:\!w\!:\!v\!:\!y \, | \,  t\!:\!x) \mapsto (u x^{7/2}\!:\!z x^2\!:\!w x^{5/2}\!:\!v x^{3/2}\!:\!y x^{1/2}\!:\!t), \\
\breve{\Theta} & \colon (u\!:\!z \, | \, w\!:\!v\!:\!y\!:\!t\!:\!x) \mapsto (u\!:\!z\!:\!w\!:\!v\!:\!y \, | \,  t\!:\!x).
\end{split}
\]  
The restriction $\upsilon := \Upsilon|_{\mcZ} \colon \mcZ \to X$ is the Kawamata blow-up of $X$ at $\msr$.
Let $\breve{Z} \subset \breve{\mbU}$ be the birational transform of $\mcZ$ and $\breve{X} := \breve{\Upsilon} (\breve{Z})$. 
Then we have the diagram
\[
\xymatrix{
\mcZ \ar[d]_{\upsilon} \ar@{--}[r]^{\breve{\theta}} & \breve{\mcZ} \ar[d]^{\breve{\upsilon}} \\
X & \breve{X}}
\]
where $\breve{\theta} := \breve{\Theta}|_{\mcZ}$ is a pseudo isomorphism, $\breve{\upsilon} := \breve{\Upsilon}|_{\breve{\mcZ}}$ is a divisorial contraction with center a point, which we denote by $\breve{\msr} \in \breve{X}$.
The variety $\breve{X} = \breve{X}_{6, 7} \subset \mbP (1_u, 1_y, 2_z, 2_t, 3_w, 5_v)$ is a Mori-Fano $3$-fold weighted complete intersection defined by the equations
\[
\begin{cases}
v u - \breve{F}_1 (u, y, z, t, w)  = 0, \\
v z + \breve{F}_2 (u, y, t, w) = 0,
\end{cases}
\]
where
\begin{equation} \label{eq:110defeqbreve}
\breve{F}_1 := G_1 (u, 1, y, z, t, w), \quad
\breve{F}_2 := G_2 (u, 1, y, t, w).
\end{equation}
Following are basic information of $\breve{X}$.
\begin{itemize}
\item $\iota_{\breve{X}} = 1$.
\item $(\breve{A}^3) = 7/10$.
\item $\Sing (\breve{X}) = \{cD/2, \frac{1}{5} (1, 2, 3)\}$.
We set $\breve{\msr} := \breve{\msp}_t \in \breve{X}$ which is the $cD/2$ point (see Lemma~\ref{lem:No110cDdet}), and $\breve{\msq} := \breve{\msp}_v \in \breve{X}$ which is the $\frac{1}{5} (1, 2, 3)$ point.
\end{itemize}

\subsubsection{Defining equations of $X$, $\hat{X}$, and $\breve{X}$}

\begin{Lem} \label{lem:110defeq}
By a choice of homogeneous coordinates, the defining polynomial of $X$ can be written as
\begin{equation} \label{eq:110defeq}
\begin{split}
F = w^2 z + w (t d_6 + d_{13}) + z^4 x & \,  + z^3 a_6 + z^2 (t b_4 + b_{11}) \\ 
& + z (t c_9 + c_{16}) + t^3 + t d_{14} + d_{21},
\end{split}
\end{equation}
where $a_i, b_i, c_d, d_i \in \mbC [x, y]$ are quasi-homogeneous polynomials of degree $i$ such that $\coeff_{d_{21}} (y^7) = 1$.
\end{Lem}

\begin{proof}
By the quasi-smoothness we have $w^2 z \in F$.
By replacing $z$, we may assume that $w^2 z$ is the unique monomial in $F$ divisible by $w^2$.
By filtering off terms divisible by $z$, we can write 
\[
F = z (w^2 + w g_8 + g_{16}) + w h_{13} (x, y, t) + h_{21} (x, y, t)
\] 
for some quasi-homogeneous polynomials $g_8, g_{16} \in \mbC [x, y, z, t]$ and $h_{13}, h_{21} \in \mbC [x, y, t]$ of the indicated degree.
By replacing $w \mapsto w - g_8/2$, we may assume $g_8 = 0$.
By replacing $t \mapsto t - e_7 (x, y, z)$ for an appropriate quasi-homogeneous polynomial $e_7 \in \mbC [x, y, z]$ of degree $7$, we may assume that there is no monomial divisible by $t^2$ except for $t^3$.
By the quasi-smoothness of $X$, we have $z^4 x, t^3, y^7 \in F$, and hence we may assume $\coeff_{F} (z^4 x) = \coeff_{F} (t^3) = \coeff_{F} (y^7)1$ by rescaling $x$, $y$, and $t$.
Then, we obtain \eqref{eq:110defeq} by simply writing down $F$.
\end{proof}

For a quasi-homogeneous polynomial $e = e (y, x)$ of degree $d$ with respect to $\wt (y, x) = (3, 1)$, we define 
\[
\grave{e} (y, u) := u^{-m/3} e (y, u^{1/3}),
\]
where $0 \le m \le 2$ is such that $d \equiv m \pmod{3}$.
Note that $\grave{e}$ is a homogeneous polynomial in variables $y$ and $u$ of degree $\lrd d/3 \rrd$.

\begin{Lem} \label{lem:110defeqhat}
Under the choice of homogeneous coordinates of $\mbP$ as in Lemma \ref{lem:110defeq}, the defining polynomial of $\hat{X} \subset \hat{\mbP}$ is as follows:
\[
\begin{split}
\hat{F} = w^2 z + w (t \grave{d}_6 + \grave{d}_{13}) + z^4 u^3 & \, + z^3 u^2 \grave{a}_6 + z^2 u^2 (t \grave{b}_4 + \grave{b}_{11}) \\
& + z u (t \grave{c}_9 + \grave{c}_{16}) + t^3 u + t u \grave{d}_{14} + \grave{d}_{21},
\end{split}
\]
where $\coeff_{\grave{d}_{21}} (y^7) = 1$.
\end{Lem}

\begin{proof}
This follows from the definition of $\hat{F}$ given in \eqref{eq:110defeqhat} and Lemma~\ref{lem:110defeq}
\end{proof}.

\begin{Lem} \label{lem:110defeqbreve}
Under the choice of homogeneous coordinates of $\mbP$ as in Lemma~\ref{lem:110defeq}, the defining polynomials of $\breve{X}$ are as follows:
\[
\begin{cases}
v u - (w^2 + z^3 + z^2 \grave{a}_6 + z u (t \grave{b}_4 + \grave{b}_{11}) + u (t \grave{c}_9 + \grave{c}_{16})) = 0, \\
v z + (w (t \grave{d}_6 + \grave{d}_{13}) + t^3 u + t u \grave{d}_{14} + \grave{d}_{21}) = 0,
\end{cases}
\]
where $\coeff_{\grave{d}_{21}} (y^7) =1$.
\end{Lem}

\begin{proof}
This follows from the definition of $\breve{F}_1$, $\breve{F}_2$ given in \eqref{eq:110defeqbreve} and Lemma~\ref{lem:110defeq}
\end{proof}.

\subsubsection{The link $\rho \colon \breve{X} \ratmap \hat{X}$ centered at the $\frac{1}{5} (1, 2, 3)$ point $\breve{\msq}$}

\begin{Lem}
There exists an elementary link $\rho \colon \breve{X} \ratmap \hat{X}$ initiated by the Kawamata blow-up at the $\frac{1}{5} (1, 2, 3)$ point $\breve{\msq} \in \breve{X}$.
The inverse link $\rho^{-1}$ is initiated by a divisorial contraction centered at the $cE$ point $\hat{\msq} \in \hat{X}$.
\end{Lem}

\begin{proof}
Let $\breve{\mcW} \to \breve{X}$ be the Kawamata blow-up of $\breve{X}$ at the $\frac{1}{5} (1, 2, 3)$ point $\breve{\msq}$.
By \cite[Section 4.2]{OkI}, it initiates an elementary link $\rho \colon \breve{X} \ratmap \bar{X}$ to a Mori-Fano hypersurface in $\mbP (1, 1, 1, 2, 3)$, and its inverse link is initiated by a divisorial contraction centered at a non-quotient singular point of $\bar{X}$.

By the arguments in \cite[Section 4.2]{OkI}, $\bar{X} = \bar{X}_7 \subset \mbP (1_u, 1_y, 2_s, 2_t, 3_w)$ is a hypersurface defined by a quasi-homogeneous polynomial $\bar{F} = \bar{F} (u, y, s, t, w)$ of degree $7$, where 
\[
\bar{F} := s \breve{F}_1 (u, y, s u, t, w) + \breve{F}_2 (u, y, t, w).
\]
Explicitly, we have
\[
\begin{split}
\bar{F} &= s (w^2 + s^3 u^3 + s^2 u^2 \grave{a}_6 + s u^2 (t \grave{b}_4 + \grave{b}_{11}) + u (t \grave{c}_9 + \grave{c}_{16})) \\
& \hspace{4.5cm} + w (t \grave{d}_6 + \grave{d}_{13}) + t^3 u + t u \grave{d}_{14} + \grave{d}_{21}.
\end{split}
\]
Moreover the center of the link $\rho^{-1}$ is the point $(0\!:\!0\!:\!1\!:\!0\!:\!0)$.
Thus $\bar{X} = \hat{X}$ and the center of $\rho^{-1}$ is $\hat{\msq}$ by simply replacing $s$ with $z$.
\end{proof}

\begin{Rem} \label{rem:100rhoinv}
We can observe that the divisorial contraction that initiates the link $\rho^{-1} \colon \hat{X} \ratmap \breve{X}$ is the weighted blow-up of $\hat{X}$ at $\hat{\msq}$ with $\wt (u, y, t, w) = (2, 1, 2, 3)$, and it is possible to construct $\rho^{-1}$ by the $2$-ray game of a suitable rank $2$ toric variety.
\end{Rem}

\subsection{Birational geometry of $X$}

By \cite[Lemma 4.1 and 4.2]{ACP}, we have the following.

\begin{Prop} \label{prop:110EL}
The Kawamata blow-ups $\varphi \colon \mcY \to X$ and $\upsilon \colon \mcZ \to X$ at the $\frac{1}{8} (1, 3, 5)$ point $\msq \in X$ and the $\frac{1}{5} (1, 2, 3)$ point $\msr \in X$, respectively, are all the maximal extractions of $X$, and $\EL (X) = \{\hat{\sigma}, \breve{\sigma}\}$.
\end{Prop}

\subsection{Birational geometry of $\hat{X}$}

\subsubsection{Divisorial contractions centered at the $cE$ point}

\begin{Lem} \label{lem:No110cEdet}
The point $\hat{\msq} \in \hat{X}$ is of type $cE_7$.
\end{Lem}

\begin{proof}
The germ $\hat{\msq} \in \hat{X}$ is the germ $(\hat{\phi} = 0) \subset \mbC^4_{u, y, t, w}$, where $\hat{\phi} = \hat{F} (u, y, 1, t, w)$.
We know that $\hat{\msq} \in \hat{X}$ is a Gorenstein terminal singularity, so that its general hyperplane section $S$ is a Du Val singularity.
By Lemma \ref{lem:SE7} below, $\hat{\msq} \in S$ is of type $E_7$, and thus $\hat{\msq} \in \hat{X}$ is of type $cE_7$. 
\end{proof}

\begin{Lem} \label{lem:SE7}
Let $h \in \mbC \{u, y, t, w\}$ be a convergent power series such that $y \in h$.
If the surface germ $\hat{\msq} \in S \subset \hat{X}$ cut by the equation $h (u, y, t, w) = 0$ is an isolated singularity, then $S$ is of type $E_7$.  
\end{Lem}

\begin{proof}
We choose homogeneous coordinates such that $\hat{F}$ is as in Lemma~\ref{lem:110defeqhat}.
The germ $\hat{\msq} \in \hat{X}$ is the germ $(\hat{\phi} = 0) \subset \mbC^4_{u, y, t, w}$, where $\hat{\phi} = \hat{F} (u, y, 1, t, w)$.
Filtering off terms divisible by $y$ in $h$, we can write $h = y \xi - e$, where $\xi \in \mbC \{u, y, t, w\}$ and $e \in \mbC \{u, t, w\}$ such that $\xi (o) \ne 0$.
We introduce new coordinates $(u, \bar{y}, z, t)$, where $\bar{y} = y \xi$. 
We have $y = g$ for some $g (u, \bar{y}, t, w) \in \mbC \{u, \bar{y}, t, w\}$ such that $\bar{y} \in g$ and $g (o) = 0$.
In this new coordinates, $S$ is equivalent to the germ
\[
\begin{split}
& (\hat{\phi} (u, g, t, w) = \bar{y} - e = 0) \subset \mbC^4_{u, \bar{y}, t, w}, \\
\cong & (\hat{\phi} (u, \bar{g}, t, w) = 0) \subset \mbC^3_{u, t, w},
\end{split}
\]
where $\bar{g} = g (u, e, t, w) \in \mbC \{u, t, w\}$.

We consider the weight $\msw$ on $u, t, w$ defined by $\msw (u, t, w) = (6, 4, 9)$.
We have $\msw (\bar{g}) \ge 4$.
For a power series $f = f (u, y)$, we define $\bar{f} = f (u, e)$.
Then we have
\[ 
\msw (\grave{a}_i (u, \bar{g})) = \msw (\grave{b}_i (u, \bar{g})) = \msw (\grave{c}_i (u, \bar{g})) = \msw (\grave{d}_i (u, \bar{g})) \ge 4 \lrd i/3 \rrd,
\]
and thus the least $\msw$-weight terms of $\hat{\phi} (u, \hat{g}, t, w)$ are
\[
\hat{\phi} (u, \bar{g}, t, w)_{\msw = 18} = w^2 + u^3 + t^3 u.
\]
By \cite[Corollary 4.7]{Paemurru}, $\hat{\msq} \in S$ is Du Val of type $E_7$.
\end{proof}

We define weighted blow-ups of $\hat{X}$ at the $cE$ point $\hat{\msq}$ as follows:
\[
\begin{split}
\hat{\varphi}_1 &\colon \hat{\mcY}_1 \to \hat{X}, \quad \wt (u, y, t, w) = (2, 1, 2, 3), \\
\hat{\varphi}_2 &\colon \hat{\mcY}_2 \to \hat{X}, \quad \wt (u, y, t, w) = (3, 1, 1, 3),
\end{split}
\]
and let $\hat{\mcE}_i$ be the $\hat{\psi}_i$-exceptional locus for $i = 1, 2$.

\begin{Lem} \label{lem:110twodivdisc1}
$\hat{\mcE}_1$ and $\hat{\mcE}_2$ give distinct divisors of discrepancy $1$ over $\hat{\msq} \in \hat{X}$.
\end{Lem}

\begin{proof}
For $i = 1, 2$, let $\hat{\msw}_i$ be the weight defined by $\hat{\msw}_1 (u, y, t, w) = (2, 1, 2, 3)$ and $\hat{\msw}_2 (u, y, t, w) = (3, 1, 1, 3)$.
We choose homogeneous coordinates so that the equation $\hat{F}$ is the one in Lemma~\ref{lem:110defeqhat}.
We set $\hat{\phi} = \hat{F} (u, y, 1, t, w)$.
Then the germ $\hat{\msq} \in \hat{X}$ coincides with the germ $(\hat{\phi} = 0) \subset \mbC^4_{u, y, t, w}$.
We have $\hat{\msw}_1 (\hat{\phi}) = \hat{\msw}_2 (\hat{\phi}) = 6$ and 
\[
\begin{split}
\hat{\phi}_{\hat{\msw}_1 = 6} &= w^2 + u^3 + \alpha_6 u^2 y^2 \\
\hat{\phi}_{\hat{\msw}_2 = 6} &= w^2 + \delta_6 w t y^2 + t^3 u,
\end{split}
\]
where $\alpha_6 = \coeff_{\grave{a}_6} (y^2)$ and $\delta_6 = \coeff_{\grave{d}_6} (y^2)$.
We see that $\hat{\phi}_{\hat{\msw}_1 = 6}$ and $\hat{\phi}_{\hat{\msw}_2 = 6}$ are irreducible polynomials.
The assertion follows from Lemma~\ref{lem:genwblexc}.
\end{proof}

\begin{Lem} \label{lem:110hatdiv1}
There are no divisorial contractions centered at the $cE_7$ point $\hat{\msq} \in \hat{X}$ other than $\hat{\varphi}$, $\hat{\varphi}_1$, and $\hat{\varphi}_2$. 
\end{Lem}

\begin{proof}
The divisorial contraction $\hat{\varphi}$ is of discrepancy $2$ over the $cE_7$ point $\hat{\msq} \in \hat{X}$.
Hence it corresponds to the one given in either Proposition \ref{prop:ncd2cE7e5} or \ref{prop:ncd2cE7e9}.
Since there are two divisors of discrepancy $1$ over $\hat{\msq} \in \hat{X}$ by Lemma~\ref{lem:110twodivdisc1}, we are in case Proposition \ref{prop:ncd2cE7e9} and the assertion follows immediately.
\end{proof}

\begin{Rem}
The morphism $\hat{\varphi}_1$ is in fact a divisorial contraction since 
it is the the one which initiates the elementary link $\rho^{-1} \colon \hat{X} \ratmap \breve{X}$ (see Remark \ref{rem:100rhoinv}).
We do not determine whether $\hat{\varphi}_2$ is a divisorial contraction or not.
We will instead exclude $\hat{\varphi}_2$ as a maximal extraction in Lemma \ref{lem:110exclphi2} if it is a divisorial contraction.
\end{Rem}

\subsubsection{A birational involution of $\hat{X}$}

We denote by $\hat{\psi} \colon \hat{\mcZ} \to \hat{X}$ the Kawamata blow-up of $\hat{X}$ at the $\frac{1}{2} (1, 1, 1)$ point $\hat{\msp}_t$.

\begin{Lem} \label{lem:110birinv}
If $\hat{\psi}$ is a maximal extraction, then there exists a birational involution $\hat{\iota}$ of $\hat{X}$ which is an elementary link initiated by $\hat{\psi}$. 
\end{Lem}

\begin{proof}
This follows from \cite[Theorem 3.3]{OkII}.
\end{proof}

\subsubsection{Maximal extractions of $\hat{X}$}

\begin{Lem} \label{lem:110hatCeq}
Suppose that there is an irreducible and reduced curve $\Gamma \subset \hat{X}$ of degree $1$ which passes through the $cE_7$ point $\hat{\msq}$ but does not pass through any other singular point.
Then $y^2 \in \grave{d}_6$ and
\[
\Gamma = (u = t - \lambda z y - \mu y^2 = w - \nu y^3 = 0)
\]
for some $\lambda, \mu, \nu \in \mbC$ with $\lambda \ne 0$ and $\nu \ne 0$.
\end{Lem}

\begin{proof}
Let $\Gamma$ be an irreducible and reduced curve of degree $1$ passing through $\hat{\msq}$ but does not pass through any other singular point of $\hat{X}$.
Let $\pi \colon \hat{X} \ratmap \mbP (1_u, 1_y, 1_z, 2_t)$ be the projection from the point $\hat{\msp}_w$.
Since $\Gamma$ does not pass through $\hat{\msp}_w$, the image $\pi (\Gamma)$ is a curve and we have $\deg \Gamma = \deg (\pi|_{\Gamma}) \deg \pi (\Gamma)$.  
It follows that $\deg \pi (\Gamma) = 1/2, 1$.

Suppose that $\deg \pi (\Gamma) = 1/2$.
Then $\pi (\Gamma) = (u = y = 0)$, and hence $\Gamma \subset (u = y = 0)_X$.
On the other hand, we have $\hat{F} (0, 0, z, t, w) = w^2 z$, and hence 
\[
(u = y = 0)_X = (u = y = z = 0) \cup (u = y = w = 0).
\]
This is impossible since $(u = y = z = 0)$ and $(u = y = w = 0)$ are irreducible and reduced cures of degree $1/6$ and $1/2$, respectively.

It follows that $\deg \pi (\Gamma) = 1$ and $\deg (\pi|_{\Gamma}) = 1$.
In this case we can write
\[
\pi (\Gamma) = (\alpha u - \beta y = t - z \ell_1 - q_1 = 0)
\]
for some $\alpha, \beta \in \mbC$, a linear and a quadratic forms $\ell_1 = \ell_1 (y, u)$ and $q_1 (u, y)$, and then
\[
\Gamma = (\alpha u + \beta y = t - \ell_1 - q_1 = w - z^2 \ell_2 - z q_2 - c = 0),
\]
where $\ell_2 = \ell_2 (u, y), q_2 = q_2 (u, y)$, and $c = c (u,y )$ are linear, quadratic, and cubic forms, respectively.
Suppose that $\beta \ne 0$.
We set $\gamma = \alpha/\beta$ and we consider the polynomial
\[
\bar{F} (u, z) := \hat{F} (u, \gamma u, z, z \bar{\ell}_1 + \bar{q}_1, z^2 \bar{\ell}_2 + z \bar{q}_2 + \bar{c}),
\]
where $\bar{\ell}_i = (u, \gamma u)$, $\bar{q}_i = q (u, \gamma u)$, and $\bar{c} = c (u, \gamma u)$.
We have
\[
\bar{F} (u, z) = \bar{\ell}_2^2 z^5 + u^3 z^4 + \cdots.
\]
The polynomial $\bar{F}$ is identically $0$ since $\Gamma \subset \hat{X}$.
This is impossible since the term $u^3 z^4 \in \bar{F}$.

It follows that $\beta = 0$.
We may assume $\alpha = 1$, $\ell_1 = \lambda y$, $q_1 = \mu y^2$, $\ell_2 = \delta y$, $q_2 = \varepsilon y^2$ and $c = \nu y^3$, where $\delta, \gamma, \lambda, \mu, \nu \in \mbC$.
We set
\[
\bar{F} := \hat{F} (0, y, z, \lambda y z + \mu y^2, \delta y z^2 + \varepsilon y^2 z + \nu y^3)
= \delta^2 y^2 z^5 + 2 \delta \varepsilon y^3 z^4 + \cdots
\]
The highest degree monomial in $\bar{F}$ with respect to $z$ is $\delta^2 y^2 z^5$, and we have $\delta = 0$.
Then the highest degree monomial in $\bar{F}$ with respect to $z$ is $\varepsilon^2 y^4 z^3$, and we have $\delta = 0$.
Finally we have
\[
\bar{F} = (\nu^2 + \nu \lambda \delta_6) y^6 z + (\nu \mu \delta_6 + 1) y^7 = 0,
\]
where $\delta_6 = \coeff_{\grave{d}_6} (y^2)$.
This shows that $\mu \ne 0$, $\nu \ne 0$, and $\delta_6 \ne 0$, and the proof is completed.
\end{proof}

\begin{Lem} \label{lem:110exclcurve}
No curve on $\hat{X}$ is a maximal center.
\end{Lem}

\begin{proof}
Let $\Gamma$ be an irreducible curve on $\hat{X}$.
By Lemma \ref{lem:mtdexclC}, it is enough to consider the case where $\deg \Gamma = 1$.

If $\Gamma$ is contained in the smooth locus of $X$, then $\Gamma$ is not a maximal center by the completely the same argument as in Step 2 of the proof of \cite[Theorem 5.1.1]{CPR}. 
If $\Gamma$ contains a quotient singular point, then there is no divisorial contraction of $\hat{X}$ with center $\Gamma$ by \cite{Kawamata}.

It remains to consider the case where $\Gamma$ is a curve of degree $1$ which passes through $\hat{\msq}$ but does not pass through any other singular points.
In this case $\Gamma$ is the curve described in Lemma \ref{lem:110hatCeq}.
A general hyperplane section $S$ of the germ $\hat{\msq} \in \hat{X}$ is defined by the equation
\[
\alpha u + \beta (t - \lambda y - \mu y^2) + \gamma (w - \nu y^3) = 0
\]
in $\hat{X}$, where $\alpha, \beta, \gamma \in \mbC$ are general.
We see that $\hat{\msq} \in S$ is an isolated singularity since $\alpha, \beta, \gamma$ are general. 
Then, since $\beta \lambda \ne 0$, we can apply Lemma \ref{lem:SE7} and conclude that $\hat{\msq} \in S$ is of type $E_7$.
By \cite[Theorem 4.1]{Tziolas}, there is no divisorial contraction of $\hat{X}$ centered along $\Gamma$.
Therefore $\Gamma$ cannot be a maximal center, and the proof is completed.
\end{proof}

\begin{Lem} \label{lem:110exclphi2}
The weighted blow-up $\hat{\varphi}_2 \colon \hat{\mcY}_2 \to \hat{X}$ of $\hat{X}$ at the $cE_7$ point $\hat{\msq}$ with $\wt (u, y, t, w) = (3, 1, 1, 3)$ is not a maximal extraction.
\end{Lem}

\begin{proof}
We may assume that $\hat{\varphi}_2$ is a divisorial contraction, that is, $\hat{\mcY}_2$ has only terminal singularities, because otherwise there is nothing to prove.

Let $S$ and $T$ be general members of the pencil $\mcL \subset |\hat{A}|$ generated by $u, y$.
We have $\Bs \mcL = \Gamma \cup \Delta$ set-theoretically, where $\Gamma = (u = y = z = 0)$ and $\Delta = (u = y = w = 0)$ are irreducible and reduced curves of degree $1/6$ and $1/2$, respectively.
We denote by $\tilde{S}$, $\tilde{T}$, $\tilde{\Gamma}$, and $\tilde{\Delta}$ the proper transforms of $S$, $T$, $\Gamma$, and $\Delta$ on $\hat{\mcY}_2$, respectively.

Let $y + \lambda u = 0$ be the equation of $T$ in $\hat{X}$, where $\lambda \in \mbC$ is general.
For the Jacobi matrix $J_{C_S}$ of the affine $C_S$, we have
\[
J_{C_S}|_{(u = y = 0)} =
\begin{pmatrix}
t^3 & 0 & w^2 & 0 & 2 w z \\
\lambda & 1 & 0 & 0 & 0
\end{pmatrix}.
\]
This shows that $S$ is quasi-smooth outside $\hat{\msq}$.
It is then easy to see that $\Sing (S) = \{\hat{\msq}, \hat{\msp}_t, \hat{\msp}_w\}$ and the singularity of $S$ at $\hat{\msp}_t$, $\hat{\msp}_w$ are of type $\frac{1}{2} (1, 1)$, $\frac{1}{3} (1, 2)$, respectively.
Since $\Gamma$ is a smooth rational curve which passes through $\hat{\msp}_t, \hat{\msp}_w$ and does not pass through $\hat{\msq}$, and $K_S = (K_X + S)|_S \sim 0$, we have
\[
(\tilde{\Gamma}^2) = (\Gamma^2) = - 2 + \frac{1}{2} + \frac{2}{3} = - \frac{5}{6}.
\]
The $\hat{\varphi}_2$-exceptional divisor $\hat{\mcE}_2$ is naturally isomorphic to 
\[
(w^2 + \delta_6 w t y^2 + t^3 u = 0) \subset \mbP (3_u, 1_y, 1_t, 3_w),
\] 
where $\delta_6 = \coeff_{\grave{d}_6} (y^2)$.
We have $\tilde{S} \sim \tilde{T} \sim \psi^*\hat{A} - E \sim - K_{\hat{\mcY}_2}$  and $\tilde{S} \cap \tilde{T} = \Xi$ set-theoretically, where
\[
\Xi = (w^2 + t^3 u = 0) \subset E. 
\]
We have $S|_T = \Gamma + 2 \Delta$ and $\tilde{S}|_{\tilde{T}} = \tilde{\Gamma} + 2 \tilde{\Delta} + 2 \Xi$.
We have $(\tilde{\Gamma} \cdot \Xi) = 0$.
Moreover, we see that $\tilde{\Delta}$ and $\Xi$ intersects at the smooth point $(0\!:\!0\!:\!1\!:\!0) \in \hat{\mcE}_2$ of $\hat{\mcY}_2$ transversally, and hence $(\tilde{\Delta} \cdot \Xi) = 1$.
Similarly we have $(\tilde{\Delta} \cdot E) = 1$.
By the computations
\[
\begin{split}
\frac{1}{6} &= (\tilde{\Gamma} \cdot \tilde{S}|_{\tilde{T}}) = (\tilde{\Gamma}^2) + 2 (\tilde{\Gamma} \cdot \tilde{\Delta}) + 2 (\tilde{\Gamma} \cdot \Xi), \\
\frac{1}{2} - 1 &= (\tilde{\Delta} \cdot \tilde{S}|_{\tilde{T}}) = (\tilde{\Gamma} \cdot \tilde{\Delta}) + 2 (\tilde{\Delta}^2) + 2 (\tilde{\Delta} \cdot \Xi),
\end{split}
\]
we obtain $(\tilde{\Gamma} \cdot \tilde{\Delta}) = 1/2$ and $(\tilde{\Delta}^2) = -3/2$.
The intersection matrix of $\tilde{\Gamma}$ and $\tilde{\Delta}$ is
\[
\begin{pmatrix}
(\tilde{\Gamma}^2) & (\tilde{\Gamma} \cdot \tilde{\Delta}) \\
(\tilde{\Gamma} \cdot \tilde{\Delta}) & (\tilde{\Delta}^2)
\end{pmatrix} =
\begin{pmatrix}
-5/6 & 1/2 \\
1/2 & - 3/2
\end{pmatrix},
\]
which is negative-definite.
By Lemma \ref{lem:exclnegdef}, $\hat{\varphi}_2$ is not a maximal extraction.
\end{proof}

\begin{Prop} \label{prop:110ELhat}
There is no maximal extraction other than $\hat{\varphi}$, $\hat{\varphi}_1$, and $\hat{\psi}$.
The contractions $\hat{\varphi}$ and $\hat{\varphi}_1$ are maximal extractions, and we have either $\EL (\hat{X}) = \{\hat{\sigma}^{-1}, \rho, \hat{\iota}\}$ or $\EL (\hat{X}) = \{\hat{\sigma}^{-1}, \rho\}$.
\end{Prop} 

\begin{proof}
This follows from Lemma \ref{lem:exclsingpt4}, \ref{lem:excl110nspt}, \ref{lem:110exclcurve}, \ref{lem:110hatdiv1}, and \ref{lem:110exclphi2}.
\end{proof}

\begin{Rem}
The singular point $\hat{\msq} \in \hat{X}$ is of type $cE_7$ admits a divisorial contraction of discrepancy $2$ and there are at most $3$ divisorial contractions.
This provides a negative answer to \cite[Question~1.6]{Paemurru}.
\end{Rem}

\subsection{Birational geometry of $\breve{X}$}

\subsubsection{Divisorial contractions to the $cD/2$ point}

\begin{Lem} \label{lem:No110cDdet}
The singularity $\breve{\msr} \in \breve{X}$ is of type $cD/2$.
\end{Lem}

\begin{proof}
We choose homogeneous coordinates so that $\breve{X}$ is the complete intersection in $\mbP (1_u, 1_y, 2_z, 2_t, 3_w, 5_v)$ by the equations in Lemma~\ref{lem:110defeqbreve}. 
Let $V$ be the complete intersection in $\mbC^5_{u, y, z, w, v}$ defined by 
\begin{equation} \label{eq:110cD2loc}
\begin{cases}
v u - (w^2 + z^3 + z^2 \grave{a}_6 + z u (\grave{b}_6 + \grave{b}_{11}) + u (\grave{c}_9 + \grave{c}_{16})) = 0, \\
v z + w (\grave{d}_6 + \grave{d}_{13}) + u + u \grave{d}_{14} + \grave{d}_{21} = 0.
\end{cases}
\end{equation}
The above equations are obtained by plugging $t = 1$ in the defining equations of $\breve{X}$, so that the germ $\breve{\msr} \in \breve{X}$ is equivalent to the germ $\bar{o} \in V/\mbZ_2 (1_u, 1_y, 0_z, 1_w, 1_v)$, where $\bar{o}$ is the image of the origin $o \in V$.
It is enough to show that the singularity $o \in V$ is of type $cD$.

By the implicit function theorem, the second equation in \eqref{eq:110cD2loc} can be replaced by the equation
\[
u = \xi (y, z, w, v)
\]
for an appropriate $\xi (y, z, w, v) \in \mbC \{y, z, w, v\}$.
It is easily observed that
\[
\xi = - (v z + \delta_6 y^2) + \xi_{\ge 3},
\]
where $\ord (\xi_{\ge 3}) \ge 3$.
Let $f (u, y, z, w, v) = 0$ be the first equation in \eqref{eq:110cD2loc}.
By plugging $t = \xi$ in this equation, the germ $o \in V$ is equivalent to the hypersurface germ
\[
o \in (f (\xi, y, z, w, v) = 0) \subset \mbC^4_{y, z, w, v}.
\]
We have 
\[
f (\xi, y, z, w, v) = w^2 + v^2 z + z^3 + f' (y, z, w, v),
\] 
where $\ord (f' (y, z, w, v)) \ge 4$.
By Lemma~\ref{lem:cricDcE} below, the singularity is of type $cD$ and the proof is completed.
\end{proof}

\begin{Lem} \label{lem:cricDcE}
Let 
\[
o \in V := (x^2 + f (x, y, z, u) = 0) \subset \mbC^4
\]
be a terminal singularity, where $f \in (x, y, z, u)^3 \mbC \{x, y, z, u\}$, and let $g (y, z, u)$ be the degree $3$ part of of $f (0, y, z, u)$.
Then it is either of type $cD$ or of type $cE$.
Moreover, it is of type $cE$ if and only if $g (y, z, u) = \ell (y, z, u)^3$ for some linear form $\ell (y, z, u)$. 
\end{Lem}

\begin{proof}
The germ $o \in V$ is clearly a Gorenstein terminal singularity, and hence it is of type $cA$, $cD$, or $cE$.
It cannot be of type $cA$ since the quadratic part of the defining equation of $V$ is $x^2$.
This shows the first assertion.
  
If $o \in V$ is of type $cE$ (resp.\ $cD$), then there is an automorphism $\chi$ of $\mbC \{x, y, z, u\}$ inducing an equivalence of germs $o \in V \cong o \in \overline{V}$, where 
\[
\overline{V} = (x^2 + \bar{f} (y, z, u) = 0) \subset \mbC^4_{x, y, z, u}
\] 
with 
\[
\bar{f} = 
\begin{cases}
y^3 + y \bar{g} (z, u) + \bar{h} (z, u) \quad, & \text{if $o \in V$ is of type $cE$}, \\
y^2 u + \lambda y z^l + \bar{g} (z,u), & \text{if $o \in V$ is of type $cD$},
\end{cases}
\]
where $\ord (\bar{g} (z, u)) \ge 3$, $\ord (\bar{h} (z, u)) \ge 4$, $\lambda \in \mbC$, $l \ge 2$ and $\bar{g} (z,u)$ consists of monomials of degree at least $3$).
We can write $\chi (x) = \ell_x + q_x + \cdots$, $\chi (y) = \ell_y + \cdots$, $\chi (z) = \ell_z + \cdots$ and $\chi (u) = \ell_u + \cdots$, where $\ell_x, \dots, \ell_u$ are linear forms, $q_x$ is a quadratic form in $x, y, z, u$ and the omitted terms in $\chi (x)$ (resp.\ $\chi (y), \dots, \chi (u)$) consist of monomials of degree at least $3$ (resp.\ $2$).
We have $\chi (x^2 + \bar{f}) = (\text{unit}) f$ for some non-zero invertible element $(\text{unit}) \in \mbC \{x, y, z, u\}$, hence we can write 
\[
x^2 + f = (\text{unit})^{-1} \chi (x^2 + \bar{f}) = (\alpha + \bar{\ell} + \cdots) \chi (x^2 + \bar{f}),
\] 
where $\alpha \in \mbC$ is non-zero, $\bar{\ell}$ is a linear form in $x, y, z, u$ and the omitted terms consist of monomials of degree at least $2$.

Suppose that $o \in V$ is of type $cE$.
Then 
\[
\chi (x^2 + \bar{f}) = \ell_x^2 + 2 \ell_x q_x + \ell_y^3 + \cdots,
\] 
where the omitted terms consist of monomials of degree at least $4$.
By comparing the quadratic part, we have $\ell_x = \pm \sqrt{\alpha}^{-1} x$, and then comparing the cubic part in $y, z, u$, we have $g (y, z, u) = \ell_y (0, y, z, u)^3$.

Suppose that $o \in V$ is of type $cD$, then by the same argument as above, we see that $g (y,z,u)$ cannot be a cube of a linear form since the degree $3$ part of $\bar{f}$ is $y^2 u + \lambda \delta_{l,3} y z^l + g_{\deg = 3} (z, u)$, where $\delta_{\lambda,3}$ is the Kronecker delta, and it is not a cube of a linear form.
\end{proof}

\begin{Lem} \label{lem:No110divcont}
The morphism $\breve{\upsilon} \colon \breve{Z} \to \breve{X}$ is the unique divisorial contraction centered at the $cD/2$ point $\breve{\msr}$.
\end{Lem}

\begin{proof}
We choose homogeneous coordinates as in Lemma~\ref{lem:110defeqbreve}.
The morphism $\breve{\upsilon}$ is the weighted blow-up of $\breve{X}$ at $\breve{\msr}$ with $\wt (u, y, z, w, v) = \frac{1}{2} (7, 1, 4, 5, 3)$, and its exceptional divisor $\breve{E}$ is isomorphic to the complete intersection in $\breve{\mbE} := \mbP (7_u, 1_y, 4_z, 5_w, 3_v)$ defined by the equations
\[
\begin{cases}
v u - w^2 - \alpha_6 z^2 y^2 - \gamma_9 u y^3 =0, \\
v z + \delta_6 w y^2 + u + y^7 = 0,
\end{cases}
\]
where $\alpha_6 = \coeff_{\grave{a}_6} (y^2)$, $\gamma_9 = \coeff_{\grave{c}_9} (y^3)$, and $\delta_6 = \coeff_{\grave{d}_6} (y^2)$. 
It is easy to see that the non-Gorenstein singular points of $\breve{Z}$ along $\breve{E}$ are $\msp_v = (0\!:\!0\!:\!0\!:\!0\!:\!1), \msp_z = (0\!:\!0\!:\!1\!:\!0\!:\!0) \in \breve{\mbE}$.
The singularity of $\breve{Z}$ at $\msp_v$ is of type $\frac{1}{3} (1, 1, 2)$.
Since $z^3$ appears in the first equation of the defining equations of $\breve{X}$ (given in Lemma~\ref{lem:110defeqbreve}), the singularity of $\breve{Z}$ at $\msp_z$ is of type $\frac{1}{4} (1, 1, 3)$.
We compute
\[
(\breve{E}^3) = \frac{2^2 \cdot 10 \cdot 7}{7 \cdot 1 \cdot 4 \cdot 5 \cdot 3} = \frac{2}{3}.
\]
By Corollary~\ref{cor:cD2dcuni}, $\breve{\upsilon}$ is the unique divisorial contraction centered at $\breve{\msr} \in \breve{X}$.
\end{proof} 

\subsubsection{Exclusion of curves}

\begin{Lem} \label{lem:110breveC}
Let $\Gamma \subset \breve{X}$ be an irreducible and reduced curve of degree $1/2$ that passes through $\breve{\msr}$ but does not pass through any other singular point of $\breve{X}$.
Then one of the following holds.
\begin{enumerate}
\item $\Gamma$ is a complete intersection curve of type $(1, 1, 5, 6)$ in $\breve{\mbP}$ and $\Gamma = (u = y = w^2 + z^3 = v = 0)$.
\item $\Gamma$ is a complete intersection curve of type $(1, 2, 3, 5)$ in $\breve{\mbP}$ and 
\[
\Gamma = (u = z - \lambda y^2 = w - \mu y z = v - \nu t y^3 - \xi y^5 = 0)
\]
for some $\lambda \ne 0, \mu, \nu, \xi \in \mbC$ with the property that if $\mu = 0$, then $\nu = 0$ and $\xi \ne 0$.
\end{enumerate}
\end{Lem}

\begin{proof}
Let $\Gamma \subset \breve{X}$ be an irreducible and reduced curve of degree $1/2$ passing through $\breve{\msr}$ but does not pass through any other singular point.
Let $\pi \colon \breve{X} \ratmap \mbP (1_u, 1_y, 2_z, 2_t, 3_w)$ be the projection which is defined outside the point $\breve{\msq}$ and which contracts the curve $\Xi := (u = y = z = w = 0)$ of degree $1/10$.
The image $\pi (\Gamma)$ is an irreducible and reduced curve of degree at most $1/2$ on $\mbP (1, 1, 2, 2, 3)$ and it does not pass through the point $\msp := (0\!:\!0\!:\!0\!:\!0\!:\!1)$ since $\breve{\msq} \notin \Gamma$ and $\pi^{-1} (\msp) = \{\breve{\msq}\}$.
It follows that the image of $\pi (\Gamma)$ under the projection $\mbP (1_u, 1_y, 2_z, 2_t)$ is an irreducible and reduced curve of degree at most $1/2$ passing through $(0\!:\!0\!:\!0\!:\!1)$.
Then we see that $\Gamma$ is a complete intersection of type either $(1, 1)$ or $(1, 2)$ in $\mbP (1, 1, 2, 2)$.
\end{proof}

\begin{Lem} \label{lem:110exclC1}
Let $\Gamma \subset \breve{X}$ be an irreducible and reduced complete intersection curve of type $(1, 1, 5, 6)$ which passes through $\breve{\msr}$ but does not pass through any other singular point of $\breve{X}$.
Then $\Gamma$ is not a maximal center.
\end{Lem}

\begin{proof}
Note that $\Gamma = (u = y = v = w^2 + z^3 = 0)$ by Lemma~\ref{lem:110breveC}.
We define $\mcM := |\breve{A}|$ and set $T := (u = 0)_X$.
Let $S \in \mcM$ be a general member.
Then $\Bs \mcM = S \cap T = \Gamma \cup \Delta$ set-theoretically, where $\Delta := (u = y = z = w = 0) \subset X$ is an irreducible and reduced curve of degree $1/10$.

We will show that $(\Gamma \cdot \Delta) = 1/2$, where the intersection number is taken on $S$.
Since $\Gamma \cap \Delta = \{\breve{\msr}\}$, it is difficult to compute the number $(\Gamma \cdot \Delta)$ directly. 
We consider the divisorial contraction $\breve{\upsilon} \colon \breve{\mcZ} \to \breve{X}$ centered at $\breve{\msr}$, which is the weighted blow-up with $\wt (u, y, z, w, v) = \frac{1}{2} (7, 1, 4, 5, 3)$.
Let $\tilde{\Gamma}$, $\tilde{\Delta}$, $\tilde{S}$, and $\tilde{T}$ be proper transforms of $\Gamma$, $\Delta$, $S$, and $T$ on $\breve{\mcZ}$, respectively.
The $\breve{\upsilon}$-exceptional divisor $\breve{\mcE}$ is naturally isomorphic to the complete intersection in $\mbP (7_u, 1_y, 4_z, 5_w, 3_v)$ defined by the equations
\[
\begin{cases}
v u - (w^2 + \alpha_6 z^2 y^2 + \gamma_9 u y^3) = 0, \\
v z + \delta_6 w y^2 + u + y^7 = 0,
\end{cases}
\]
where $\alpha_6 = \coeff_{\grave{a}_6} (y^2)$, $\gamma_9 = \coeff_{\grave{c}_9} (y^3)$, and $\delta_6 = \coeff_{\grave{d}_6} (y^2)$. 
We have $\tilde{\Gamma} \cap \breve{\mcE} = \{\msp_1\}$ and $\tilde{\Delta} \cap \breve{\mcE} = \{\msp_2\}$, where 
\[
\msp_1 = (0\!:\!0\!:\!1\!:\!0\!:\!0), \quad
\msp_2 = (0\!:\!0\!:\!0\!:\!0\!:\!1) \in \breve{\mcE} \subset \mbP (7_u, 1_y, 4_z, 5_w, 3_v).
\]
This shows that $(\tilde{\Gamma} \cdot \tilde{\Delta}) = 0$, where the intersection number is taken on the surface $\tilde{S}$.
Since $\breve{\upsilon}^*S = \tilde{S} + \frac{1}{2} \breve{\mcE}$, we have $K_{\tilde{S}} = (\breve{\upsilon}|_{\tilde{S}})^*K_S \sim 0$.
We see that $\tilde{S}$ has a singular point of type $\frac{1}{4} (1, 3)$ at the point $\msp_1$ and is smooth along $\tilde{\Gamma} \setminus \{\msp_1\}$.
It follows that $(\tilde{\Gamma}^2) = -2 + \frac{3}{4} = - \frac{5}{4}$.
Similarly, $\tilde{S}$ has a singular point of type $\frac{1}{3} (1, 2)$ at $\msp_2$ and of type $\frac{1}{5} (2, 3)$ at $\breve{\msr}$.
It follows that $(\tilde{\Delta}^2) = - 2 + \frac{2}{3} + \frac{4}{5} = -\frac{8}{15}$.
By taking the intersection number of the divisor
\[
(\breve{\upsilon}^*\breve{A} - \frac{7}{2} \breve{\mcE})|_{\tilde{S}} \sim \tilde{T}|_{\tilde{S}} = \tilde{\Gamma} + 2 \tilde{\Delta}
\]
and $\tilde{\Gamma}$, and then $\tilde{\Delta}$, we obtain $(E \cdot \tilde{\Gamma}) = 1/2$ and $(E \cdot \tilde{\Delta}) = 1/3$, where $E = \breve{\mcE}|_{\tilde{S}}$.
We write $(\breve{\upsilon}|_{\tilde{S}})^* \Gamma = \tilde{\Gamma} + m E$.
We have
\[
0 = (E \cdot (\breve{\upsilon}|_{\tilde{S}})^*\Gamma) = (E \cdot \tilde{\Gamma}) + m (E^2) = \frac{1}{2} - \frac{1}{3} m,
\]
where we note that $(E^2) = (\breve{\mcE}^2 \cdot \tilde{S}) = - \frac{1}{2} (\breve{\mcE}^3) = - 1/3$.
Hence $m = 3/2$ and we compute
\[
(\Delta \cdot \Gamma) = (\tilde{\Delta} \cdot (\breve{\upsilon}|_{\tilde{S}})^*\Gamma) = (\tilde{\Delta} \cdot \tilde{\Gamma}) + \frac{3}{2} (\tilde{\Delta} \cdot E) = \frac{1}{2}.
\]
Therefore we have $(\Gamma \cdot \Delta) \ge (\breve{A} \cdot \Delta)$, and by Lemma~\ref{lem:exclCint}, $\Gamma$ is not a maximal center.
\end{proof}

\begin{Lem} \label{lem:110exclC2}
Let $\Gamma \subset \breve{X}$ be an irreducible and reduced complete intersection curve of type $(1, 2, 3, 5)$ which passes through $\breve{\msr}$ but does not pass through any other singular point of $\breve{X}$.
Then $\Gamma$ is not a maximal center.
\end{Lem}

\begin{proof}
By Lemma \ref{lem:110breveC}, we have
\[
\Gamma = (u = z - \lambda y^2 = w - \mu z y = v - \nu t y^3 - \xi y^5 = 0)
\]
for some $\lambda \ne 0, \mu, \nu, \xi \in \mbC$.
Let $\mcM$ be the linear system on $\breve{X}$ generated by the sections $u^3$ and $w - \mu z y$.
The base locus of $\mcM$ is $\Gamma \cup \Delta$, where $\Delta := (u = y = z = w = 0)$ is an irreducible and reduced curve of degree $1/10$.
Let $S \in \mcM$ be a general member which is cut out by the equation $w - \mu z y - \eta u^3 = 0$ on $\breve{X}$ for a general $\eta \in \mbC$, and set $S := (u = 0)_X$.
Then we have $T|_S = \Gamma + 16 \Delta$.

We claim that $S$ is smooth outside $\{\breve{q}, \breve{\msr}\}$.
It is clear that $S$ is smooth outside $\Bs \mcM = \Gamma \cup \Delta$.
Since the scheme-theoretic intersection $S \cap T$ is reduced along $\Gamma$ and $\Gamma$ is a quasi-smooth curve, $S$ is smooth along $\Gamma \setminus (\Gamma \cap \Delta) = \Gamma \setminus \{\breve{\msq}\}$.
Let $J_{C_S}$ be the Jacobi matrix of the affine cone $C_S \subset \mbA^6$.
Then, for the restriction of $J_{C_S}$ to the affine cone $(u = y = z = w = 0)$ of $\Delta$, we have
\[
J_{C_S}|_{C_{\Delta}} =
\begin{pmatrix}
v & 0 & 0 & 0 & 0 & 0 \\
t^3 & 0 & v & 0 & 0 & 0 \\
0 & 0 & 0 & 0 & 1 & 0
\end{pmatrix}.
\]  
This shows that $S$ is quasi-smooth, and hence smooth, along $\Delta \setminus \{\breve{q}\}$.
This proves the claim.

Let $\breve{\Upsilon} \colon \breve{\mbU} \to \breve{\mbP}$ be the weighted blow-up of $\breve{\mbP}$ at $\breve{\msq}$ with $\wt (u, y, z, w, v) = (7, 1, 4, 5, 3)$ with exceptional divisor $\breve{\mbE}$.
Note that $\breve{\Upsilon}|_{\breve{\mcZ}} \colon \breve{\mcZ} \to \breve{X}$ is the divisorial contraction $\breve{\upsilon}$, where $\breve{\mcZ}$ is the proper transform of $\breve{X}$ on $\breve{\mbU}$.
Let $\tilde{S}$, $\tilde{T}$, $\tilde{\Gamma}$, and $\tilde{\Delta}$ be the proper transforms of $S$, $T$, $\Gamma$, and $\Delta$ on $\breve{\mbU}$.
We set $E := \breve{\mcE}|_{\tilde{S}}$, where we recall that $\breve{\mcE} = \breve{\mbE}|_{\breve{\mcZ}}$ is the $\breve{\upsilon}$-exceptional divisor.
We have $\breve{\upsilon}^*S = \tilde{S} + \frac{5}{2} E$, and hence $K_{\tilde{S}} = (\breve{\upsilon}|_{\tilde{S}})^* K_S - 2 E$.

We will show that $(E \cdot \tilde{\Gamma}) = 1/2$ and $(E \cdot \tilde{\Delta}) = 1/3$.
For a quasi-homogeneous polynomial $e = e (u, y, z, t, w, v)$, we set $\mbD_e := (e = 0) \subset \breve{\mbP}$ which is a Weil divisor on $\breve{\mbP}$ and denote by $\tilde{\mbD}_e$ its proper transform on $\mbU$.
We have 
\[
\begin{split}
\Gamma &= \mbD_u \cdot \mbD_{z - \lambda y^2} \cdot \mbD_{w - \mu z y} \cdot \mbD_{v - \nu t y^3 - \xi y^5}, \\
\Delta &= \mbD_u \cdot \mbD_y \cdot \mbD_z \cdot \mbD_w.
\end{split}
\]
We have a natural isomorphism $\breve{\mbE} \cong \mbP (7_u, 1_y, 4_z, 5_w, 3_v)$ and under this identification we have
\[
\begin{split}
& \tilde{\mbD}_u \cap \tilde{\mbD}_{z - \lambda y^2} \cap \tilde{\mbD}_{w - \mu z y} \cap \tilde{\mbD}_{v - \nu t y^3 - \xi y^5} \cap \breve{\mbE} = \{\msp_z\}, \\
& \tilde{\mbD}_u \cap \tilde{\mbD}_y \cap \tilde{\mbD}_z \cap \tilde{\mbD}_w \cap \breve{\mbE} = \{\msp_v\},
\end{split}
\]
where $\msp_z = (0\!:\!0\!:\!1\!:\!0\!:\!0) \in \breve{\mbE}$ and $\msp_v = (0\!:\!0\!:\!0\!:\!0\!:\!1) \in \breve{\mbE}$.
It follows that 
\[
\begin{split}
\tilde{\Gamma} &= \tilde{\mbD}_u \cdot \tilde{\mbD}_{z - \lambda y^2} \cdot \tilde{\mbD}_{w - \mu z y} \cdot \tilde{\mbD}_{v - \nu t y^3 - \xi y^5}, \\
\tilde{\Delta} &= \tilde{\mbD}_u \cdot \tilde{\mbD}_y \cdot \tilde{\mbD}_z \cdot \mbD_w.
\end{split}
\]
Let $\breve{\mbA}$ be the Weil divisor class on $\breve{\mbP}$ which is the ample generator of $\Cl (\breve{\mbP}) \cong \mbZ$ so that $\breve{\mbA}|_{\breve{X}} = \breve{A}$.
We have
\[
\begin{split}
(\tilde{\Gamma} \cdot E) &= (\tilde{\Gamma} \cdot \breve{\mbE}) \\
&= (\tilde{\mbD}_u \cdot \tilde{\mbD}_{z - \lambda y^2} \cdot \tilde{\mbD}_{w - \mu z y} \cdot \tilde{\mbD}_{v - \nu t y^3 - \xi y^5} \cdot \breve{\mbE}) \\
&= (\breve{\Upsilon}^*\breve{\mbA} - \frac{7}{2} \breve{\mbE} \cdot 2 \breve{\Upsilon}^*\breve{\mbA} - \breve{\mbE} \cdot 3 \breve{\Upsilon}^*\breve{\mbA} - \frac{5}{2} \breve{\mbE} \cdot 5 \breve{\Upsilon}^*\breve{\mbA} - \frac{3}{2} \breve{\mbE} \cdot \breve{\mbE}) \\
&= \frac{7 \cdot 5 \cdot 3}{2^3} (\breve{\mbE}^5) = \frac{1}{2},
\end{split}
\]
and similarly
\[
\begin{split}
(\tilde{\Delta} \cdot E) &= (\breve{\Upsilon}^*\breve{\mbA} - \frac{7}{2} \breve{\mbE} \cdot \breve{\Upsilon}^*\breve{\mbA} - \frac{1}{2} \breve{\mbE} \cdot 2 \breve{\Upsilon}^*\breve{\mbA} - 2 \breve{\mbE} \cdot 3 \breve{\Upsilon}^*\breve{\mbA} - \frac{5}{2} \breve{\mbE}) \\
&= \frac{7 \cdot 2 \cdot 5}{2^3} (\breve{\mbE}^5) = \frac{1}{3}.
\end{split}
\]

We write $(\breve{\upsilon}|_{\tilde{S}})^*\Gamma = \tilde{\Gamma} + m E$.
Then by the computation
\[
0 = (E \cdot (\breve{\upsilon}|_{\tilde{S}})^*\Gamma) = (E \cdot \tilde{\Gamma}) + m (E^2) = \frac{1}{2} - \frac{1}{3} m,
\]
we have $m = 3/2$, that is, $(\breve{\upsilon}|_{\tilde{S}})^*\Gamma = \tilde{\Gamma} + \frac{3}{2} E$.
Similarly we have $(\breve{\upsilon}|_{\tilde{S}})^*\Delta = \tilde{\Delta} + E$.
Note that $(\tilde{\Gamma} \cdot \tilde{\Delta}) = 0$ since they are disjoint.
We have 
\[
(\Gamma \cdot \Delta) = (\tilde{\Gamma} \cdot (\breve{\upsilon}|_{\tilde{S}})^*\Delta) = (\tilde{\Gamma} \cdot \tilde{\Delta}) + (\tilde{\Gamma} \cdot E) = \frac{1}{2} > (-K_X \cdot \Delta).
\] 
By Lemma~\ref{lem:exclCint}, $\Gamma$ is not a maximal center.
\end{proof}

\begin{Lem} \label{lem:110exclC}
No curve on $\breve{X}$ is a maximal center.
\end{Lem}

\begin{proof}
Let $\Gamma$ be an irreducible and reduced curve on $\breve{X}$.
We exclude $\Gamma$ as a maximal center.
We may assume that $\Gamma$ does not pass through the $\frac{1}{5} (1, 2, 3)$ singular point $\breve{\msq}$ because otherwise there is no divisorial contraction centered along $\Gamma$.
If $\deg \Gamma \ge (-K_{\breve{X}}^3) = 7/10$, then $\Gamma$ is not a maximal center by Lemma~\ref{lem:mtdexclC}.
Hence we may assume $\deg \Gamma < 1$.
In this case $\Gamma$ passes through the $cD/2$ point $\breve{\msr}$ and $\deg \Gamma = 1/2$.
By Lemmas \ref{lem:110breveC}, \ref{lem:110exclC1}, and \ref{lem:110exclC2}, $\Gamma$ is not a maximal center.
\end{proof}

\subsubsection{Exclusion of smooth points}

We set $\Delta := (u = y = z = w = 0) \subset \breve{X}$ which is an irreducible and reduced curve.

\begin{Lem} \label{lem:110exclsmpt}
Let $\msp$ be a smooth point of $\breve{X}$ which is not contained in the curve $\Delta$.
Then $\msp$ is not a maximal center.
\end{Lem}

\begin{proof}
Let $\msp = (\alpha\!:\!\beta\!:\!\gamma\!:\!\delta\!:\!\varepsilon\!:\!\eta)$ be a smooth point of $\breve{X}$ which is not contained in $\Delta$.

We claim that $5 \breve{A}$ isolates $\msp$.
If $\alpha \ne 0$, then the set 
\[
\{\alpha y - \beta u, \alpha^2 z - \gamma u, \alpha^2 t - \delta u^2, \alpha^3 w - \varepsilon u^3, \alpha^5 v - \eta u^5\}
\]
isolates $\msp$, and hence $5 \breve{A}$ is a $\msp$-isolating class.
The same conclusion holds if $\beta \ne 0$.
Suppose that $\alpha = \beta = 0$.
Then since $(u = y = 0)_X = \Gamma \cup \Delta$ set-theoretically, where $\Gamma = (u = y = v = w^2 + z^3 = 0)$, and since $\msp \notin \Delta$, we have $\msp \in \Gamma \setminus \{\breve{r}\}$.
In particular we can write $\msp = (0\!:\!0\!:\!-1\!:\!\delta\!:\!0\!:\!1)$ for some $\delta \in \mbC$.
In this case the set
\[
\{u, y, t + \delta z, v\}
\]
isolates $\msp$, and hence $5 \breve{A}$ is a $\msp$-isolating set.
This proves the claim.
By Lemma \ref{lem:mtdexclsmpt}, $\msp$ is not a maximal center since $5 < 4/(\breve{A}^3) = 40/7$.
\end{proof}

\begin{Prop} \label{prop:110ELbreveD}
We have 
\[
\EL (\breve{X}) \setminus \bigcup_{\msp \in \breve{X}_{\operatorname{reg}} \cap \Delta} \EL_{\msp} (\breve{X}) = \{\breve{\sigma}^{-1}, \rho\}.
\]
\end{Prop}

\begin{proof}
This follows from Lemma~\ref{lem:No110divcont}, \ref{lem:110exclC}, and \ref{lem:110exclsmpt}. 
\end{proof}

\begin{Lem} \label{lem:110SgenD}
Let $\msp$ be a smooth point of $\breve{X}$ contained in the curve $\Delta := (u = y = z = w = 0) \subset \breve{X}$, and let $\mcM \sim_{\mbQ} n \breve{A}$ be a movable linear system on $\breve{X}$ such that $(X, \frac{1}{n} \mcM)$ is canonical at the generic point of $\Delta$.
Then $\msp$ is not a maximal center with respect to $\mcM$.
\end{Lem}

\begin{proof}
Suppose that a smooth point $\msp$ of $X$ contained in $\Delta$ is a maximal center with respect to a movable linear system $\mcM \sim_{\mbQ} n \breve{A}$ such that the pair $(X, \frac{1}{n} \mcM)$ is canonical at the generic point of $\Delta$.

Let $S \in |\breve{A}|$ be a general member and $T := (u = 0)_X$.
As in the proof of Lemma~\ref{lem:110exclC1}, $T|_S = \Gamma + 2 \Delta$, where $\Gamma = (u = y = v = w^2 + z^3 = 0)$, $S$ is smooth outside $\{\breve{\msq}, \breve{\msr}\}$, and $(\Gamma \cdot \Delta) = \frac{1}{2}$.
By taking intersection number of $T|_S = \Gamma + 2 \Delta$ and $\Delta$, we obtain $(\Delta^2) = - 1/5$. 

By the inversion of adjunction, the pair $(S, \frac{1}{n} \mcM|_S)$ is not log canonical at $\msp$ and it is log canonical at the generic point of $\Delta$.
Let $M \in \mcM$ be a general member and write 
\[
M_S := \frac{1}{n} M|_S = \gamma \Gamma + \delta \Delta + C,
\] 
where $C$ is an effective divisor on $S$ which contains neither $\Gamma$ nor $\Delta$ in its support.
Note that $\delta \le 1$ since $(S, M_H)$ is log canonical at the generic point of $\Delta$.  
It follows that the pair $(S, \gamma \Gamma + \Delta + C)$ is not log canonical at $\msp$.
By the inversion of adjunction, the pair $(\Delta, (\gamma \Gamma + C)|_{\Delta})$ is not log canonical at $\msp$.
Hence $\mult_{\msp} (\gamma \Gamma + C)|_{\Delta}) > 1$.
On the other hand, we have
\[
(\gamma \Gamma + C \cdot \Delta) = (M_S - \delta \Delta \cdot \Delta) = \frac{1}{10} + \frac{1}{5} \delta \le \frac{3}{10}.
\]
This is a contradiction and $\msp$ cannot be a maximal center with respect to $\mcM$.
\end{proof}

Theorem~\ref{thm:main2}, and hence Theorem~\ref{thm:main}, for Family \textnumero~110 follow from Propositions~\ref{prop:Sgen}, \ref{prop:110EL}, \ref{prop:110ELhat}, \ref{prop:110ELbreveD}, and Lemma~\ref{lem:110SgenD}.


\end{document}